\documentclass[11pt,letterpaper]{amsart}
\usepackage{mathtools}
\usepackage{amssymb}
\usepackage{mathrsfs}
\usepackage[pagebackref,colorlinks=true,linkcolor=blue,citecolor=blue]{hyperref}
\usepackage{tikz-cd}

 \UseRawInputEncoding

\tikzset{
  symbol/.style={
    draw=none,
    every to/.append style={
      edge node={node [sloped, allow upside down, auto=false]{$#1$}}}
  }
}

\newcommand{\Z}{\mathbb{Z}}

\newcommand{\R}{\mathbb{R}}
\newcommand{\BC}{\mathbb{C}}

\newcommand{\OO}{{\mathcal O}}

\newcommand{\SL}{\mathrm{SL}}
\newcommand{\GL}{\mathrm{GL}}
\newcommand{\SO}{\mathrm{SO}}
\newcommand{\Sp}{\mathrm{Sp}}

\newcommand{\Fr}{\mathrm{Fr}}
\newcommand{\RG}{\mathrm{G}}

\newcommand{\Rep}{\underline{\mathrm{Rep}}}

\newcommand{\half}[1]{\frac{#1}{2}}

\newcommand{\comment}[1]{}
\newcommand{\EE}{\mathcal{E}}

\newtheorem{thm}{Theorem}[section]
\newtheorem{cor}[thm]{Corollary}
\newtheorem{lemma}[thm]{Lemma}
\newtheorem{prop}[thm]{Proposition}
\newtheorem {conj}[thm]{Conjecture}
\newtheorem {assu}[thm]{Assumption}
\newtheorem {ques/conj}[thm]{Question/Conjecture}
\newtheorem {ques}[thm]{Question}
\newtheorem{defn}[thm]{Definition}
\newtheorem{remark}[thm]{Remark}

\newtheorem{exmp}[thm]{Example}

\newtheorem{assumption}[thm]{Working Hypotheses}

\newtheorem*{globalcond*}{Global Condition}
\newtheorem*{localcond*}{Local Condition}

\newtheorem*{globalconj*}{Global Conjecture}
\newtheorem*{localconj*}{Local Conjecture}

\newtheorem*{nonzero*}{Conjecture on the non-vanishing of the normalized intertwining operators}
\newtheorem*{holo*}{Conjecture on the holomorphicity of the normalized intertwining operators}

\DeclareMathOperator{\Ad}{Ad}

\DeclareMathOperator{\Aut}{Aut}
\DeclareMathOperator{\Hom}{Hom}

\DeclareMathOperator{\End}{End}
\DeclareMathOperator{\rank}{rank}

\numberwithin{equation}{section}

\let\oldbullet\bullet
\renewcommand{\bullet}{{\vcenter{\hbox{\tiny$\oldbullet$}}}}

\begin{document}

\title[Closure ordering conjecture on local Arthur packets]{The closure ordering conjecture on local Arthur packets of classical groups}

\author{Alexander Hazeltine}
\address{Department of Mathematics\\
University of Michigan\\
Ann Arbor, MI, 48109, USA}
\email{ahazelti@umich.edu}

\author{Baiying Liu}
\address{Department of Mathematics\\
Purdue University\\
West Lafayette, IN, 47907, USA}
\email{liu2053@purdue.edu}

\author{Chi-Heng Lo}
\address{Department of Mathematics\\
Purdue University\\
West Lafayette, IN, 47907, USA}
\email{lo93@purdue.edu}

\author{Qing Zhang}
\address{School of Mathematics and Statistics, Huazhong University of Science and Technology, Wuhan, 430074, China}
\email{qingzh@hust.edu.cn}

%    General info

\date{\today}


\thanks{The research of the second-named author is partially supported by the NSF Grants DMS-1702218, DMS-1848058, and by start-up funds from the Department of Mathematics at Purdue University. The fourth-named author is partially supported by the NSFC grant 12371010.}

\begin{abstract}
In this paper, we prove the closure ordering conjecture on the local $L$-parameters of representations in local Arthur packets of $\mathrm{G}_n=\mathrm{Sp}_{2n}, \mathrm{SO}_{2n+1}$ over a non-Archimedean local field of characteristic zero. Precisely, given any representation $\pi$ in a local Arthur packet $\Pi_{\psi}$, the closure of the local $L$-parameter of $\pi$ in the Vogan variety must contain the local $L$-parameter corresponding to $\psi$. 
    This conjecture reveals a geometric nature of local Arthur packets and is inspired by the work of Adams, Barbasch, and Vogan, and the work of Cunningham, Fiori, Moussaoui, Mracek, and Xu, on ABV-packets. 
    As an application, for general quasi-split connected reductive groups, 
    we show that the closure ordering conjecture implies the enhanced Shahidi conjecture, under certain reasonable assumptions. 
    This provides a framework towards the enhanced Shahidi conjecture in general.  
    We verify these assumptions for $\mathrm{G}_n$, hence give a new proof of the enhanced Shahidi conjecture. At last, we show that local Arthur packets cannot be fully contained in other ones, which is in contrast to the situation over Archimedean local fields and has its own interests.
\end{abstract}

\maketitle

\section{Introduction}
Let $F$ be a non-Archimedean local field of characteristic zero.  Let $\RG$ be a connected reductive group defined over $F$, $G=\mathrm{G}(F)$, $\widehat{G}(\BC)$ be the complex dual group, and  ${}^L{\RG}$ be the Langlands dual group of $G$. Let $\Pi(G)$ be the set of equivalence classes of irreducible admissible representations of $G$, and $\Phi(G)$ be the set of local $L$-parameters of $G$. The Local Langlands Conjecture asserts that there is a map $\phi \mapsto \Pi_\phi$, where $\phi\in \Phi(G)$, called the local Langlands correspondence, such that $\Pi_\phi$ is a finite subset of $\Pi(G)$ and the set $\{\Pi_\phi \, | \, \phi\in\Phi(G)\}$ forms a partition of $\Pi(G)$. The set $\Pi_\phi$ is called the local $L$-packet attached to $\phi.$ For any  $\pi\in\Pi(G)$, let $\phi_\pi$ be the local $L$-parameter for $\pi.$ 

To characterize the local components of square-integrable automorphic representations, for quasi-split classical groups, Arthur introduced the theory of local Arthur packets which are enlarged from local $L$-packets of Arthur type, and proved the Local Langlands Conjecture for symplectic groups and special orthogonal groups. 
Local Arthur packets are parameterized by local Arthur parameters. 
For simplicity, let $\RG_n=\Sp_{2n}, \SO_{2n+1}$ be a  symplectic or split odd special orthogonal group and let $G_n=\RG_n(F)$. Local Arthur parameters of $G_n$ are defined as $\widehat{\mathrm{G}}_n(\BC)$-conjugacy classes of
direct sum of irreducible representations
$$\psi: W_F \times \SL_2(\mathbb{C}) \times \SL_2(\mathbb{C}) \rightarrow \widehat{\mathrm{G}}_n(\BC),$$
\begin{equation}\label{lap}
  \psi = \bigoplus_{i=1}^r \phi_i|\cdot|^{x_i} \otimes S_{a_i} \otimes S_{b_i},  
\end{equation}
satisfying the following conditions: 
\begin{enumerate}
    \item [(1)]$\phi_i$ is a $d_i$-dimensional irreducible representation of $W_F$ whose image is bounded and consists of semi-simple elements;
    \item [(2)] $x_i \in \R$ and $|x_i|<\half{1}$;
    \item [(3)]the restrictions of $\psi$ to the two copies of $\SL_2(\mathbb{C})$ are analytic, $S_k$ is the $k$-dimensional irreducible representation of $\SL_2(\mathbb{C})$, and 
    $$\sum_{i=1}^r d_ia_ib_i = N:= 
\begin{cases}
2n+1 & \text{ when } G_n=\Sp_{2n}(F),\\
2n & \text{ when } G_n=\SO_{2n+1}(F).
\end{cases}
$$ 
\end{enumerate}
The first copy of $\SL_2(\mathbb{C})$ is called the Deligne-$\SL_2(\mathbb{C})$, denoted by $\SL_2^D(\mathbb{C})$. The second copy of $\SL_2(\mathbb{C})$ is called the Arthur-$\SL_2(\mathbb{C})$, denoted by $\SL_2^A(\mathbb{C})$. 
A local Arthur parameter $\psi$ given in \eqref{lap} is called {\it generic} if $b_i=1$ for $i=1, \ldots, r$; and is called {\it tempered} if additionally $x_i=0$ for $i=1, \ldots, r$. 
A local Arthur packet $\Pi_{\psi}$ is called generic (resp. tempered) if the corresponding local Arthur parameter $\psi$ is so.  
We let $\Psi^{+}(G_n)$ be the set of local Arthur parameters of $G_n$ and $\Psi(G_n)$ be the subset of $\Psi^+(G_n)$ consisting of local Arthur parameters $\psi$ whose restriction to $W_F$ is bounded. In other words, $\psi$ is in $\Psi(G_n)$ if and only if $x_i=0$ for $i=1,\dots, r$ in the decomposition \eqref{lap}. 

Given a local Arthur parameter $\psi$ as in \eqref{lap}, the local Arthur packet $\Pi_{\psi}$ defined in \cite[Theorem 2.2.1 and formula (1.5.1)]{Art13} is a finite multi-set of irreducible representations of $G_n$, satisfying certain twisted endoscopic character identities. 
Arthur showed that 
$\Pi_{\psi}$ consists of unitary representations when $\psi \in \Psi(G_n)$ (\cite[Theorem 1.5.1]{Art13}) and conjectured that $\Pi_{\psi}$ also consists of unitary representations when $\psi \in \Psi^+(G_n)$ (\cite[Conjecture 8.3.1]{Art13}). In a series of papers (\cite{Moe06a, Moe06b, Moe09a, Moe10, Moe11a}), M{\oe}glin explicitly constructed each local Arthur packet $\Pi_{\psi}$ and showed that it is multiplicity free. 
Then, Xu (\cite{Xu21a}) gave an algorithm to determine whether the representations in M{\oe}glin's construction are nonzero. In recent work (\cite{Ato20b}), 
Atobe gave a reformulation on  M{\oe}glin's construction, using the derivatives (see \S \ref{sec derivatives}) considered in \cite{Jan14, AM23}, which provides a way to compute the $L$-data.

To each local Arthur parameter $\psi,$ one can associate a local $L$-parameter $\phi_{\psi}$  defined as follows:
\begin{align}\label{eq phi_psi}
    \phi_{\psi}(w, x) := \psi\left(w, x, \begin{pmatrix}
        |w|^{\frac{1}{2}} & 0 \\
        0 & |w|^{-\frac{1}{2}}\\
\end{pmatrix}\right).
\end{align}
Arthur showed that the map $\psi\mapsto\phi_\psi$ is injective and the local $L$-packet $\Pi_{\phi_{\psi}}$ is contained in $\Pi_{\psi}$ (see \cite[Proposition 7.4.1]{Art13}). We say that a local $L$-parameter $\phi$, hence the corresponding local $L$-packet $\Pi_{\phi}$, is of Arthur type if $\phi=\phi_\psi$ for some local Arthur parameter $\psi.$
We also say that a representation $\pi$ is of Arthur type if $\pi\in\Pi_\psi$ for some local Arthur parameter $\psi$ and let $\Pi_A(G_n)$ denote the set of representations of $G_n$ of Arthur type. 

Since local Arthur packets are enlarged from local $L$-packets of Arthur type, 
while local $L$-packets are always disjoint, local Arthur packets often have nontrivial intersections. 
This intersection problem accounts for the difficulties of many problems related to local Arthur packets.
For symplectic and split odd special orthogonal groups, Atobe (\cite{Ato23}), the first three-named authors (\cite{HLL22}), independently gave different algorithms to determine whether a given representation is of Arthur type, and to determine all local Arthur packets containing a given representation of Arthur type. 

Note that $\Pi_{\psi}$ generally strictly contains $\Pi_{\phi_{\psi}}$.
Towards understanding how local Arthur packets are patched up via subsets of 
local $L$-packets, a direct question is as follows.

\begin{ques}\label{que intro}
Fix a local Arthur parameter $ \psi$, 
besides $\phi_{\psi}$, what are the other local $L$-parameters $\phi$ such that $\Pi_{\phi} \cap \Pi_{\psi} \neq \emptyset$?
\end{ques}

Given any local Arthur packet $\Pi_{\psi}$, one can use Atobe's reformulation to compute all local $L$-parameters $\phi$ such that $\Pi_{\phi} \cap \Pi_{\psi} \neq \emptyset$. 
Or, given any local $L$-packet $\Pi_{\phi}$, one can apply the algorithms in \cite{Ato23} or \cite{HLL22} to determine whether each representation in $\Pi_{\phi}$ is of Arthur type or not, and in the aﬀirmative case, compute all the local Arthur packets containing it. 
However, for many problems related to local Arthur packets, it is desirable to develop theoretical characterizations of local $L$-parameters appears in local Arthur packets. 

Towards Question \ref{que intro}, there is a closure ordering conjecture inspired by the work of local ABV packets (\cite{ABV92, Vog93, CFMMX22}).

\begin{conj}[Closure Ordering Conjecture, {\cite[Conjecture 2.1]{Xu21b}}]\label{main conj}
Let $\RG$ be a connected reductive group over $F$. Assume that there is a local Arthur packets theory for $G$ as conjectured in \cite[Conjecture 6.1]{Art89}. Let $\psi$ be a local Arthur parameter of $G$. Then for any 
$\pi \in \Pi_{\psi}$, we have
\begin{align} \label{eq main conj}
\phi_{\pi} \geq_C \phi_{\psi},   
\end{align}
Where $\geq_C$ is the closure ordering defined on the Vogan variety (see Definition \ref{def C ordering}).
\end{conj}

Taking Conjecture \ref{main conj} into account, towards Question \ref{que intro}, what remains is that for those $\phi \geq_C \phi_{\psi}$ in the Vogan variety, to determine which one satisfies the property that $\Pi_{\phi} \cap \Pi_{\psi} \neq \emptyset$. This is expected to be a hard question. 

We remark that the theory of local ABV packets over non-Archimedean local fields has recently been developed and studied by Cunningham, Fiori, Moussaoui, Mracek, and Xu in \cite{CFMMX22}. In particular, they prove the ABV analogue of Conjecture \ref{main conj} (see \cite[Proposition 7.10]{CFMMX22}) using geometric methods, namely, given any local ABV packet $\Pi_{\phi}^{ABV}$, for any $\pi \in \Pi_{\phi}^{ABV}$, $\phi_{\pi} \geq_C \phi$. 
The Vogan conjecture (see \cite[Section 8.3, Conjecture 1]{CFMMX22}), which is still open currently, states that if $\phi$ is of Arthur type, say $\phi=\phi_{\psi}$, then the corresponding local ABV packet will coincide with the local Arthur packet $\Pi_{\psi}$. In this paper, we prove Conjecture \ref{main conj} for symplectic and split odd special orthogonal groups, applying the intersection theory of local Arthur packets developed in \cite{HLL22}. This provides evidence for the Vogan conjecture in these cases. 

% \begin{thm}\label{thm main intro}
% Conjecture \ref{main conj} is true for $\RG_n=\Sp_{2n}$ or split $\SO_{2n+1}$. 
% \end{thm}

% Theorem \ref{thm main intro} follows directly from the following theorem. 

\begin{thm}[Theorem \ref{thm three orderings} and \S\ref{sec proof}]\label{thm max intro}
Let $\RG_n$ be $\Sp_{2n}$ or split $\SO_{2n+1}$. If $\pi\in\Pi_A(G_n)$, then we have
\begin{enumerate}
    \item  for any $\psi \in \Psi(\pi):=\{\psi\ |\ \pi \in \Pi_{\psi}\}$,
    \[  \phi_{\psi^{max}(\pi)} \geq_C \phi_{\psi}; \]
    \item \[ \phi_{\pi} \geq_C \phi_{\psi^{max}(\pi)}. \]
\end{enumerate}
Hence, Conjecture \ref{main conj} is true for $\RG_n=\Sp_{2n}$ or split $\SO_{2n+1}$. 
    Here, $\psi^{max}(\pi)$ is ``the" local Arthur parameter of $\pi$, introduced in \cite{HLL22} (see Definition \ref{def max min}).
\end{thm}

Theorem \ref{thm max intro} Part (1) is included in Theorem \ref{thm three orderings}, which is about studying the structure of the set $\Psi(\pi)$.
We remark that there are five orderings on $\Psi(\pi)$:
\begin{enumerate}
    \item Operator ordering $\geq_O$ (see \cite[Definition 11.3]{HLL22} and Definition \ref{def operator ordering});
    \item Deligne partition ordering $\geq_D$ (see Definition \ref{def Dordering});
    \item Arthur partition ordering $\geq_A$ (see \cite[Definition 11.5]{HLL22} and Definition \ref{def Jordering});
    \item Normalized intertwining operators' zeros ordering $\geq_N$ (see \cite[Definition 11.11]{HLL22});
    \item Closure ordering $\geq_C$ (see Definition \ref{def C ordering}).
\end{enumerate}
Under these five orderings, we have the following result on extremal elements in $\Psi(\pi)$.

\begin{thm}[{Theorems \ref{thm max min intro}, \ref{thm three orderings}, \cite[\S 11.2]{HLL22}}]
\label{thm five ordering}
    For any $\pi \in \Pi_{A}(G_n)$, 
under any of these five orderings $\geq_X$ on $\Psi(\pi)$, $\psi^{max}(\pi)$ and $ \psi^{min}(\pi)$ (see Definition \ref{def max min}) are the unique elements in $\Psi(\pi)$ such that for any $\psi \in \Psi(\pi)$ the following inequality holds
\begin{equation}\label{clo rel for max}
    \psi^{max}(\pi) \geq_X \psi \geq_X \psi^{min}(\pi). 
\end{equation} 

\end{thm}

The proof of Theorem \ref{thm five ordering}, which implies Part (1) of Theorem \ref{thm max intro}, relies on certain operators
introduced in \cite{HLL22} (see \S \ref{sec Operators on extended multi-segments} and Definition \ref{def operators on parameters}).

Theorem \ref{thm max intro} Part (2) is proved by induction on the rank of $G_n$ and one ingredient is an important property of $\psi^{max}(\pi)$  (Proposition \ref{prop max triangle}) which implies that the local $L$-parameter of $\pi$ and $\phi_{\psi^{max}(\pi)}$ share certain common direct summands. 

We summarize that the key ingredients in the proof of Conjecture \ref{main conj} above are the reduction to $\psi^{max}(\pi)$ and the operators introduced in \cite{HLL22} on the intersection problem of local Arthur packets.

Note that a particular feature of the closure ordering $\geq_C$ is that it can be defined for general connected reductive group $\RG$. Suppose that there is a theory of local Arthur packets for $ G=\RG(F)$ and let $\pi \in \Pi_{A}(G)$, we conjecture below that there are also unique elements $\psi^{max}(\pi)$ and $ \psi^{min}(\pi)$ satisfying the inequality \eqref{clo rel for max}. This would give a generalization of the definitions of $\psi^{max}(\pi)$ and $ \psi^{min}(\pi)$.

\begin{conj}\label{conj max intro}
Let $\RG$ be a connected reductive group defined over a non-Archimedean local field. Assume that there is a local Arthur packets theory for $G=\RG(F)$ as conjectured in \cite[Conjecture 6.1]{Art89}. Let $\pi\in\Pi_A(G)$. Then, for any $\psi_1,\psi_2\in\Psi(\pi)$, we have $\lambda_{\phi_{\psi_1}}=\lambda_{\phi_{\psi_2}}.$ Furthermore, there are unique elements $\psi^{max}(\pi)$ and $\psi^{min}(\pi)$ in $\Psi(\pi)$ such that
\begin{align*}
    \psi^{max}(\pi) \geq_C \psi \geq_C \psi^{min}(\pi),
\end{align*} 
for any local Arthur parameter $\psi \in \Psi(\pi)$. 
\end{conj}

A further interesting question here is to study the structure of non-extremal elements in $\Psi(\pi)$, for example, under the closure ordering $\geq_C$.

The famous Shahidi conjecture states that 
for any quasi-split reductive group $\RG$, tempered $L$-packets have generic members (\cite[Conjecture 9.4]{Sha90}). This conjecture has been generalized by Jiang (Jiang conjecture, \cite{Jia14}, see also \cite[Conjecture 1.6]{LS22}) to consider the upper bound of wavefront sets of representations in local Arthur packets. In this paper, as an application, we show that 
under certain assumptions, Conjecture \ref{main conj} implies one direction of the enhanced Shahidi conjecture as follows. 

\begin{conj}[{\cite[Conjecture 1.5]{LS22}}, Enhanced Shahidi Conjecture]\label{Enhanced Shahidi conjecture intro}
For any quasi-split reductive group $\RG$, suppose that there is a theory of local Arthur packets for $G=\RG(F)$ as conjectured in \cite[Conjecture 6.1]{Art89}. Then, a local Arthur packet is generic if and only if it has a generic member.
\end{conj}

In \cite{LS22}, the second-named author and Shahidi proved Conjecture \ref{Enhanced Shahidi conjecture intro} for quasi-split classical groups, with certain assumption.  
In \cite{HLL22}, the first three-named authors proved Conjecture \ref{Enhanced Shahidi conjecture intro} for symplectic and split odd special orthogonal groups, without any assumption. 

For any quasi-split connected reductive group $\RG$, suppose that there is a theory of local Arthur packets for $G=\RG(F)$ as conjectured in \cite[Conjecture 6.1]{Art89}, we show that Conjecture \ref{main conj}, along with Working Hypotheses \ref{assumption}, implies that a local Arthur packet of $G$ is tempered if it has a generic member, which is an essential part of the enhanced Shahidi conjecture (see Theorem \ref{proof of enhanced shahidi}). Then, as an application of Theorem \ref{thm max intro}, we verify these assumptions for symplectic and split odd special orthogonal groups (see Theorem \ref{dual packets} and Lemma \ref{dual generic}), hence giving a new proof of the enhanced Shahidi conjecture \ref{Enhanced Shahidi conjecture intro} in these cases (see Theorem \ref{new proof of enhanced Shahidi conjecture}). The merit of this new proof is that it provides a framework of proving the enhanced Shahidi conjecture for general quasi-split connected reductive groups.

Conjecture \ref{main conj} reveals a very interesting geometric nature of local Arthur packets and the closure ordering result in Theorem \ref{thm max intro} has also been applied to many other problems related to local Arthur packets. We name a few as follows.

\begin{enumerate}
    \item In \cite{CHLLRZ24}, jointly with Cunningham and Ray, we are
    working toward proving the Vogan conjecture on the equalities of local Arthur packets and local ABV packets of split odd orthogonal groups. The closure ordering result in Theorem \ref{thm max intro} plays an important role in this work. Among others, we apply it to gave a nice characterization of local Arthur packets corresponding to simple local Arthur parameters $\psi =  \phi \otimes S_{a} \otimes S_{b}$. More precisely, $\pi \in \Pi_{\psi}$ if and only if $\phi_{\pi} \geq_C \phi_{\psi}$ and $\phi_{\widehat{\pi}} \geq_C \phi_{\widehat{\psi}}$. Here $\widehat{\psi}=\phi \otimes S_{b} \otimes S_{a}$, the dual of $\psi$ (see Definition \ref{def operators on parameters}).

\item 
Combining the recent work of Ciubotaru-Mason-Brown-Okada (\cite{CMO21, CMO22, CMO23}) and Waldspurger (\cite{Wal18}), applying Theorem \ref{thm max intro}, jointly with Shahidi, the first three-named authors showed that the Jiang conjecture holds for any local Arthur parameter $\psi$ of $\Sp_{2n}(F)$ or split $\SO_{2n+1}(F)$ which is trivial on $W_F$ (see in \cite[\S 7]{HLLS23}).

\item In \cite{HLLS23}, jointly with Shahidi, the first three-named authors have been studying the upper bound of wavefront sets of irreducible admissible representations of connected reductive groups. We formulated a new conjecture on the upper bound and showed that it can be reduced to that of anti-discrete representations, namely, those whose Aubert-Zelevinsky involution are discrete series. In particular, assuming Conjecture \ref{main conj}, we showed that this new conjecture is equivalent to the Jiang conjecture.
\end{enumerate}

It is known that local Arthur packets could have nontrivial intersections in general. On the other hand, it is an interesting question that whether a local Arthur packet can be fully contained in another one.
Barbasch and Trapa, M{\oe}glin and Renard (see \cite[Introduction and Theorem 12.5]{MR17}) showed that the answer is affirmative for classical groups over complex fields. In this paper, applying certain operators introduced in \cite{HLL22} (see \S \ref{sec Operators on extended multi-segments} and Definition \ref{def operators on parameters}) and a reduction to $\psi^{max}(\pi)$, we show that this can not happen for $\RG_n$ over non-Archimedean local fields. 

\begin{thm}[Theorem \ref{noncontainment}]\label{noncontainment intro}
Given any $\psi_1,\psi_2 \in \Psi^+(G_n)$. 
If $\Pi_{\psi_1} \supseteq \Pi_{\psi_2}$, then $\psi_1=\psi_2$.
\end{thm}

As a direct implication of Theorem \ref{noncontainment intro}, given a local Arthur packet $\Pi_{\psi}$, there are no proper subsets providing stable distributions which are compatible with endoscopic transfers (for the compatibility properties, see \cite[Theorem 2.2.1(a), (2.2.3) and (2.2.4)]{Art13}). This result has its own interests. 

We remark that the methods in this paper are expected to extend to other classical groups, once the corresponding analogous theory in \cite{HLL22} is developed. 

Following is the structure of this paper. In \S \ref{sec notation}, we recall necessary notation and preliminaries. 
In \S \ref{structure and orderings}, we introduce various orderings on the set $\Psi(\pi)$ and study its structure.
In \S \ref{sec operator}, we prove Theorem 
\ref{thm five ordering} for $C$ ordering, which implies Theorem \ref{thm max intro} Part (1). In \S \ref{sec proof}, we prove Theorem \ref{thm max intro} Part (2). 
In \S \ref{sec Shahidi}, we give our new proof of Conjecture \ref{Enhanced Shahidi conjecture intro} for symplectic and split odd special orthogonal groups.  In \S \ref{sec noncontainment}, we prove Theorem \ref{noncontainment intro}.

\subsection*{Acknowledgements} 

The authors would like to thank  Dihua Jiang and Freydoon Shahidi for the constant support and encouragement. The fourth-named author would like to thank Clifton Cunningham for introducing  his Voganish project to him, which is an inspiration of the topics in this paper, and for his constant support and encouragement. 
The authors would like to thank Hiraku Atobe for helpful comments and suggestions. The authors also would like to thank Bin Xu for helpful communications on the question of non-containment of local Arthur packets.

\section{Notation and preliminaries}\label{sec notation}
Let $F$ be a non-Archimedean local field of characteristic zero and let $q=q_F$ denote the cardinality of its residue field. Let $\RG_n$ denote one of the groups $\Sp_{2n}$ or $\SO_{2n+1}$ defined and split over $F$ and let $G_n=\RG_n(F)$. Let $|\cdot|$ denote the normalized absolute value of $F$ or the Weil group $W_F$. We also consider it as a character $\GL_n(F) \to \BC^{\times}$ by composing with the determinant. In this section, we introduce necessary notation and preliminaries.

\subsection{Langlands classification}
In this subsection, we recall the Langlands classification for $\GL_n(F)$ and $G_n.$ These follow from the Langlands classification for $p$-adic reductive groups. For a more general setup, we refer to \cite{Kon03}. 

Let $n$ be a positive integer and fix a Borel subgroup of $\GL_n$ to be the subgroup of upper triangular matrices. Let $P$ be a standard parabolic subgroup of $\GL_n(F)$ with Levi subgroup $M\cong \GL_{n_1}(F)\times\cdots\times\GL_{n_r}(F).$ Let $\tau_i\in \Pi(\GL_{n_i}(F))$ for $i=1,2,\dots,r.$ We set
$$
\tau_1\times\cdots\times\tau_r := \mathrm{Ind}_{P}^{\GL_n(F)}(\tau_1\otimes\cdots\otimes\tau_r)
$$
to be the normalized parabolic induction. Let $\rho$ be an irreducible unitary supercuspidal representation of $\GL_n(F).$ For $x,y\in\mathbb{R}$ such that $x-y$ is a non-negative integer, we define a \emph{segment}, denoted $[x,y]_\rho$, by
$$
[x,y]_\rho=\{\rho|\cdot|^x, \rho|\cdot|^{x-1},\dots,\rho|\cdot|^y\}.
$$
We denote the \emph{Steinberg representation} attached to the segment $[x,y]_\rho$ by $\Delta_\rho[x,y].$ This is the unique irreducible subrepresentation of 
\[\rho|\cdot|^x\times\cdots\times\rho|\cdot|^y.\]
It is an essentially discrete series representation of $\GL_{n(x-y+1)}(F).$ When it is clear in context, we refer to both $[x,y]_\rho$ and $\Delta_\rho[x,y]$ as segments. We also set $Z_\rho[y,x]$ to be the unique irreducible quotient of $\rho|\cdot|^x\times\cdots\times\rho|\cdot|^y.$ In the case $y=x+1,$ we set $\Delta_\rho[x,x+1]=Z_\rho[x+1,x]$ to be the trivial representation of $\GL_0(F).$

Now we state the Langlands classification for $\GL_n(F).$ Let $\tau\in \Pi(\GL_n(F)).$ Then $\tau$ can be realized as a unique irreducible subrepresentation of a standard representation
$$\Delta_{\rho_1}[x_1,y_1]\times\cdots\times\Delta_{\rho_r}[x_r,y_r],$$
where $\rho_i$ is an irreducible unitary supercuspidal representation of $\GL_{n_i}(F),$ $[x_i,y_i]_{\rho_i}$ is a segment, and $x_1+y_1\leq\cdots\leq x_r+y_r.$ The multi-set of segments $\{ [x_i,y_i] \}_{i=1,\dots, r}$ with above requirement is an invariant of $\tau$ that uniquely characterizes it. In this setting, we write
$$
\tau=L(\Delta_{\rho_1}[x_1,y_1],\dots,\Delta_{\rho_r}[x_r,y_r]).
$$

Next, fix an $F$-rational Borel subgroup of $G_n$ and let $P$ be a standard parabolic subgroup of $G_n$ with Levi subgroup $M\cong\GL_{n_1}(F)\times\cdots\times\GL_{n_r}(F)\times G_{m}.$ Let $\tau_i\in\Pi(\GL_{n_i}(F))$ for $i=1,2,\dots,r$ and $\sigma\in\Pi(G_{n_0})$. We set
$$
\tau_1\times\cdots\times\tau_r\rtimes\sigma := \mathrm{Ind}_{P}^{G_n}(\tau_1\otimes\cdots\otimes\tau_r\otimes\sigma)
$$
to be the normalized parabolic induction. 

The Langlands classification for $G_n$ states that any $\pi\in\Pi(G_n)$ can be realized as a unique irreducible subrepresentation of a standard representation
\[\Delta_{\rho_1}[x_1,y_1]\times\cdots\times\Delta_{\rho_r}[x_r,y_r]\rtimes\pi_0,\]
where $\rho_i$ is an irreducible unitary supercuspidal representation of $\GL_{n_i}(F),$ $x_1+y_1\leq\cdots\leq x_r+y_r<0,$ and $\pi_0$ is an irreducible tempered representation of $G_{n_0}.$  The multi-set of segments $\{ [x_i,y_i] \}_{i=1,\dots, r}$ and the tempered representation $\pi_0$ with above requirement are invariants of $\pi$ that uniquely characterize it. We denote this by
\begin{align}\label{eq L-data of pi 1}
    \pi=L(\Delta_{\rho_1}[x_1,y_1],\dots,\Delta_{\rho_r}[x_r,y_r];\pi_0)
\end{align}
and call $(\Delta_{\rho_1}[x_1,y_1],\dots,\Delta_{\rho_r}[x_r,y_r];\pi_0)$ the Langlands data, or $L$-data, of $\pi.$ 

Finally, to specify the tempered representation $\pi_0$ in each $L$-data, we recall Arthur's classification of tempered representations.

\begin{thm}[{\cite[Theorem 1.5.1]{Art13}}]\label{thm Arthur tempered}
Any irreducible tempered representation of $G_n$ lies in $\Pi_\psi$ for some tempered local Arthur parameter $\psi.$ Moreover, if $\psi_1 $ and $\psi_2$ are two non-isomorphic tempered local Arthur parameters, then 
$$
\Pi_{\psi_1}\cap\Pi_{\psi_2}=\emptyset.
$$
Finally, if one fixes a choice of Whittaker datum for $G_n$ and $\psi$ is tempered, then there is a bijective map between the tempered local Arthur packet $\Pi_{\psi}$ and $\widehat{\mathcal{S}}_\psi$, the Pontryagin dual of the component group
\[\mathcal{S}_\psi:= \pi_0(\textrm{Cent}(\textrm{im}(\psi), \widehat{G}_n) / Z(\widehat{G}_n)^{\Gamma}). \]
\end{thm}

Hereinafter, we implicitly fix a choice of Whittaker datum for $G_n.$ When $\psi$ is tempered and of good parity, we write $\pi(\psi,\varepsilon)$ or $\pi(\phi_\psi,\varepsilon)$ for the element of $\Pi_\psi$ corresponding to $\varepsilon\in\widehat{\mathcal{S}}_\psi$ via the bijection in Theorem \ref{thm Arthur tempered}.

As a consequence, for each irreducible representation $\pi$ of $G_n$, we may rewrite \eqref{eq L-data of pi 1} as
\begin{align*}
    \pi=L(\Delta_{\rho_1}[x_1,y_1],\dots,\Delta_{\rho_r}[x_r,y_r];\pi(\phi,\varepsilon)),
\end{align*}
where $\phi= \phi_{\psi}$ is a tempered $L$-parameter (thus $\psi$ is also tempered) and $\varepsilon$ is an irreducible representation of the component group $\mathcal{S}_{\psi}$, which is always a character in the case of $G_n$.

\subsection{Good parity}
In this subsection, we recall the good parity condition for local Arthur parameters and irreducible representations, and a result of M{\oe}glin (Theorem \ref{thm reduction to gp}) that reduces the construction of general local Arthur packets to the construction of local Arthur packets of good parity.

Recall from \eqref{lap} that a local Arthur parameter $\psi$ of $G_n$ is of the form
\begin{equation}\label{A-param decomp}
    \psi = \bigoplus_{i=1}^r \phi_i|\cdot|^{x_i} \otimes S_{a_i} \otimes S_{b_i}.
\end{equation}
By the Local Langlands Correspondence for general linear groups, we often identify the $d_i$-dimensional bounded irreducible representation $\phi_i$ of $W_F$ with an irreducible unitary supercuspidal representation $\rho_i$ of $\GL_{d_i}(F)$ (\cite{Hen00, HT01, Sch13}), and write $\rho_i|\cdot|^{x_i} \otimes S_{a_i} \otimes S_{b_i}$ instead of $\phi_i|\cdot|^{x_i} \otimes S_{a_i} \otimes S_{b_i}$.

If $\phi_i$ is self-dual, then it preserves a non-degenerate bilinear form which is either orthogonal or symplectic. We say $\phi_i$ is orthogonal or symplectic correspondingly. We say the summand $\phi_i|\cdot|^{x_i} \otimes S_{a_i} \otimes S_{b_i}$ is of \emph{good parity} if it is self-dual of the same type as $\psi$. More explicitly, this means $x_i=0$, $\phi_i$ is self-dual, 
\begin{itemize}
    \item if $G_n=\Sp_{2n}(F)$ and $\phi_i$ is orthogonal (resp. symplectic), then $a_i+b_i$ is even (resp. odd);
    \item if $G_n=\SO_{2n+1}(F)$ and $\phi_i$ is orthogonal (resp. symplectic), then $a_i+b_i$ is odd, (resp. even).
\end{itemize}
We say $\psi$ is of \emph{good parity} if every irreducible summand in the decomposition \eqref{A-param decomp} is of good parity.

Recall that the local Arthur parameter $\psi$ has image in $\widehat{G}_n$. Therefore, if a summand $\phi_i|\cdot|^{x_i} \otimes S_{a_i} \otimes S_{b_i}$ is not of good parity, then there exists $j \in \{1,\dots, r\} \setminus\{i\}$ such that
\[ \phi_j|\cdot|^{x_j} \otimes S_{a_j} \otimes S_{b_j} \cong (\phi_i|\cdot|^{x_i} \otimes S_{a_i} \otimes S_{b_i} )^{\vee}. \]
As a consequence, we may decompose 
\begin{align}\label{eq decomp of A-par gp}
     \psi= \psi_1 \oplus \psi_0 \oplus \psi_{1}^{\vee},
\end{align}
where $\psi_0$ is of good parity and any irreducible summand in $\psi_1$ is not of good parity. Note that the choice of $\psi_1$ is not unique, while $\psi_0$ is uniquely determined by $\psi$.

For any choice of $\psi_1$, write
\[ \psi_1= \bigoplus_{i=1}^s \phi_i|\cdot|^{x_i} \otimes S_{a_i} \otimes S_{b_i}. \]
We associate an irreducible representation of general linear group by 
$$
\tau_{\psi_1}:=\bigtimes_{i=1}^s L\left(
\Delta_{\rho_i}\left[\frac{a_i-b_i}{2}+x_i, \frac{a_i+b_i}{2}-1+x_i\right], \dots, \Delta_{\rho_i}\left[
\frac{-a_i-b_i}{2}+1+x_i, \frac{b_i-a_i}{2}+x_i\right]
\right).
$$
The following result of M{\oe}glin reduces the construction of $\Pi_{\psi}$ to the construction of $\Pi_{\psi_0}$ of good parity.

\begin{thm}[{\cite[Proposition 5.1]{Moe11b}}]\label{thm reduction to gp}
With above notation, for any $\pi\in\Pi_{\psi_0}$, the induced representation $\tau_{\psi_1}\rtimes\pi$ is irreducible, independent of the choice of $\psi_1$, and
$$
\Pi_\psi=\{\tau_{\psi_1}\rtimes\pi \, | \, \pi\in\Pi_{\psi_0}\}.
$$
\end{thm}

Next, we define the \emph{good parity} condition for irreducible representations.

\begin{defn}\label{def good parity reps}
We say an irreducible representation
\[ \pi= L(\Delta_{\rho_1}[x_1,y_1],\dots,\Delta_{\rho_r}[x_r,y_r]; \pi(\phi,\varepsilon))\] 
of $G_n$ is of \emph{good parity} if 
\begin{enumerate}
    \item [$\oldbullet$] the $L$-parameter $\phi$ is of good parity. More precisely, the local Arthur parameter defined by $\psi(w,x,y)=\phi(w,x)$ is of good parity. Note that $\phi_\psi=\phi.$
    \item [$\oldbullet$] For $1 \leq i \leq r$, $x_i,y_i \in \half{1} \Z$ and $\rho_i \otimes S_{x_i-y_i+1}\otimes S_1 $ is of good parity.
\end{enumerate}
We define $\Pi^{gp}(G_n)$ to be the subset of $\Pi(G_n)$ which consists of representations of good parity, and define $\Pi_{A}^{gp}(G_n):= \Pi_{A}(G_n) \cap \Pi^{gp}(G_n).$ 
\end{defn}

By the explicit construction of local Arthur packets of good parity in \cite{Ato20b}, if $\psi$ is of good parity and $\pi \in \Pi_{\psi}$, then $\pi$ is also of good parity.

\subsection{Derivatives}\label{sec derivatives}
In this subsection, we recall the notation of derivatives considered in \cite{Jan14, AM23} and \cite{Ato20b}.

Let $\pi$ be a smooth representation of $G_n$ of finite length. We let $Jac_{P}(\pi)$ denote the Jacquet module of $\pi$ with respect to a parabolic subgroup $P$ of $G_n.$  We also denote the semisimplification of $Jac_{P}(\pi)$ by $[Jac_{P}(\pi)].$

\begin{defn}\ 
\begin{enumerate}
    \item [(1)]Let $P_d$ be a standard parabolic subgroup of $G_n$ with Levi subgroup isomorphic to $\GL_{d}(F)\times G_{n-d},$ $x\in\mathbb{R},$ and $\rho$ be an irreducible unitary self-dual supercuspidal representation of $\GL_d(F).$ We define the $\rho|\cdot|^x$-derivative of $\pi$, denoted $D_{\rho|\cdot|^x}(\pi),$ to be a semisimple representation satisfying
$$
[Jac_{P_d}(\pi)]=\rho|\cdot|^x\otimes D_{\rho|\cdot|^x}(\pi) + \sum_i \tau_i\otimes\pi_i,
$$
where the sum is over all irreducible representations $\tau_i$ of $\GL_d(F)$ such that $\tau_i\not\cong\rho|\cdot|^x.$
\item [(2)] We define $D_{\rho|\cdot|^x}^{(k)}(\pi)$ recursively by $D_{\rho|\cdot|^{x}}^{(1)}(\pi)=D_{\rho|\cdot|^{x}}(\pi)$ and
$$
D_{\rho|\cdot|^{x}}^{(k)}(\pi)=\frac{1}{k}D_{\rho|\cdot|^{x}}\circ D_{\rho|\cdot|^{x}}^{(k-1)}(\pi).
$$
\item [(3)] For a sequence of real numbers $\{x_1,\dots ,x_r\}$, we denote the composition of derivatives by
\[ D_{\rho|\cdot|^{x_1,\dots ,x_r}}(\pi):=D_{\rho|\cdot|^{x_r}} \circ \cdots \circ D_{\rho|\cdot|^{x_1}}(\pi).\]
\item [(4)] We say that $D_{\rho|\cdot|^{x}}^{(k)}(\pi)$ is the \emph{highest $\rho|\cdot|^{x}$-derivative} of $\pi$ if $D_{\rho|\cdot|^{x}}^{(k)}(\pi)\neq 0$, but $D_{\rho|\cdot|^{x}}^{(k+1)}(\pi)=0.$
\end{enumerate}
\end{defn}
We recall the following properties of derivatives.

\begin{thm}[{\cite[Lemma 3.1.3]{Jan14}, \cite[Proposition 6.1, Theorem 7.1]{AM23}}]\label{thm derivative-socle}
Let $\rho$ be an irreducible unitary self-dual supercuspidal representation of $\GL_d(F),$ $\pi\in \Pi(G_n),$ and $x\in\mathbb{R}\setminus\{0\}.$ Then the highest $\rho|\cdot|^{x}$-derivative of $\pi$, say $D_{\rho|\cdot|^{x}}^{(k)}(\pi),$ is irreducible. Moreover, the $L$-data of $D_{\rho|\cdot|^{x}}^{(k)}(\pi)$ can be explicitly computed from the $L$-data of $\pi$.
\end{thm}

The following lemma is useful to argue the non-vanishing of composition of non-highest derivatives.

\begin{lemma}[{\cite[Lemma 2.6]{HLL22}}]\label{lem Frobenius}
Suppose $\rho$ is a self-dual supercuspidal representation of $\GL_d(F)$, and $\pi$ is an irreducible representation of $G_{n+dt}$. There exists a representation $\sigma$ of $G_n$ and $x_1,\dots,x_t\in\mathbb{R}$ such that
\[ \pi \hookrightarrow \rho|\cdot|^{x_1} \times \cdots \times \rho|\cdot|^{x_t} \rtimes \sigma \]
if and only if
\[ D_{\rho|\cdot|^{x_1,\ldots,x_t}}( \pi) \neq 0. \]
In this case, if $\sigma$ has a unique irreducible subrepresentation $\sigma'$, then 
\[   D_{\rho|\cdot|^{x_1,\dots,x_t}}( \pi) \geq \sigma'\]
in the sense of Grothendieck group.
\end{lemma}

For $\pi\in\Pi^{gp}(G_n),$ we describe the possible changes of the segments part of the $L$-data of $\pi$ under the highest derivative $D_{\rho|\cdot|^{\alpha}}^{(k)}$ when $k=1$ in the following lemma.

\begin{lemma}\label{lem highest derivative of order 1}
Suppose 
\[ \pi= L(\Delta_{\rho_1}[x_1,-y_1],\dots ,\Delta_{\rho_f}[x_f,-y_f]; \pi(\phi,\varepsilon))\]
is a representation of good parity of $G_n$ and $\rho\otimes S_{2\alpha+1}$ is of the same type as $G_n$. Assume that $D_{\rho|\cdot|^{\alpha}}^{(1)}(\pi)$ is a highest derivative. Then the derivative $D_{\rho|\cdot|^{\alpha}}$ either leaves the multi-set 
\[\{\Delta_{\rho_1}[x_1,-y_1],\dots ,\Delta_{\rho_f}[x_f,-y_f]\}\]
of $\pi$ unchanged, or changes it in one of the following ways.
\begin{enumerate}
    \item [(i)] Replace $\Delta_{\rho_i}[x_i,-y_i]$ by $\Delta_{\rho_i}[x_i-1,-y_i]$ for some $i$ such that $\rho_i\cong \rho$ and $x_i=\alpha$.
     \item [(ii)] Replace $\Delta_{\rho_i}[x_i,-y_i]$ by $\Delta_{\rho_i}[x_i,-y_i+1]$ for some $i$ such that $\rho_i\cong \rho$ and $y_i=\alpha$.
      \item [(iii)] Insert $\Delta_{\rho}[\alpha-1,-\alpha]$.
\end{enumerate}
\end{lemma}
\begin{proof}
This follows from the explicit formula in \cite[Theorem 7.1]{AM23}.
\end{proof}

\subsection{Aubert-Zelevinsky involution}

Let $\RG$ be a connected reductive group defined over $F$. Let $\pi$ be an irreducible representation of $G=\RG(F)$. In \cite{Aub95}, Aubert showed that there exists $\varepsilon\in\{\pm 1\}$ such that
\begin{align}\label{eq A-Z dual}
\widehat{\pi}:=\varepsilon\sum_P (-1)^{\mathrm{dim}(A_P)}[\mathrm{Ind}_{P}^{G}(\textrm{Jac}_P(\pi))]
\end{align}
gives an irreducible representation. Here the sum is over all standard parabolic subgroups $P$ of $G$ and $A_P$ is the maximal split torus of the center of the Levi subgroup of $P.$ We say $\widehat{\pi}$ is the Aubert-Zelevinsky dual or Aubert-Zelevinsky involution of $\pi.$ 

For $\Sp_{2n}(F)$ and split $\SO_{2n+1}(F)$, Atobe and M{\'i}nguez showed that the Aubert-Zelevinsky involution is compatible with derivatives (\cite[Proposition 3.9]{AM23}), and they give an algorithm to compute the involution in terms of the $L$-data (\cite[Algorithm 4.1]{AM23}).

\section{\texorpdfstring{Structure of $\Psi(\pi)$}{}}\label{structure and orderings}

Let $\pi$ be an irreducible representation of $G_n$. In this section, we recall the structure of the set
\[ \Psi(\pi) := \{ \psi \in \Psi^+(G_n)\ | \ \pi \in \Pi_{\psi}  \}\]
studied in \cite{HLL22}, and give the definition of the four orderings $\geq_{O}, \geq_A, \geq_C, \geq_D$ on $\Psi(\pi)$. Especially, we recall the definition of $\psi^{max}(\pi), \psi^{min}(\pi) \in \Psi(\pi)$ (Definition \ref{def max min}) and their characterization properties via these orderings (Theorems \ref{thm max min intro}, \ref{thm three orderings}).

\subsection{\texorpdfstring{Extended multi-segments and the operator ordering $\geq_O$}{}}\label{sec Operators on extended multi-segments}

In this subsection, we recall the definition of extended multi-segments and various operators on them. Then we recall the construction of $\Psi(\pi)$ given in \cite{HLL22} and \cite{Ato23} and its structure studied in \cite{HLL22}, especially the distinguished members $\psi^{max}(\pi)$ and $\psi^{min}(\pi)$. Finally, we recall definition of the operator ordering $\geq_O$.

First, we recall the definition of extended multi-segments.

\begin{defn} [{\cite[Definition 3.1]{Ato20b}}]
(Extended multi-segments)\label{def multi-segment}
\begin{enumerate}
\item
An \emph{extended segment} is a triple $([A,B]_\rho, l, \eta)$,
where
\begin{itemize}
\item
$[A,B]_\rho = \{\rho|\cdot|^A, \rho|\cdot|^{A-1}, \dots, \rho|\cdot|^B \}$ is a segment 
for an irreducible unitary supercuspidal representation $\rho$ of some $\GL_d(F)$; 
\item
$l \in \Z$ with $0 \leq l \leq \frac{b}{2}$, where $b = \#[A,B]_\rho = A-B+1$; 
\item
$\eta \in \{\pm1\}$. 
\end{itemize}
\item Consider a multi-set of extended segments of the form $$\{([A_i,B_i]_{\rho},l_i,\eta_i)\}_{i \in I_{\rho}}.$$
We say a total order $>$ on $I_{\rho}$ is admissible (or satisfies (P)) if
\[ A_i< A_j, B_i< B_j\Longrightarrow i<j. \]
We say an admissible order $>$ satisfies (P') if
\[  B_i< B_j\Longrightarrow i<j. \]
\item
An \emph{extended multi-segment} for $G_n$ is 
an equivalence class (via the equivalence defined below) of multi-sets of extended segments 
\[
\EE = \cup_{\rho}\{ ([A_i,B_i]_{\rho}, l_i, \eta_i) \}_{i \in (I_\rho,>)}
\]
such that 
\begin{itemize}
\item
$I_\rho$ is a totally ordered finite set with a fixed total order $>$ satisfies (P);

\item
$A_i + B_i \geq 0$ for all $\rho$ and $i \in I_\rho$; 

\item
as a representation of $W_F \times \SL_2(\BC) \times \SL_2(\BC)$, 
\[
\psi_{\EE} = \bigoplus_\rho \bigoplus_{i \in I_\rho} \rho \otimes S_{a_i} \otimes S_{b_i} 
\]
where $(a_i, b_i) = (A_i+B_i+1, A_i-B_i+1)$,
is a local Arthur parameter for $G_n$ of good parity. We shall denote $\psi_{\EE}$ the local Arthur parameter associated with $\EE$. 
\item The sign condition
\begin{align*}
\prod_{\rho} \prod_{i \in I_\rho} (-1)^{[\frac{b_i}{2}]+l_i} \eta_i^{b_i} = 1
\end{align*}
holds.
\end{itemize}

\item
Two extended segments $([A,B]_\rho, l, \eta)$ and $([A',B']_{\rho'}, l', \eta')$ are \emph{weakly equivalent} 
if 
\begin{itemize}
\item
$[A,B]_\rho = [A',B']_{\rho'}$; 
\item
$l = l'$; and 
\item
$\eta = \eta'$ whenever $l = l' < \frac{b}{2}$. 
\end{itemize}
We say that two extended multi-segments 
$$\EE = \cup_{\rho}\{ ([A_i,B_i]_{\rho}, l_i, \eta_i) \}_{i \in (I_\rho,>)}$$
and 
$$\EE' = \cup_{\rho}\{ ([A'_i,B'_i]_{\rho}, l'_i, \eta'_i) \}_{i \in (I_\rho,>)}$$
are \emph{weakly equivalent}
if for any $\rho$ and $i \in I_\rho$, the extended segments $([A_i,B_i]_\rho, l_i, \eta_i)$ and $([A'_i,B'_i]_{\rho}, l'_i, \eta'_i)$ are weakly equivalent.
\item For each extended multi-segment $\EE$, we let $\pi(\EE)$ denote the representation associated with $\EE$ as in \cite[\S 3.2]{Ato20b} (or see \cite[Definition 3.4]{HLL22}). $\pi(\EE)$ is either irreducible or zero. We denote $\Rep$ the set of extended multi-segments that give nonzero representations, and $\Rep^{(P')}$ the subset of $\Rep$ consists of extended multi-segments whose total order on any $I_{\rho}$ satisfies (P').
\end{enumerate}
\end{defn}

Atobe showed that local Arthur packets of good parity can be constructed by extended multi-segments as follows.

\begin{thm}[{\cite[Theorem 3.4]{Ato20b}}]\label{thm Atobe's reformulation}
Suppose $\psi= \bigoplus_{\rho} \bigoplus_{i \in I_{\rho}} \rho \otimes S_{a_i} \otimes S_{b_i}$ is a local Arthur parameter of $G_n$ of good parity. Fix an admissible order $>$ on $I_{\rho}$ for each $\rho$ that satisfies ($P'$) if $\half{a_i-b_i}<0$ for some $i \in I_{\rho}$. Then
\[ \bigoplus_{\pi \in \Pi_{\psi}} \pi= \bigoplus_{\EE} \pi(\EE),\]
where $\EE$ runs over all extended multi-segments with $\psi_{\EE}=\psi$ and $\pi(\EE) \neq 0$, and the total orders are the ones fixed above.
\end{thm}

In this paper, we use the following operators on extended multi-segments. We omit their precise definitions, which can be found in \cite{Ato20b} and \cite{HLL22}. Instead, we describe their effect on the local Arthur parameters explicitly in Definition \ref{def operators on parameters} below.
\begin{enumerate}
    \item [$\oldbullet$] Row exchange: $R_k$ (\cite[\S4.2]{Ato20b} or \cite[Definition 3.14]{HLL22} ).
    \item [$\oldbullet$] Union-intersection: $ui_{i,j}$ (\cite[\S5.2]{Ato20b} or \cite[Definitions 3.22, 5.1]{HLL22} ).
    \item [$\oldbullet$] Aubert-Zelevinsky dual: $dual$ (\cite[Definition 6.1]{Ato20b} or \cite[Definition 3.27]{HLL22}).
    \item [$\oldbullet$] Partial dual: $dual_k^{-}, dual_k^{+}$  (\cite[Definition 6.5]{HLL22}).
\end{enumerate}

We recall the definition of the operators $sh_j^{d}, add_j^{d}$ on extended multi-segments that will be used in the computation.

\begin{defn}(shift, add)\\
Let $\EE = \cup_{\rho}\{ ([A_i,B_i]_{\rho}, l_i, \eta_i) \}_{i \in (I_\rho,>)}$ be an extended multi-segment. For $j \in I_{\rho'}$ and $d \in \Z$, we define the following operators. It is immediate that the operators commute with each other and so we denote the composition by summation. 

\begin{enumerate}
    \item [1.] $sh_j^{d}(\EE)= \cup_{\rho}\{ ([A_i',B_i']_{\rho}, l_i, \eta_i) \}_{i \in (I_\rho,>)}$ with 
    \[ [A_i',B_i']_{\rho}= \begin{cases}
    [A_i+d,B_i+d]_{\rho} & \text{ if }\rho=\rho' \text{ and } i = j,\\
     [A_i,B_i]_{\rho} & \text{ otherwise, }\end{cases} \]
    and $sh^d_{\rho'}=\sum_{j\in I_{\rho'}} sh_j^{d}$.
     \item [2.] $add_j^{d}(\EE)= \cup_{\rho}\{ ([A_i',B_i']_{\rho}, l_i', \eta_i) \}_{i \in (I_\rho,>)}$ with 
    \[ ([A_i',B_i']_{\rho},l_i')= \begin{cases}
    ([A_i+d,B_i-d]_{\rho},l_i+d) & \text{ if }\rho=\rho' \text{ and } i = j,\\
     ([A_i,B_i]_{\rho},l_i) & \text{ otherwise, }\end{cases} 
    \]
    and $add^d_{\rho'}=\sum_{j\in I_{\rho'}} add_j^{d}$. We remove the extended segments in $add_j^{d}(\EE)$ of the form $([A,A-1]_{\rho'},0,\ast)$.
\end{enumerate}
We only use these notations in the case that the resulting object is still an extended multi-segment.
\end{defn}

Here are some identities between these operators.
\begin{lemma}[{\cite[Lemma 3.30(ii), and Corollary 5.5]{HLL22}}]\label{lem equalities of operators}
Let 
$$\EE=\cup_{\rho} \{([A_i,B_i]_{\rho},l_i,\eta_i)\}_{i \in (I_\rho, >)}$$
be an extended multi-segment.
\begin{enumerate}
    \item For any $k \in I_{\rho}$ and $t \in \Z$, we have 
        \[dual \circ sh_{k}^t(\EE)= add_{k}^t \circ dual(\EE)\]
        if any side of the equation is still an extended multi-segment.
        \item Suppose $\EE \in \Rep^{(P')}$ and $ ui_{i,j}$ is applicable on $\EE$ not of type 3'. Then we have 
        \[ dual \circ ui_{j,i} \circ dual \circ ui_{i,j}(\EE)=\EE. \]
\end{enumerate}
\end{lemma}

One of the main results of \cite{HLL22} is the following theorem.

\begin{thm}[{\cite[\S 6]{HLL22}, \cite{Ato23}}]\label{thm operators extended multi-segment}
    Suppose $\pi\in \Pi_A^{gp}(G_n)$, and take $\EE \in \Rep^{(P')}$ such that $\pi=\pi(\EE)$. Then $\pi(\EE) \cong \pi(\EE')$ if and only if $\EE'$ can be obtained from $\EE$ by applying a sequence of operators involving $ R_k, ui_{i,j}, dual \circ ui_{j,i} \circ dual$ and $dual_k^{-}$ and their inverses.
\end{thm}

Remark that the statement in \cite{Ato23} is slightly different since he used different set of operators, but the results are logically equivalent. Note that when $\pi$ is of good parity, we have
\[ \Psi(\pi)= \{ \psi_{\EE} \ | \ \pi(\EE)= \pi \}.\]
Therefore, together with Theorem \ref{thm reduction to gp}, the above theorem gives a way to compute $\Psi(\pi)$. A more precise formula or algorithm can be found in \cite[Theorem 7.4]{HLL22}.

We shall often omit the operator $R_k$ since $\psi_{\EE}= \psi_{R_k(\EE)}$. In \cite[\S 10]{HLL22}, the first three-named authors observed that among the other six types of operators in Theorem \ref{thm operators extended multi-segment} (including the inverses)
, $ ui_{i,j}^{-1}$, $dual \circ ui_{j,i} \circ dual $ and $dual_k^{-}$ raise the ``temperedness" of local Arthur parameters under certain measurement of temperedness (see Theorem \ref{thm three orderings}(1) below). This observation leads to the following definition. 

\begin{defn}
We say an operator $T$ is a \emph{raising} operator if it is of the form $ ui_{i,j}^{-1}$, $dual \circ ui_{j,i} \circ dual,$ or $dual_k^{-}$.
\end{defn}

We say an extended multi-segment $\EE \in \Rep$ is \emph{absolutely maximal} if there are no raising operators applicable on $\EE$, and we say $\EE$ is \emph{absolutely minimal} if there are no inverse of raising operators applicable on $\EE$. Note that $\EE$ is absolutely minimal if and only if and only if $dual(\EE)$ is absolutely maximal. With the notation introduced so far, we summarize the other main results of \cite{HLL22} we need in the following theorem.

\begin{thm}[{\cite[ \S 11]{HLL22}}]\label{thm E max}
Suppose $\pi\in \Pi_A^{gp}(G_n)$ and take $\EE \in \Rep^{(P')}$ such that $\pi=\pi(\EE)$.
\begin{enumerate}
    \item There exists a unique absolutely maximal (minimal) member in the set \[\Psi(\EE):=\{\EE'\ | \ \pi(\EE')=\pi(\EE)\}/(\text{row exchanges}),\]
    which we denote by $\EE^{|max|}$ (resp. $\EE^{|min|}$).
    \item After row exchanges, there exists a sequence of raising operators $\{T_i\}_{i=1}^m$, $1 \leq r\leq m$ such that
    \begin{align*}
        \EE^{|max|}&=T_1 \circ \cdots \circ T_r(\EE),\\
        \EE^{|min|}&=T_m^{-1} \circ \cdots \circ T_{r+1}^{-1}(\EE).
    \end{align*}
\end{enumerate}
\end{thm}

From Part (1) of the above theorem, we give the following definition.

\begin{defn}\label{def max min}
Let $\pi\in \Pi_A(G_n)$. Write $\pi= \tau_{\psi_1} \rtimes \pi_0$ as in Theorem \ref{thm reduction to gp}, where $\pi_0$ is of good parity. Take any $\EE$ such that $\pi(\EE)=\pi_0$. Then we define
\begin{align*}
    \psi^{max}(\pi)&:= \psi_1  \oplus \psi_{\EE^{|max|}} \oplus  \psi_1^{\vee},\\
        \psi^{min}(\pi)&:= \psi_1  \oplus \psi_{\EE^{|min|}}+ \psi_1^{\vee}.
\end{align*}
\end{defn}

Intuitively, $\psi^{max}(\pi)$ (resp. $\psi^{min}(\pi)$) is the ``most tempered" (resp. ``least tempered") member in the set $\Psi(\pi)$. Also, Part (1) of the following proposition supports the intuition that $\psi^{max}(\pi)$ is ``closest" to the $L$-parameter $\phi_{\pi}$, and hence we may call $\psi^{max}(\pi)$ ``the" local Arthur parameter of $\pi$.

\begin{prop}[{\cite[\S 10]{HLL22}}]\label{properties of max and min}
Let $\pi \in \Pi_{A}(G_n)$.
\begin{enumerate}
    \item $\phi_\pi$ is of Arthur type if and only if $\pi \in \Pi_{\phi_{\psi^{max}(\pi)}}$.
    \item $\psi^{min}(\pi)= \widehat{\psi^{max}(\widehat{\pi})}$,
    where $\widehat{\pi}$ is the Aubert-Zelevinsky involution of $\pi$.
    \item  $\psi^{max}(\pi)= \psi^{min}(\pi)$ if and only if $\Psi(\pi)$ is a singleton.
\end{enumerate}
\end{prop}
Note that Theorem \ref{thm max intro} generalizes Part (1) of the above proposition.

Now we translate the definitions and results above in terms of local Arthur parameters. If $\psi_0$ is a local Arthur parameter of good parity, we often write
\begin{align}\label{eq local Arthur parameter}
    \psi_0 = \bigoplus_{\rho}\bigoplus_{i\in I_\rho} \rho \otimes S_{a_i} \otimes S_{b_i},
\end{align}
where the first sum runs over certain finite set of
irreducible unitary supercuspidal representations $\rho$ of $\GL_d(F)$, $d \in \mathbb{Z}_{\geq 1}$, and $I_\rho$ denotes an indexing set. When we are working with multiple local Arthur parameters, we denote the indexing set by $I_{\rho}(\psi_0)$.

\begin{defn}\label{def operators on parameters}
Suppose $\psi \in \Psi^+(G_n)$. Decompose $\psi= \psi_1 \oplus \psi_0 \oplus \psi_1^{\vee}$ as in \eqref{eq decomp of A-par gp} and write
\[ \psi_0=\bigoplus_{\rho} \bigoplus_{i \in I_{\rho}} \rho \otimes S_{a_i} \otimes S_{b_i}.  \]
Then for $i,j,k \in I_{\rho}$, we define the operators $dual$, $ui_{i,j}$ and $dual_k^{-}$ as follows. 
\begin{enumerate}
    \item Define $dual(\psi)=\widehat{\psi}$ to be the local Arthur parameter given by
    \begin{equation}\label{psi dual}
  \widehat{\psi}(w,x,y):= \psi(w,y,x).  
\end{equation}
Equivalently, $\widehat{\psi}$ can be obtained from $\psi$ by switching the $a_i$'s and $b_i$'s in the decomposition \eqref{lap}. We identify the index set $I_{\rho}(\psi_0)$ with $I_{\rho}(\widehat{\psi}_0)$ in the obvious way.
    \item For $r \in I_{\rho}$, let $A_r= \half{a_r+b_r}-1$ and $B_r= \half{a_r-b_r}$. Then we may rewrite the decomposition of $\psi_0$ as 
    \[ \psi_0= \bigoplus_{\rho} \bigoplus_{i \in I_\rho} \rho \otimes S_{A_i+B_i+1} \otimes S_{A_i-B_i+1}.\]
    The operator $ui_{i,j}$ is applicable on $\psi$ if the following conditions hold.
    \begin{enumerate}
    \item [$\oldbullet$] $A_j \geq A_i+1 \geq B_j >B_i.$
        \item  [$\oldbullet$] For any $r \in I_{\rho}$, if $B_i<B_r<B_j$, then $A_r \leq A_i $ or $A_r \geq A_j$.
    \end{enumerate}
    In this case, we define $ui_{i,j}(\psi_0)$ by replacing the summands 
    \[ \rho\otimes S_{A_i+B_i+1}\otimes S_{A_i-B_i+1} +\rho\otimes S_{A_j+B_j+1}\otimes S_{A_j-B_j+1} \]
of $\psi_0$ with
    \[\rho\otimes S_{A_j+B_i+1}\otimes S_{A_j-B_i+1} +\rho\otimes S_{A_i+B_j+1}\otimes S_{A_i-B_j+1}. \]
    If $A_i+1-B_j=0$, then we omit the last summand, and say this $ui_{i,j}$ is of type 3'. Finally, we define $ui_{i,j}(\psi)= \psi_1 \oplus ui_{i,j}(\psi_0) \oplus \psi_1^{\vee}$.
    \item The operator $dual_k^{-}$ is applicable on $\psi$ if $b_k=a_k+1$. In this case, we define $dual_k^{-}(\psi_0)$ by replacing the summand
    \[ \rho \otimes S_{a_k}\otimes S_{a_{k}+1}\]
    of $\psi_0$ with 
    \[ \rho \otimes S_{a_k+1}\otimes S_{a_{k}},\]
    and we define $dual_k^{-}(\psi)= \psi_1 \oplus dual_k^{-}(\psi_0) \oplus \psi_1^{\vee}$. 
    \item 
    Let $T$ be any of the operators above or their inverses. If $T$ is not applicable on $\psi$, then we define $T(\psi)=\psi$.
\end{enumerate}
\end{defn}

\begin{remark}\label{rmk operator}
\begin{enumerate}
    \item We say $ui_{i,j}^{-1}$ is applicable on $\psi$ if there exists $\psi'$ such that $ui_{i,j}(\psi')=\psi$, and we define $ui_{i,j}^{-1}(\psi)=\psi'$ in this case. Note that $i,j$ are indices of $I_{\rho}(\psi')$ but not of $I_{\rho}(\psi)$. However, there is an obvious way to identify $ I_{\rho}(\psi')$ with $I_{\rho}(\psi)$ or  $I_{\rho}(\psi) \sqcup \{j\}$ if $ui_{i,j}$ is of type 3'.
    \item The composition $dual \circ ui_{j,i} \circ dual=:T$ can be described as follows. Keep the notation in above definition. $T$ is applicable on $\psi$ only if
    \[A_i \geq A_j+1 \geq -B_i > -B_j.\]
    When it is applicable, it replaces the summands
    \[\rho\otimes S_{A_i+B_i+1}\otimes S_{A_i-B_i+1} +\rho\otimes S_{A_j+B_j+1}\otimes S_{A_j-B_j+1}\]
    of $\psi$ with
    \[\rho\otimes S_{A_i+B_j+1}\otimes S_{A_i-B_j+1}+\rho\otimes S_{A_j+B_i+1}\otimes S_{A_j-B_i+1}.\]
    If $ A_j+1=-B_i$, then we omit the last summand.
    \item Note that it is possible that $\Pi_{\psi} \cap \Pi_{T(\psi)}=\emptyset$. For example, let $\rho$ be a symplectic representation, $\psi_1=\rho\otimes S_1\otimes S_1 + \rho\otimes S_3\otimes S_1,$ and $\psi_2=\rho\otimes S_2\otimes S_2.$ Then $ui_{1,2}(\psi_1)=\psi_2.$ However, $\Pi_{\psi_1} \cap \Pi_{\psi_2}=\emptyset$, since $\Pi_{\psi_1}$ consists of 2 tempered representations while $\Pi_{\psi_2}$ consists of a single non-tempered representation. One way to see this is that $ui_{1,2}$ is not applicable on any extended multi-segment $\EE$ such that $\psi_\EE=\psi_1.$
\end{enumerate}
\end{remark}

With this definition, Theorem \ref{thm operators extended multi-segment} can be translated as follows.
 
\begin{thm}[{\cite[Theorem 1.4]{HLL22}}]\label{thm structure of Psi(pi) intro}
If $\pi\in\Pi_A(G_n)$ and $\psi_1, \psi_2\in \Psi(\pi),$ then there exists a sequence of operators $\{T_l\}_{l=1}^m$ such that
\[ \psi_1= T_1 \circ \cdots \circ T_m (\psi_2),\]
where each $T_l$ is one of the operators $ui_{i,j}$,  $ dual \circ ui_{j,i}\circ dual$, $dual_k^{-}$, or their inverses.
\end{thm}

For $\pi\in \Pi_{A}(G_n)$, the raising operators induce the \emph{operator ordering} $\geq_O$ on $\Psi(\pi).$

\begin{defn}\label{def operator ordering}
We define a partial order $\geq_{O}$ on $\Psi^+(G_n)$ by $\psi_1 \geq_{O} \psi_2$ if $\psi_1=\psi_2$ or there exists a sequence of raising operators $\{T_l\}_{l=1}^m$ such that
\[ \psi_1= T_1 \circ \cdots \circ T_m(\psi_2).\]
\end{defn}

We shall often restrict this ordering to $\Psi(\pi)$ where $\pi \in \Pi_{A}(G_n)$. With this definition, Theorem \ref{thm E max} can be rephrased as follows.

\begin{thm}[{\cite[\S 11]{HLL22}}]\label{thm max min intro}

Let $\RG_n$ be $\Sp_{2n}$ or split $\SO_{2n+1}$ and $\pi \in \Pi_{A}(G_n)$. The parameters $\psi^{max}(\pi)$ and $ \psi^{min}(\pi)$ are the unique elements in $\Psi(\pi)$ such that for any $\psi \in \Psi(\pi)$, the following inequality holds
\begin{equation*}
    \psi^{max}(\pi) \geq_O \psi \geq_O \psi^{min}(\pi).
\end{equation*}  
\end{thm}

In the following subsections, we define $X$ orderings $\geq_X$  on $\Psi^+(G_n)$, where $X \in \{A,C,D\}$. These orderings measure the ``temperedness" in the sense that $\psi$ is a maximal element in $\Psi(G_n)_{\lambda_{\phi_\psi}}$ (see \eqref{eq psi lambda}) with respect to these orderings if and only if $\psi$ is tempered. The following analogue of Theorem \ref{thm max min intro} holds for these orderings, which supports the intuition that $\psi^{max}(\pi)$ (resp. $\psi^{min}(\pi)$) is the ``most tempered" (resp. ``least tempered") member in the set $\Psi(\pi)$.

\begin{thm}\label{thm three orderings}
Let $X\in \{A,C,D\}$. 
\begin{enumerate}
    \item For any $\psi \in \Psi^+(G_n)$ and any raising operator $T$ applicable on $\psi$, we have 
    \[T(\psi) \gneq_X \psi.\]
    In particular, if $\psi \geq_{O} \psi'$, then $\psi \geq_X \psi'$.
    \item Let $\pi\in \Pi_{A}(G_n)$. The parameters $\psi^{max}(\pi)$ and $ \psi^{min}(\pi)$ are the unique elements in $\Psi(\pi)$ such that for any $\psi \in \Psi(\pi)$, the following inequality holds
    \begin{align}\label{clo rel for max 2}
        \psi^{max}(\pi) \geq_X \psi \geq_X \psi^{min}(\pi).
    \end{align}
\end{enumerate}
\end{thm}

Remark that Part (2) follows from Part (1) and Theorem \ref{thm max min intro}. The above theorem for $A$ ordering is proved in \cite[\S 11.1]{HLL22}. We shall prove the theorem for $C$ ordering in \S \ref{sec computation for operators} and for $D$ ordering in \S \ref{sec noncontainment}. The relations among the four orderings are summarized in \S \ref{sec relations}.

\subsection{\texorpdfstring{Vogan variety and the closure ordering $\geq_C$}{}}
In this subsection, we recall the definition of Vogan variety and give the definition of the partial ordering $\geq_C$.

We consider connected reductive groups $\RG$ defined over $F$ and $G=\RG(F)$. Recall that $L$-group of $G$ is ${}^L G=\widehat{\RG}(\BC)\rtimes W_F$.
For $\phi\in \Phi(G)$, we associate a homomorphism 
$$\lambda_{\phi} : W_F \to {}^L G$$
by
\[ \lambda_{\phi}(w):= \phi\left(w,\begin{pmatrix}
|w|^{\half{1}} & \\ & |w|^{\half{-1}}
\end{pmatrix} \right), \]
which is an \emph{infinitesimal character} of $G$, i.e., it is a continuous section of the projection ${}^L G\rightarrow W_F$ and its image consists of semi-simple elements (see \cite[\S 4.1]{CFMMX22}). Conversely, for each infinitesimal character $\lambda$, denote
\[ \Phi(G)_{\lambda}:=\{\phi\in\Phi(G) \ | \ \lambda_{\phi}=\lambda\}. \]

Following \cite{Vog93}, we consider the sets
\begin{align*}
K_{\lambda}&:= \{g \in \widehat{\RG}(\BC) \ | \ \lambda(w)g= g\lambda(w), \forall w \in I_F \}, \\
H_{\lambda}&:= \{g \in \widehat{\RG}(\BC) \ | \ \lambda(w)g= g\lambda(w), \forall w \in W_F \},
\end{align*}
where $I_F$ is the inertia subgroup of $F.$ Note that by $g$, here we mean $(g,1)\in{}^L G.$ The \emph{Vogan variety} $V_{\lambda}$ for infinitesimal character $\lambda$ is defined by
\[ V_{\lambda}:= \{ x \in \mathrm{Lie}(K_\lambda) \ | \ \Ad(\lambda(w))x= |w|x, \forall w \in W_F\}, \]
where $\mathrm{Lie}(K_\lambda)$ is the Lie algebra of $K_\lambda$ ($K_\lambda$ is a reductive group). The vector space $V_{\lambda}$ admits an action of the group $H_\lambda$
with finitely many orbits (\cite[Proposition 5.6]{CFMMX22}). For each $\phi \in \Phi(G)_{\lambda}$, the element
\[ X_{\phi}:= d (\phi|_{\SL_2(\BC)}) \left(  \begin{pmatrix}
0&1\\0&0
\end{pmatrix}\right)\]
is in $V_{\lambda}$. Let $C_{\phi}$ denote the $H_{\lambda}$-orbit of $X_{\phi}$. Then we obtain a map
\begin{align*}
     \Phi(G)_{\lambda}& \to V_{\lambda}/H_{\lambda},\\
     \phi& \mapsto C_{\phi},
\end{align*}
which is in fact a bijection (see \cite[Proposition 4.2]{CFMMX22}). The geometric structure of $V_{\lambda}$ gives $\Phi(G)_{\lambda}$ a partial order $\geq_C$ defined by the closure relation as follows.
\begin{defn}\label{def C ordering}
For each infinitesimal character $\lambda$ of $G$, the finite set $\Phi(G)_{\lambda}$ is equipped with a partial ordering $\geq_C$ defined by $\phi_{1} \geq_C \phi_2$ if $\overline{C_{\phi_1}} \supseteq C_{\phi_2}$.
\end{defn}

We assume that there is an analogous definition of $\Psi^+(G)$ and assume the map $ \psi\mapsto \phi_{\psi}$ defined by \eqref{eq phi_psi} is an injection from $\Psi^{+}(G)$ to $\Phi(G)$. For an infinitesimal character $\lambda$ of $G$, we define the set
\begin{align}\label{eq psi lambda}
     \Psi(G)_{\lambda}:=\{\psi \in \Psi^+(G)\ | \ \lambda_{\phi_{\psi}}=\lambda\},
\end{align}
which can be viewed as a subset of $\Phi(G)_{\lambda}$. M{\oe}glin proved that all representations in a local Arthur packet $\Pi_{\psi}$ share the same extended cuspidal support (see \cite[Proposition 4.1]{Moe09b}) assuming there is a theory of local Arthur packets as in \cite[Conjecture 6.1]{Art89}. As a consequence, for any $\pi \in \Pi_{\psi}$, we have $\lambda_{\phi_{\pi}}=\lambda_{\phi_{\psi}}$, and for a fixed $\pi\in \Pi_A(G_n)$, we have an inclusion $\Psi(\pi) \subseteq \Psi(G)_{\lambda_{\phi_{\pi}}}$. Then we may restrict the partial order $\geq_C$ from $\Phi(G)_{\lambda}$ to $\Psi(G)_{\lambda}$ or $\Psi(\pi)$.
\begin{defn}\label{def Cordering}
For each infinitesimal character $\lambda$ of $G$, we define a partial order $\geq_{C}$ on the set $\Psi(G)_{\lambda}$ by $\psi_1 \geq_C \psi_2$ if $\phi_{\psi_1} \geq_C \phi_{\psi_2}$.
\end{defn}

If $\phi$ is a tempered $L$-parameter, then $\phi$ must be the unique maximal element in $\Phi(G_n)_{\lambda}$ under the closure ordering (see Proposition \ref{CFMZ}). Therefore, the closure ordering gives a measurement of ``temperedness," i.e., if $\psi_1 \geq_C \psi_2$, then $\psi_1$ is ``more tempered" than $\psi_2$.

% In \S \ref{sec computation for operators}, we prove the following analogue of Theorem \ref{thm max min intro} for the closure ordering $\geq_C$, which supports the intuition that $\psi^{max}(\pi)$ (resp. $\psi^{min}(\pi)$) is the ``most tempered" (resp. ``least tempered") member in the set $\Psi(\pi)$.

% \begin{thm}\label{thm operator intro}
% Let $\RG_n$ be $\Sp_{2n}$ or split $\SO_{2n+1}$. 
% \begin{enumerate}
%     \item If $T$ is a raising operator, then
%     \[T(\psi) \geq_{C} \psi,\]
%     for any $\psi \in \Psi(G_n)$. 
%     In other words, if $\psi \geq_{O} \psi'$, then $\psi \geq_{C} \psi'$.
%     \item Let $\pi \in \Pi_{A}(G_n)$. The parameters $\psi^{max}(\pi)$ and $ \psi^{min}(\pi)$ are the unique elements in $\Psi(\pi)$ such that for any $\psi \in \Psi(\pi)$, the following inequality holds
% \begin{equation}\label{clo rel for max 2}
%     \psi^{max}(\pi) \geq_C \psi \geq_C \psi^{min}(\pi). 
% \end{equation} 
% \end{enumerate}
% \end{thm}

Since the closure ordering $\geq_C$ can be defined for any $G$, suppose there is a theory of local Arthur packets for $ G$ and let $\pi \in \Pi_{A}(G)$. We conjecture that there are also unique elements $\psi^{max}(\pi)$ and $ \psi^{min}(\pi)$ satisfying the inequality \eqref{clo rel for max 2}. This would give a generalization of the definition of $\psi^{max}(\pi)$ and $ \psi^{min}(\pi)$.

\begin{conj}\label{conj max}
Let $\RG$ be a connected reductive group defined over a non-Archimedean local field. Assume that there is a local Arthur packets theory for $G$ as conjectured in \cite[Conjecture 6.1]{Art89}. Let $\pi\in\Pi_A(G)$. Then, for any $\psi_1,\psi_2\in\Psi(\pi)$, we have $\lambda_{\phi_{\psi_1}}=\lambda_{\phi_{\psi_2}}.$ Furthermore, there are unique elements $\psi^{max}(\pi)$ and $\psi^{min}(\pi)$ in $\Psi(\pi)$ such that for any local Arthur parameter $\psi \in \Psi(\pi)$, the following inequality holds
\begin{align*}
    \psi^{max}(\pi) \geq_C \psi \geq_C \psi^{min}(\pi).
\end{align*} 
\end{conj}

We also expect that there is an analogous conjecture for the Archimedean case and it is a very interesting question to define ``the" local Arthur parameters for irreducible representations of Arthur type for any real or complex connected reductive group.

\subsection{\texorpdfstring{The partition orderings $\geq_{D},$ $\geq_A$}{}}
In this subsection, we recall the definition of the orderings $\geq_{D}$ and $\geq_{A}$ from \cite[\S 11.1]{HLL22}. Let $\psi$ be a local Arthur parameter of $G_n$ and write
\[ \psi= \bigoplus_{i=1}^r \phi_i|\cdot|^{x_i} \otimes S_{a_i} \otimes S_{b_i}. \]
Let $d_i= \dim \phi_i $ (as a representation of $W_F$). Then we consider the partition
\begin{align*}
    \underline{p}^A(\psi)&:=[ b_1^{d_1a_1},\dots, b_r^{d_r a_r} ],
\end{align*}
which is exactly the partition corresponding to the nilpotent orbit of $\widehat{\RG}_n(\mathbb{C})$ containing the element 
\[d(\psi|_{\SL_2^{A}(\mathbb{C})}) \left( \begin{pmatrix}
0 &1 \\0 &0
\end{pmatrix} \right). \]
 We remark that the partition $\underline{p}^A(\psi)$ is a key ingredient in the Jiang conjecture (\cite[Conjecture 4.2]{Jia14} and \cite[Conjecture 1.6]{LS22}). Now we define the ordering $\geq_A$.

\begin{defn}\label{def Jordering}
We define a preorder $\geq_A$ on $\Psi^+(G_n)$ by $\psi_1 \geq_A \psi_2$ 
if 
$$\underline{p}^{A}(\psi_1) \leq \underline{p}^{A}(\psi_2)$$
under the dominance order of partitions. 
\end{defn}

Note that $\psi \in \Psi(G_n)$ is tempered if and only if $\underline{p}^A(\psi)=[1^{N}]$ ($N=2n$ or $2n+1$), which is the smallest partition under the dominance order. That is the reason that we reverse the direction of the inequality for the partitions in the above definition. Thus, $\geq_A$ gives a measurement of ``temperedness". I.e., if $\psi_1 \geq_A \psi_2$, then $\psi_1$ is ``more tempered" than $\psi_2$. Also note that it is possible that $\underline{p}^A(\psi_1)=\underline{p}^A(\psi_2)$ but $\psi_1 \neq \psi_2$, and hence $\geq_A$ is only a preorder but not a partial order.

% The following analogue of Theorem \ref{thm max min intro} for the ordering $\geq_A$ is proved in \cite[\S 11.1]{HLL22}, which again supports the intuition that $\psi^{max}(\pi)$ (resp. $\psi^{min}(\pi)$) is the ``most tempered" (resp. ``least tempered") member in the set $\Psi(\pi)$.

% \begin{thm}[{\cite[\S 11.1]{HLL22}}]\label{thm Jiang's partition intro} \
% \begin{enumerate}
%     \item If $T$ is a raising operator, then 
%     \[T(\psi) \geq_A \psi,\]
%     for any $\psi \in \Psi(G_n).$
%     In other words, if $\psi \geq_{O} \psi'$, then $\psi \geq_A \psi'$.
%     \item Let $\pi\in \Pi_{A}(G_n)$. The parameters $\psi^{max}(\pi)$ and $ \psi^{min}(\pi)$ are the unique elements in $\Psi(\pi)$ such that for any $\psi \in \Psi(\pi)$, the following inequality holds
%     \[ \psi^{max}(\pi) \geq_A \psi \geq_A \psi^{min}(\pi).\]
% \end{enumerate}
% \end{thm}

It is natural to consider the restriction to the Deligne-$\SL_2(\mathbb{C})$ instead of the Arthur-$\SL_2(\mathbb{C})$. Let
\[\underline{p}^D(\psi):=[a_1^{d_1b_1},\dots, a_r^{d_rb_r}],\]
which is exactly the partition corresponding to the nilpotent orbit of $\widehat{\RG}_n(\BC)$ containing the element
\[ d(\psi|_{\SL_2^{D}(\mathbb{C})})\left(\begin{pmatrix}
0&1\\0&0
\end{pmatrix}\right).\]
Note that $\underline{p}^D(\psi)=\underline{p}^A(\widehat{\psi})$ (see \eqref{psi dual} for the definition of $\widehat{\psi}$). Now we define the last ordering $\geq_{D}$.

\begin{defn}\label{def Dordering}
We define a preorder $\geq_{D}$ on $\Psi^+(G_n)$ by $\psi_1 \geq_{D} \psi_2$ if $\underline{p}^D(\psi_1)\geq \underline{p}^D(\psi_2)$.
\end{defn}

% At the end of \S \ref{sec noncontainment}, we prove the following analogue of Theorem \ref{thm max min intro} for the ordering $\geq_D$. 

% \begin{thm}\label{thm D ordering intro} \
% \begin{enumerate} 
%     \item If $T$ is a raising operator, then
%     \[T(\psi) \geq_{D} \psi,\]
%     for any $\psi\in\Psi(G_n).$ In other words, if $\psi \geq_{O} \psi'$, then $\psi \geq_{D} \psi'$.
%     \item Let $\pi\in \Pi_{A}(G_n)$. The parameters $\psi^{max}(\pi)$ and $ \psi^{min}(\pi)$ are the unique elements in $\Psi(\pi)$such that for any $\psi \in \Psi(\pi)$, the following inequality holds
%     \[ \psi^{max}(\pi) \geq_D \psi \geq_D \psi^{min}(\pi).\]
% \end{enumerate}
% \end{thm}

\subsection{Relation among the four orderings}\label{sec relations}
In this subsection, we summarize the relations among the four orderings in short. By Part (1) of Theorem \ref{thm three orderings},  the operator ordering $\geq_{O}$ implies $\geq_A,$ $\geq_D$, and $\geq_C$,  but the converse is not true in general (see Example \ref{exmp order comparison}). By certain geometric argument, one can show that the closure ordering $\geq_C$ implies $\geq_D$ (a detailed combinatorial proof is provided in Corollary \ref{cor comparison closure partition}), but the converse is also not true in general (again, see Example \ref{exmp order comparison}). The orderings $\geq_A$ and $\geq_D$ are related by taking the dual, i.e., $\psi\geq_A \psi'$ if and only if $\widehat{\psi'}\geq_D\widehat{\psi}$.  It is certainly an interesting question to study how the non-extremal elements in $\Psi(\pi)$ behave under these four orderings, particularly the partial order $\geq_C$. This is a work in progress of the authors.

\section{ \texorpdfstring{The $C$ ordering}{}}\label{sec operator}
The goal of this section is to prove Theorem \ref{thm three orderings} for the closure ordering $\geq_C$. Let $T$ be a raising operator applicable on a local Arthur parameter $\psi$. The relation between $\phi_1:=\phi_{T(\psi)}$ and $\phi_2:=\phi_{\psi}$ is stated explicitly in Definition \ref{def operators on parameters} and Remark \ref{rmk operator}. First, we develop two reduction lemmas that construct unramified $L$-parameters (i.e., trivial on the inertia group) $\phi_{1,ur}$ and $\phi_{2,ur}$ such that $ \phi_1 \geq_C \phi_2$ if $\phi_{1,ur} \geq_C \phi_{2,ur}$ (see \S \ref{sec two reduction lemmas}, Lemmas \ref{lem common part} and \ref{lem reduction to unramified} below). Then we recall the notion of \emph{rank triangles} from \cite[\S 10.2.1]{CFMMX22} that determines the closure ordering for unramified $L$-parameters (see \S\ref{sec closure ordering for unramified}). Finally, we show that $\phi_{1,ur} \geq_C \phi_{2,ur}$ by explicit computation and prove Theorem \ref{thm three orderings} for $C$ ordering in \S\ref{sec computation for operators}.

Part (2) of Theorem \ref{thm three orderings} follows directly from its Part (1) and Theorem \ref{thm max min intro}.
Let us demonstrate how to show Theorem \ref{thm three orderings} (1) for $C$ ordering using the reduction lemmas below in the following example.

\begin{exmp}
Let $\rho_1$ (resp. $\rho_2$) be an irreducible symplectic (resp. orthogonal) representation of dimension $2m_1$ (resp. $2m_2+1$) of $W_F$ such that $\det(\rho_2)$ is trivial. Consider the following two local Arthur parameters of $\Sp_{4m_1+2m_1}(F)$ of good parity
\begin{align*}
    \psi_1&=\rho_1 \otimes S_{1} \otimes S_2 + \rho_2 \otimes S_1 \otimes S_1,\\
    \psi_2&=\rho_1 \otimes S_{2} \otimes S_1 + \rho_2 \otimes S_1 \otimes S_1.
\end{align*}
Note that $\psi_2= dual_1^{-}(\psi_1)$, where the index $1$ is the unique element in $I_{\rho_1}$. We have
\begin{align*}
    \phi_{\psi_1}&= \rho_1|\cdot|^{\half{1}}\otimes S_1+\rho_1|\cdot|^{\half{-1}}\otimes S_1+ \rho_2\otimes S_1,\\
    \phi_{\psi_2}&= \rho_1 \otimes S_2+ \rho_2 \otimes S_1.
\end{align*}
First, Lemma \ref{lem common part} below allows us to cancel the common part $\rho_2 \otimes S_1$. To be explicit, consider the following local $L$-parameters of (split) $\SO_{4m_1}(F)$
\begin{align*}
    \phi_1'&=\rho_1|\cdot|^{\half{1}}\otimes S_1+\rho_1|\cdot|^{\half{-1}}\otimes S_1,\\
    \phi_2'&= \rho_1 \otimes S_2.
\end{align*}
Then $\phi_{\psi_2} \geq_C \phi_{\psi_1}$ if $\phi_2' \geq_C \phi_1'$. Next, Lemma \ref{lem reduction to unramified} below allows us to reduce to the unramified case. To be explicit, consider the following local $L$-parameters of $\SO_{3}(F)$
\begin{align*}
    \phi_{1,ur}'&=|\cdot|^{\half{1}}\otimes S_1+|\cdot|^{\half{-1}}\otimes S_1,\\
    \phi_{2,ur}'&= |\cdot|^{0}\otimes S_2.
\end{align*}
Then $\phi_2' \geq_C \phi_1'$ if and only if $\phi_{2,ur}' \geq_{C} \phi_{1,ur}'$. Finally, one can compute their rank triangles to conclude that $\phi_{2,ur}'\geq_C \phi_{1,ur}'$. In summary, we have verified \[dual_1^{-}(\psi_1)=\psi_2\geq_C \psi_1.\]
\end{exmp}

\subsection{Two reduction lemmas}\label{sec two reduction lemmas}

In this subsection, we let $\RG_n$ denote one of the quasi-split groups $\Sp_{2n}, \SO_{2n+1}$ or $\SO_{2n}$, and let $\phi_1, \phi_2$ be two local $L$-parameters of $G_n$ that share the same infinitesimal character $\lambda$. We introduce two reduction lemmas for comparing the closure ordering for $\phi_1$ and $\phi_2$.

The first lemma allows us to cancel the common part of $\phi_1$ and $\phi_2$ that is self-dual.

\begin{lemma}\label{lem common part}
If $\phi_1=\phi+ \phi_1'$ and $\phi_2=\phi+ \phi_2'$ are local $L$-parameters of $G_n$ such that $\lambda_{\phi_1}=\lambda_{\phi_2}$ and $\phi, \phi_1', \phi_2'$ are all self-dual, then $\phi_1 \geq_C \phi_2 $ if $\phi_1' \geq_C \phi_2'$. Here we regard $\phi_1', \phi_2'$ as local $L$-parameters of other quasi-split classical groups.
\end{lemma}
\begin{proof}
Write $\lambda':=\lambda_{\phi_i'}$ for short. Then we may write $\lambda:= \lambda_{\phi_1}= \lambda_{\phi}\oplus \lambda'$. Let $E$ (resp. $E_{\phi},E'$) be the space that the image of $\lambda$ (resp. $\lambda_{\phi}, \lambda'$) acts on. Then $E= E_{\phi} \oplus E'$, and both $\phi_i$ and $\lambda$ have image in $\Aut(E_{\phi}) \times \Aut(E') \subseteq \Aut(E)$. In other words, we may visualize the decompositions as follows after choosing a basis 
\[ \phi_i= \left( \begin{array}{c|ccc}
     \phi& &&  \\
     \hline
     & &&\\
          & &\phi_i'&\\
               & &&\\
\end{array} \right),\ \ \lambda_{\phi_i}= \left( \begin{array}{c|ccc}
     \lambda_{\phi} & &&  \\
     \hline
     & &&\\
          & &\lambda'&\\
               & &&\\
\end{array} \right).  \]

Let $\iota$ be the injection from $ \End(E_{\phi}) \times \End(E')$ to $\End(E)$, and denote
\[X:=d(\phi|_{\SL_2(\BC)})\left( \begin{pmatrix}
     0&1\\0&0
     \end{pmatrix} \right).\]
Consider the following sets
\[H':= \{\iota(I,h) \in \Aut(E) \ | \ h \in H_{\lambda'} \},\ V':=\{ \iota(X,v) \in \End(E)\ | \ v \in V_{\lambda'}\}.\]
That is
\[  H'=\left( \begin{array}{c|ccc}
     I& &&  \\
     \hline
     & &&\\
          & &H_{\lambda'}&\\
               & &&\\
\end{array} \right),\ V'=\left( \begin{array}{c|ccc}
     X& &&  \\
     \hline
     & &&\\
          & &V_{\lambda'}&\\
               & &&\\
\end{array} \right). \]
One can check that $H'$ is a subgroup of $H_{\lambda}$, and $V'$ is a closed subvariety of $V_{\lambda}$. Clearly, we have $H' \cong H_{\lambda'}$, $V' \cong V_{\lambda'}$, and the conjugation action of $H'$ on $V'$ can be identified with the adjoint action of $H_{\lambda'}$ on $V_{\lambda'}$. Let 
\[ X_i:= d(\phi_i|_{\SL_2(\BC)})\left( \begin{pmatrix}
     0&1\\0&0
     \end{pmatrix} \right),\ X_i':=d(\phi_i'|_{\SL_2(\BC)})\left( \begin{pmatrix}
     0&1\\0&0
     \end{pmatrix} \right).\]
Then $X_i= \iota(X,X_i')$. That is
\[ X_i= \left(\begin{array}{c|ccc}
     X& &&  \\
     \hline
     & &&\\
          & &X_i'&\\
               & &&\\
\end{array} \right). \]

The assumption $\phi_1'\geq_C \phi_2'$ implies that $X_2'$ is in the closure of the $H'$ orbit of $X_1'$. Therefore, $X_2$ is in the closure of the $H_{\lambda}$ orbit of $X_1$, and hence $\phi_1 \geq_C \phi_2$. This completes the proof of the lemma.
\end{proof}

As a corollary, we show that it suffices to consider representations of good parity for Theorem \ref{thm max intro} and Theorem \ref{thm three orderings} for $C$ ordering.

\begin{cor}\label{cor reduction to gp}
Theorem \ref{thm three orderings} for $C$ ordering (resp. Theorem \ref{thm max intro})  holds for any representation $\pi\in \Pi_A(G_n)$ if and only if it holds for any representation $\pi_0\in \Pi_{A}^{gp}(G_n)$.
\end{cor}
\begin{proof}
For an arbitrary representation $\pi \in \Pi_{\psi}$, we want to show that
\begin{align}\label{eq reduction to gp}
    \phi_{\pi} \geq_C \phi_{\psi^{max}(\pi)} \geq_C \phi_{\psi} \geq_C \phi_{\psi^{min}(\pi)}.
\end{align}
Applying Theorem \ref{thm reduction to gp}, we have
\begin{align*}
    \begin{cases}
    \psi&=( \psi_1 \oplus \psi_1^{\vee})\oplus \psi_0,\\
    \psi^{max}(\pi)&=( \psi_1 \oplus \psi_1^{\vee})\oplus \psi^{max}(\pi_0),\\
    \psi^{min}(\pi)&=( \psi_1 \oplus \psi_1^{\vee})\oplus \psi^{min}(\pi_0),
    \end{cases}
  \ \
    \begin{cases}
    \phi_\psi&=(\phi_{\psi_1}\oplus\phi_{\psi_1^{\vee}})\oplus \phi_{\psi_0},\\
    \phi_{\psi^{max}(\pi)}&=(\phi_{\psi_1}\oplus\phi_{\psi_1^{\vee}})\oplus \phi_{\psi^{max}(\pi_0)},\\
    \phi_{\psi^{min}(\pi)}&=(\phi_{\psi_1}\oplus\phi_{\psi_1^{\vee}})\oplus \phi_{\psi^{min}(\pi_0)},\\
    \phi_{\pi}&=(\phi_{\psi_1}\oplus\phi_{\psi_1^{\vee}})\oplus \phi_{\pi_0},
    \end{cases}  
\end{align*}
where $\pi_0 \in \Pi_{\psi_0} \subseteq \Pi_{A}^{gp}(G_{n_0})$. Therefore, Lemma \ref{lem common part} implies that \eqref{eq reduction to gp} holds if 
\[ \phi_{\pi_0} \geq_C \phi_{\psi^{max}(\pi_0)} \geq_C \phi_{\psi_0} \geq_C \phi_{\psi^{min}(\pi_0)},\]
which follows from the assumption. This completes the proof of the corollary.
\end{proof}

Now assume the restriction of $\phi_1, \phi_2$ on $W_F$ is a direct sum of twists of the same irreducible self-dual representation $\rho$ of $W_F$. To show $\phi_1 \geq_C \phi_2,$ the following lemma allows us to reduce to show $\phi_{1,ur} \geq_C \phi_{2,ur}$ where both $\phi_{1,ur}, \phi_{2,ur}$ are unramified local $L$-parameters of quasi-split classical groups.

\begin{lemma}\label{lem reduction to unramified}
Suppose $\phi_1,\phi_2$ are local $L$-parameters of $G_n$ that share the same infinitesimal character. We assume they are of the form
\[\phi_i= \bigoplus_{j \in I_i} \rho|\cdot|^{x_j} \otimes S_{a_j},\]
where $\rho$ is a self-dual irreducible representation of $W_F$ and $x_j \in \R$. Then we define 
\[ \phi_{i,ur}:= \bigoplus_{j \in I_i} |\cdot|^{x_j} \otimes S_{a_j},\]
which are local $L$-parameters of quasi-split $\Sp_{2m}(F), \SO_{2m+1}(F)$ or $\SO_{2m}(F)$. Then $\phi_{1,ur}$ and $\phi_{2,ur}$ also share the same infinitesimal character, and $\phi_1\geq_C \phi_2$ if and only if $\phi_{1,ur} \geq_C \phi_{2,ur}$.
\end{lemma}
\begin{proof}
First, note that $\rho\otimes \lambda_{\phi_{i,ur}}= \lambda_{\phi_i}$, so $\lambda_{\phi_{1,ur}}=\lambda_{\phi_{2,ur}}$. We denote $\lambda_{ur}:=\lambda_{\phi_{1,ur}}$. Let $E$ (resp. $E_{\rho},E_{ur}$) be the space that the image of $\lambda$ (resp. $\rho, \lambda_{ur}$) acts on. That is, we view 
\begin{align*}
    \phi_i&: W_F \times \SL_2(\BC) \to \GL(E),\\
        \rho \otimes S_1&: W_F \times \SL_2(\BC) \to \GL(E_{\rho}),\\
        \phi_{i,ur}&: W_F \times \SL_2(\BC) \to \GL(E_{ur}).
\end{align*}
Then we may identify $\phi_i=(\rho \otimes S_1) \otimes \phi_{i,ur}$, where the second $\otimes$ is the tensor of representations of $W_F \times \SL_2(\BC)$, and $ E= E_{\rho} \otimes E_{ur}$. Consider
\[ H':=\{ I\otimes h \in \Aut(E)\ | \ h \in H_{\lambda_{ur}} \},\ V':=\{ I \otimes v \in \End(E)\ | \ v \in V_{\lambda_{ur}} \}, \]
where we view $V_{\lambda_{ur}}$ as a closed subset of $ \End(E_{ur})$, and $H_{\lambda_{ur}}$ as a subgroup of $\Aut(E_{ur})$. We have $H'$ is isomorphic to $H_{\lambda_{ur}}$ and $V'$ is isomorphic to $V_{\lambda_{ur}}$. The conjugation action of $H'$ on $V'$ can be identified with the action of $H_{\lambda_{ur}}$ on $V_{\lambda_{ur}}$. Note that
\begin{align*}
     \End(E_{ur}) &\hookrightarrow \End(E),\\
    h &\mapsto I \otimes h 
\end{align*}
is a closed immersion, and hence $V'$ is closed in $\End(E)$. To prove the lemma, it suffices to verify the following claims:
\begin{enumerate}
    \item $H'=H_{\lambda}$, and
    \item $V'=V_{\lambda}$.
\end{enumerate}

Now we verify Claim (1). We first show that $H' \subseteq H_{\lambda}$, i.e., for any $h \in H_{\lambda_{ur}}$,  $I \otimes h$ is also in $\widehat{\RG}_n(\BC)$ and commutes with the image of $\lambda$. Since the identity preserves the bilinear form on $E_{\rho}$ and $ h$ preserves the bilinear form on $E_{ur}$, we see that $I \otimes  h$ preserves the symplectic or orthogonal form on $E=E_{\rho} \otimes E_{ur}$. This shows that $I \otimes h \in \widehat{\RG}_n(\BC)$. To check $I \otimes h$ commutes with the image of $\lambda$, it suffices to check that 
\[\lambda(w) \circ (I\otimes h)(e_1 \otimes e_2)= (I\otimes h) \circ \lambda(w) (e_1 \otimes e_2),\]
for any $w \in W_F, e_1 \in E_{\rho}$ and $e_2 \in E_{ur}$. Recall that we have $\lambda= \rho \otimes \lambda_{ur}$. Thus
\begin{align*}
    \lambda(w) \circ (I\otimes h)(e_1 \otimes e_2)
    &= \lambda(w) ( e_1\otimes h(e_2))\\
    &= \rho(w)(e_1)\otimes (\lambda_{ur}(w) \circ h(e_2))\\
    &=\rho(w)(e_1)\otimes (  h \circ \lambda_{ur}(w)(e_2))\\
    &= (I\otimes h)( \rho(w)(e_1) \otimes \lambda_{ur}(w)( e_2))\\
    &= (I\otimes h )\circ \lambda(w) (e_1\otimes e_2).
\end{align*}
This completes the verification that $H' \subseteq H_{\lambda}$.

Now we show that $H_{\lambda} \subseteq H'$. Write $ \lambda_{ur}=\oplus_{i=1}^r |\cdot|^{y_{i}} $ and choose a basis $\{e_1, \dots , e_r\}$ of $E_{ur}$ such that 
\[\lambda(w)(e_{\rho} \otimes e_i)= |w|^{y_i}\rho(w)(e_{\rho}) \otimes e_i,   \]
for any $e_{\rho} \in E_{\rho}$ and $w \in W_F$. Denote $E_i:= E_{\rho} \otimes \BC e_i$, which is the underlying space of the irreducible representation $\rho \otimes |\cdot|^{y_i}$. Take a $T \in \Aut(E_{\rho}\otimes E_{ur})$ that commutes with the image of $\lambda$. From the assumptions, $ T(E_i)$ is invariant under the action of $W_F$, and hence $T(E_i)=E_j$ for some $j$ such that $y_i=y_j$. Then we may write 
\[ T|_{E_i}( e_{\rho} \otimes e_i )= T_i(e_{\rho}) \otimes e_j, \]
for some $T_i \in \Aut(E_{\rho})$. Since $T$ commutes with the image of $\lambda$, $T_i$ also commutes with the image of $\rho$. By Schur's lemma, the irreducibility of $\rho$ implies that $T_i$ is a scalar multiplication. Therefore, we may write $T=I \otimes h$ for some $h \in \Aut(E_{ur})$. Moreover, since $T$ sends $E_i$ to $E_j$ such that $y_i=y_j$, the linear map $h$ also commutes with the image of $\lambda_{ur}$. Finally, we further assume $T$ is in $\widehat{\RG}_n(\BC)$, i.e., $T=I\otimes h$ preserves the symplectic or orthogonal form on $E_{\rho} \otimes E_{ur}$. Then $h$ must also preserve the symplectic or orthogonal form on $E_{ur}$. This completes the proof that $H_{\lambda} \subseteq H'$ and the verification of Claim (1).

For Claim (2), we first check that $V' \subseteq V_{\lambda}$, i.e., $V' \subseteq \widehat{\mathfrak{g}}_n$, and any element $I \otimes v$ in $V'$ satisfies the following equation for any $w \in W_F$
\[ \Ad(\lambda(w))(I \otimes v)=|w|I \otimes v.  \]
Let $\widehat{\RG}_{\lambda_{ur}}(\BC)$ be the subgroup of $\Aut(E_{ur})$ that preserves the symplectic or orthogonal form on $E_{ur}$. From the proof of Claim (1) above, the map
\begin{align*}
    \widehat{\RG}_{\lambda_{ur}}(\BC) &\hookrightarrow \widehat{\RG}_n(\BC),\\
    g &\mapsto I \otimes g
\end{align*}
is a Lie group homomorphism. Thus $V'$, a subset of the image of differential of above map, is contained in the Lie algebra of $\widehat{\RG}_n(\BC)$. 

Next, we check that
\[ \lambda(w) \circ (I\otimes v) \circ \lambda(w)^{-1} (e_1 \otimes e_2)=(|w| I\otimes v)(e_1 \otimes e_2), \]
for any $w \in W_F$, $e_1 \in E_{\rho}$ and $e_2 \in E_{ur}$. Using the relation $\lambda=\rho \otimes \lambda_{ur}$, we have
\begin{align*}
    \lambda(w) \circ (I\otimes v) \circ \lambda(w)^{-1} (e_1 \otimes e_2)
    =&\,(\rho(w) \circ I \circ \rho(w^{-1})(e_1) \otimes (\lambda_{ur}(w)\circ v \circ \lambda_{ur}(w)^{-1})(e_2)\\
    =&\, e_1 \otimes |w| v(e_2)\\
    =&\,(|w| I\otimes v)(e_1 \otimes e_2).
\end{align*}
This completes the verification that $V' \subseteq V_{\lambda}$.

Finally, we show that $ V_{\lambda} \subseteq V'$. Continuing the notation in the proof of Claim (1), a similar argument shows that if $T \in \End(E_{\rho}\otimes E')$ satisfies the equation
\[ \Ad(\lambda(w))T= |w|T,  \]
for any $w \in W_F$, then $T$ sends $E_i$ to $E_j$ such that $y_j=y_i+1$, and $T= I\otimes v$ for some $v \in \End(E_{ur})$. Moreover, if $T$ is in $\widehat{\mathfrak{g}}_n$, then $v$ is in $\widehat{\mathfrak{g}}_{\lambda_{ur}},$ the Lie algebra of $ \widehat{\RG}_{\lambda_{ur}}(\BC)$. This proves $V_{\lambda} \subseteq V'$, which completes the verification of Claim (2) and the lemma.
\end{proof}

\subsection{Closure Ordering for unramified \texorpdfstring{$L$}{L}-parameters}\label{sec closure ordering for unramified}
In this subsection, let $\RG_n$ denote one of the split group $\Sp_{2n}, \SO_{2n+1}$ or $\SO_{2n}$, and let $\phi_1, \phi_2$ be two unramified local $L$-parameters of $G_n$ that share the same infinitesimal character $\lambda$. We first recall that the closure ordering of $\phi_1,\phi_2$ is determined by rank triangles (see Definition \ref{def rank matrix} below) from \cite[\S 10.2.1]{CFMMX22}, and how to compute the rank triangle of each $\phi_i$ from \cite[\S 1.5]{CFK22} and \cite[\S 3.3]{CR22}. Finally, we compare the orderings $\geq_C$ and $\geq_{D}.$

First, we recall the description  of the closure ordering for nilpotent orbits of a classical Lie algebra $\mathfrak{g}$. For each nilpotent orbit $\OO$ of $\mathfrak{g}$ (over $\BC$), there is an associated partition $\underline{p}(\OO)$ (see \cite[\S 5.1]{CM93}). For any $X \in \OO$, the \emph{rank sequence} $(\rank(X^k))_{k \in \Z_{\geq 0}}$ can be computed from the partition $\underline{p}(\OO)$, and vice versa (see \cite[6.2.3]{CM93}).
% We recall the dominance order, which is a partial order on the set of partitions, in the definition below.
% \begin{defn}\label{def dominance order}
% Suppose $\underline{p}=[p_1,\dots ,p_r]$ and $\underline{q}=[q_1, \dots, q_s]$ are two partitions of $n$, i.e., 
% \begin{enumerate}
%     \item [$\oldbullet$]$\sum_{i=1}^r p_i=n= \sum_{j=1}^s q_j$, and
%     \item [$\oldbullet$] $\{p_i\}_{i=1}^r$ and $\{q_j\}_{j=1}^s$ are both non-raising sequence of positive integers.
% \end{enumerate}
%  If for any $1 \leq k \leq r$, we have
% \[ \sum_{i=1}^k p_i \geq \sum_{j=1}^k q_j,  \]
% then we say $\underline{p}$ precedes $\underline{q}$ in the dominance order and write $\underline{p} \geq \underline{q}$.
% \end{defn}
 A classical theorem due to Gerstenhaber and Hesselink (see {\cite[Theorem 6.2.5]{CM93}}) states that the closure ordering for nilpotent orbits of $\mathfrak{g}$ is equivalent to the dominance order of the corresponding partitions.

\begin{thm}[Gerstenhaber, Hesselink]\label{thm partition}
Let $\OO_1, \OO_2$ be two nilpotent orbits of a classical Lie algebra $\mathfrak{g}(\BC)$. Then $\overline{\OO_1} \supsetneq \overline{\OO_2}$ if and only if $\underline{p}(\OO_1) \gneq \underline{p}(\OO_2)$.
\end{thm}

Now let $\phi_i= \oplus_{j \in I_{i}} |\cdot|^{x_j} \otimes S_{a_j}$, for $i=1,2$, be unramified local $L$-parameters of a split classical group $G$ that share the same infinitesimal character $\lambda$. We may associate $\phi_i$ a partition $\underline{p}(\phi_i)$ by
\[ \underline{p}(\phi_i):=[ a_{j} ]_{j \in I_i},\]
which is exactly the partition associated to the $\widehat{\RG}(\BC)$-orbit of the nilpotent element $$d(\phi_i|_{\SL_2(\BC)})\left( \begin{pmatrix}
0&1\\0&0
\end{pmatrix}\right).$$
Note that if $\phi=\phi_{\psi}$ for some local Arthur parameter $\psi$, then $\underline{p}(\phi)=\underline{p}^{D}(\psi)$. By Theorem \ref{thm partition}, if $ \phi_1 \geq_C \phi_2$, then $\underline{p}(\phi_1) \geq \underline{p}(\phi_2)$ since $H_{\lambda}$ is a subgroup of $\widehat{\RG}(\BC)$. However, the converse does not hold in general. See Example \ref{exmp order comparison}(2).

To get a complete description of the closure ordering for $H_{\lambda}$-orbits in $V_{\lambda}$, we recall the concept of \emph{rank triangle} considered in \cite{CFMMX22}. The rank triangle of $\phi$ not only keeps the partition information for $\phi|_{\SL_2(\BC)}$ (see Corollary \ref{cor comparison closure partition}(1)), but also contains the twist information for $\phi|_{W_F}$.

Let $\widehat{\RG}_n(\BC) \to \GL(E)$ be the standard representation of $\widehat{\RG}_n(\BC)$. Therefore, $E$ is $(2n+1)$-dimensional if $G_n=\Sp_{2n}(F)$, and $2n$-dimensional otherwise. Take any $\Fr \in W_F$ such that $|\Fr|=q$. Let $\{q^{\lambda_0}, \dots , q^{\lambda_r}\}$ be the set of distinct eigenvalues of $\lambda(\Fr):W_F \to \GL(E)$, and denote $E_{\alpha}$ the $q^{\lambda_{\alpha}}$-eigenspace of $\lambda(\Fr)$ in $E$. For simplicity, we keep the following assumptions.
\begin{assu}\label{assu good parity}
With the notation above, we assume the following:
\begin{enumerate}
    \item [$\oldbullet$] $\lambda_{\alpha} \in \half{1}\Z$, 
    \item [$\oldbullet$] $\lambda_{\alpha}-\lambda_{\beta} \in \Z_{<0}$ for $\alpha <\beta$, and
    \item [$\oldbullet$] $\lambda_{\alpha}=-\lambda_{r-\alpha}$.
\end{enumerate}
\end{assu}

These assumptions are natural in the case we consider. To be explicit, let $\pi \in \Pi_{\psi}$, where $\psi$ is a local Arthur parameter of good parity of $\Sp_{2n}$ or split $\SO_{2n+1}$. Then for any subrepresentation $\phi'$ of $\phi_{\pi}$ of the form $\phi'= \oplus_{j \in I} \rho|\cdot|^{x_j} \otimes S_{a_j}$, the parameter
\[ \phi_{ur}':=  \oplus_{j \in I} |\cdot|^{x_j} \otimes S_{a_j},\]
coming from the reduction in Lemma \ref{lem reduction to unramified}, satisfies the three assumptions. Note that the third assumption follows from the first two assumptions and the fact that $\lambda_{\alpha}$ has image in $\widehat{\RG}_n(\BC)$.

Recall that when $\lambda$ is unramified,
\[ V_{\lambda}:= \{ x \in \widehat{\mathfrak{g}}_n\ | \ \text{Ad}(\lambda(\text{Fr}))x=qx \}.\]
Therefore, we may identify $V_{\lambda}$ as a subset of 
$$\Hom(E_0,E_1) \times \cdots \times \Hom(E_{r-1},E_{r})$$
and write each element $v \in V_{\lambda}$ as $v= (v_1,\dots ,v_r)$, where $v_{\alpha}\in \Hom(E_{\alpha-1},E_{\alpha})$. Clearly, $v_{\alpha}=0$ if $\lambda_{\alpha}-\lambda_{\alpha-1}>1$. We consider the following definition of \emph{rank triangles}, where we put the triangle into an upper triangular matrix.

\begin{defn}\label{def rank matrix}
Suppose $\lambda$ satisfies Assumption \ref{assu good parity}. Under the notation in the previous paragraphs, for each $v=(v_1,\dots ,v_r)\in V_{\lambda}$, we define
$$v_{\alpha\beta}\in \Hom(E_{\alpha-1},E_{\beta})$$
by
\[ v_{\alpha\beta}:=\begin{cases} v_{\beta} \circ \cdots \circ v_{\alpha}, &\text{if }\alpha \leq \beta,\\
0, &\text{otherwise.}\end{cases} \]
Then we define an upper triangular $r\times r$ matrix $r(v)=(r_{\alpha\beta}(v))_{\alpha\beta}$ by
\[ r_{\alpha \beta}(v):=\rank(v_{\alpha\beta}). \]
Suppose $\phi$ is an unramified local $L$-parameter where $\lambda_{\phi}$ satisfies Assumption \ref{assu good parity}. Take any element $v\in C_{\phi} \subseteq V_{\lambda_{\phi}}$ and define the \emph{rank triangle} $$r(\phi)=(r_{\alpha\beta}(\phi))_{\alpha\beta}$$
of $\phi$ by
\[ r_{\alpha\beta}(\phi):= r_{\alpha\beta}(v).\]
\end{defn}

The following lemma is contained in \cite[\S 10.2.1]{CFMMX22}.
\begin{lemma}\label{lem rank matrix closure relation}
Suppose $\phi_1, \phi_2$ are unramified local $L$-parameters of the group $G_n$ that share the same infinitesimal character $\lambda$ satisfying Assumption \ref{assu good parity}. Then $\phi_1 \gneq_C \phi_2$ if and only if $r_{\alpha\beta}(\phi_1) \geq r_{\alpha\beta}(\phi_2)$ for all $\alpha,\beta \in \{1,\dots ,r\}$ and $r(\phi_1) \neq r(\phi_2)$.
\end{lemma}

Now we rephrase how to compute the rank triangle $(r_{\alpha\beta}(\phi))_{\alpha\beta}$ from the decomposition of $\phi$ (\cite[\S 1.5]{CFK22} and \cite[\S 3.3]{CR22}), and give a proof of it for completeness. Assume $\lambda_{\phi}$ satisfies Assumption \ref{assu good parity} and continue the notation from above. Write
\[ \phi=\bigoplus_{i \in I} |\cdot|^{x_i} \otimes S_{a_i}. \]
To the irreducible representation $|\cdot|^x\otimes S_a$, we associate a set
$$Q_{|\cdot|^x\otimes S_a}:=\{q^{\frac{a-1}{2}+x},q^{\frac{a-3}{2}+x},\dots,q^{\frac{1-a}{2}+x}\}.$$ 
To each direct summand $|\cdot|^x\otimes S_a$ of $\phi$, we define an $r\times r$ matrix $$M_{|\cdot|^x\otimes S_a}:=(m_{\alpha\beta})_{\alpha\beta}$$
as follows.
If $\alpha>\beta,$ then $m_{\alpha\beta}:=0$, i.e., $M_{|\cdot|^x\otimes S_a}$ is upper triangular. If $\alpha\leq \beta,$ then we define $m_{\alpha\beta}:=1$ if $\{q^{\lambda_{\alpha-1}},q^{\lambda_{\alpha}},\dots,q^{\lambda_{\beta}}\}\subseteq Q_{|\cdot|^x\otimes S_a}$. Otherwise, we set $m_{\alpha\beta}:=0.$ By convention, if $a=1,$ then we set $M_{|\cdot|^x \otimes S_1}:=0_{r\times r}.$

\begin{lemma}\label{lem computation of rank matrix}
If $\phi=\bigoplus_{i \in I} |\cdot|^{x_i}\otimes S_{a_i}$ and $\lambda_{\phi}$ satisfies Assumption \ref{assu good parity}, then $r(\phi)=\sum_{i \in I} M_{|\cdot|^{x_i}\otimes S_{a_i}}.$
\end{lemma}

\begin{proof}
For $i \in I$, let $E^{i} \subseteq E$ be the subspace that $|\cdot|^{x_i} \otimes S_{a_i}$ acts on. Then for each eigenspace $E_{\alpha}$ of $\lambda(\Fr)$, we have a decomposition
\[ E_{\alpha}= \bigoplus_{i \in I} E_{\alpha}^{i}, \]
where $E_{\alpha}^{i}= E_{\alpha} \cap E^{i}$. Let $X:= d (\phi|_{\SL_2(\BC)})\left(\begin{pmatrix}
0&1\\0&0
\end{pmatrix}\right)=(v_1, \dots, v_r).$ We compute the rank of \[ v_{\alpha\beta}:=v_{\beta} \circ \cdots \circ v_\alpha \in \Hom(E_{\alpha-1}, E_{\beta})= \bigtimes_{i,j \in I} \Hom(E_{\alpha}^{i}, E_{\beta}^{j}) \]
for any $\beta \geq \alpha$. Since each $ |\cdot|^{x_i} \otimes S_{a_i}$ is a subrepresentation of $\phi$, we see that $X$ leaves the subspaces $E^i$ invariant, and hence $v_{\alpha\beta}$ is indeed in a smaller set
\[v_{\alpha\beta} \in  \bigtimes_{i \in I} \Hom(E_{\alpha-1}^i, E_{\beta}^i). \]
Let $v_{\alpha\beta}^i $ denote
the projection of $v_{\alpha\beta}$ to $\Hom(E_{\alpha-1}^i, E_{\beta}^i)$ . We have \[\rank(v_{\alpha\beta})= \sum_{i\in I} \rank(v_{\alpha\beta}^i),\]
 and hence
 \[r_{\alpha\beta}(\phi)= \rank(v_{ \alpha\beta})=\sum_{i \in I}\rank(v_{ \alpha\beta}^i).\]
Thus, it remains to check that $ \rank(v_{ \alpha\beta}^i)=(M_{|\cdot|^{x_i}\otimes S_{a_i}})_{\alpha\beta}$.
 
Suppose $ (M_{|\cdot|^{x_i}\otimes S_{a_i}})_{\alpha\beta}=1$. Then the subspaces $E_{\alpha-1+t}^i$ are all one dimensional for $t=0,\dots ,\beta-\alpha+1$. In this case, $v_{ \alpha-1+t}^i \in \Hom(E_{\alpha-1+t}^i,E_{\alpha+t}^i)$ are all isomorphisms for $t=0,\dots ,\beta-\alpha$, so $\rank(v_{\alpha\beta}^i)=1$.
Suppose $ (M_{|\cdot|^{x_i}\otimes S_{a_i}})_{\alpha\beta}=0$. Then at least one of $E_{\alpha-1}^i$ or $E_{\beta}^i$ is zero, and hence $v_{ \alpha\beta}^i \in \Hom(E_{\alpha-1}^i, E_{\beta}^i) $ automatically has rank zero. This completes the verification that $ \rank(v_{ \alpha\beta}^i)=(M_{|\cdot|^{x_i}\otimes S_{a_i}})_{\alpha\beta}$ and the proof of the lemma.
\end{proof}

Now we give several examples for the computation of rank triangles and the comparison between four orderings $\geq_{O}$ (Definition \ref{def operator ordering}), $\geq_{C}$ (Definition \ref{def Cordering}), $\geq_A$ (Definition \ref{def Jordering}), and $\geq_{D}$ (Definition \ref{def Dordering}). 
\begin{exmp}\label{exmp order comparison}
Let $\rho$ be the trivial representation of $W_F$. See \cite[\S 2.4]{HLL22} for the parametrization of tempered representations.
\begin{enumerate}
    \item [(1)] Consider the tempered representation $\pi_1$ of $\SO_{9}(F)$ given by 
    \[\pi_1=\pi\left(\half{1}^{-},\half{1}^{-},\half{3}^{+}\right).\]
    Then applying \cite[Algorithm 7.9, Theorem 7.4]{HLL22}, we see $\Psi(\pi_1)$ consists of four local Arthur parameters given as follows.
    \begin{align*}
        \psi_1&=\rho \otimes S_2 \otimes S_1+ \rho \otimes S_2 \otimes S_1+ \rho \otimes S_4 \otimes S_1,\\
        \psi_2&=\rho \otimes S_1 \otimes S_2+ \rho \otimes S_2 \otimes S_1+ \rho \otimes S_4 \otimes S_1,\\
        \psi_3&=\rho \otimes S_2 \otimes S_1+ \rho \otimes S_3 \otimes S_2,\\
        \psi_4&=\rho \otimes S_1 \otimes S_2+ \rho \otimes S_3 \otimes S_2.
    \end{align*}
Here is the diagram for the operator ordering $\geq_{O}$ on $\Psi(\pi_1)$.
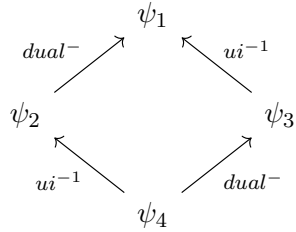
\begin{center}
  \begin{tikzcd}
    &\psi_1&\\
    \psi_2\ar[ru, "dual^{-}"]&&\psi_3\ar[lu,"ui^{-1}"']\\
    &\psi_4\ar[lu,"ui^{-1}"]\ar[ru,"dual^{-}"']&
    \end{tikzcd} 
\end{center}
Thus $\psi_1= \psi^{max}(\pi_1)$ and $\psi_4= \psi^{min}(\pi_1)$. Now we compute the rank triangles for 
\[\phi_{\psi_3}= \rho \otimes S_2 + \rho |\cdot|^{\half{1}} + \rho |\cdot|^{\half{-1}} \otimes S_3. \]
The multi-set of eigenvalues of $\lambda_{\phi_{\psi_1}}$ is $$\{q^{\half{3}},q^{\half{1}},q^{\half{1}},q^{\half{1}},q^{\half{-1}},q^{\half{-1}},q^{\half{-1}},q^{\half{-3}}\},$$
and hence we set $(\lambda_0,\lambda_1,\lambda_2,\lambda_3)=(\half{-3},\half{-1},\half{1},\half{3})$. Thus,
\[M_{S_2}= \begin{pmatrix} 0&0&0\\&1&0\\&&0 \end{pmatrix},\ M_{|\cdot|^{\half{1}} \otimes S_3}= \begin{pmatrix} 0&0&0\\&1&1\\&&1 \end{pmatrix},M_{|\cdot|^{\half{-1}} \otimes S_3}= \begin{pmatrix} 1&1&0\\&1&0\\&&0 \end{pmatrix}, \]
and hence
\[r(\phi_{\psi_3})= M_{S_2}+M_{|\cdot|^{\half{1}} \otimes S_3}+M_{|\cdot|^{\half{-1}} \otimes S_3}=\begin{pmatrix} 1&1&0\\&3&1\\&&1 \end{pmatrix}.  \]
Similarly, one can compute that
\[ r(\phi_{\psi_1} )= \begin{pmatrix} 1&1&1\\&3&1\\&&1 \end{pmatrix},\ r(\phi_{\psi_2} )= \begin{pmatrix} 1&1&1\\&2&1\\&&1 \end{pmatrix}, \ r(\phi_{\psi_4} )= \begin{pmatrix} 1&1&0\\&2&1\\&&1 \end{pmatrix}.\]
Therefore, the closure ordering $\geq_C$ is identical with the operator ordering $\geq_{O}$ on $\Psi(\pi_1)$ in this case. Next, it follows from the definition that
\[
\begin{tabular}{llll}
    $\underline{p}^A(\psi_1)=[1^8]$,& $\underline{p}^A(\psi_2)=[2,1^6]$, &$\underline{p}^A(\psi_3)=[2^3,1^2]$,  &$\underline{p}^A(\psi_4)=[2^4]$,\\
      $\underline{p}^D(\psi_1)=[4,2^2]$, &$\underline{p}^D(\psi_2)=[4,2,1^2]$, &$\underline{p}^D(\psi_3)=[3^2,2]$, &$\underline{p}^D(\psi_4)=[3^2,1^2]$,
\end{tabular}
\]
and hence $\geq_D$ is also identical with $\geq_{O}$ on $\Psi(\pi_1)$, while $\geq_A$ is a total order given by
\[ \psi_1 \geq_A \psi_{2} \geq_A \psi_3 \geq_A\psi_4.\]
Under any of the orderings on $\Psi(\pi_1)$, $\psi^{max}(\pi)=\psi_1$ (resp. $\psi^{min}(\pi)=\psi_4$) is the unique maximal (resp. minimal) element.

\item [(2)] We take $\pi_2$ to be the Aubert-Zelevinsky involution of $\pi_1$ above. Then $\Psi(\pi_2)=\{\widehat{\psi}_1,\widehat{\psi}_2,\widehat{\psi}_3,\widehat{\psi}_4\}$. Here is the diagram for $\geq_{O}$ on $\Psi(\pi_2)$.
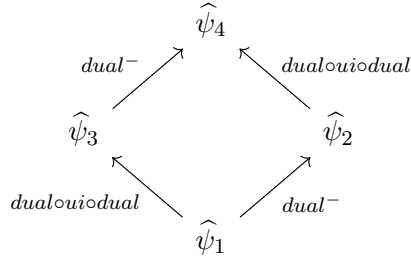
\begin{center}
  \begin{tikzcd}
    &\widehat{\psi}_4&\\
    \widehat{\psi}_3\ar[ru, "dual^{-}"]&&\widehat{\psi}_2\ar[lu,"dual \circ ui \circ dual"']\\
    &\widehat{\psi}_1\ar[lu,"dual \circ ui \circ dual"]\ar[ru,"dual^{-}"']&
    \end{tikzcd}
\end{center}
In this example, $ \geq_A$ is identical with $ \geq_{O}$ on $\Psi(\pi_2)$, while $\geq_{C}$ and $\geq_D$ are identical total orders given by
\[\widehat{\psi}_4 \geq_D \widehat{\psi}_3 \geq_D \widehat{\psi}_2 \geq_D \widehat{\psi}_1.\]
We remark that $\widehat{\psi}_3 \geq_C \widehat{\psi}_2$ but $\psi_2 \not\geq_C \psi_3$, i.e., the map $\psi \mapsto \widehat{\psi}$ is not order reversing with respect to $\geq_C$. On the other hand, it is order reversing with respect to $\geq_{O}$ (see Lemma \ref{lem order reversing}).
\item [(3)] The following example shows that $\geq_C$ is not always identical with $\geq_D$. Let $\pi$ be the unique supercuspidal representation in the tempered local $L$-packet of 
\[ \phi= \rho \otimes S_2 + \rho \otimes S_4 + \rho \otimes S_6\]
of $\SO_{13}(F)$. We have $|\Psi(\pi)|=18$. Consider $\psi_1, \psi_2 \in \Psi(\pi)$ given by
\begin{align*}
    \psi_1&=\rho \otimes S_1 \otimes S_2+ \rho \otimes S_1 \otimes S_4+ \rho \otimes S_{6} \otimes S_1,\\
    \psi_2&= \rho \otimes S_1 \otimes S_6+ \rho \otimes S_3 \otimes S_2.
\end{align*}
One can compute that $\psi_1 \geq_{D} \psi_2$, but $\psi_1$ and $ \psi_2$ are not comparable under $\geq_{C}$.
\end{enumerate}
\end{exmp}

In the corollary below, we show how to recover the partition $\underline{p}(\phi)$ from the rank triangle $r(\phi)$. Also, we give a purely combinatorial proof that $\phi \geq_C \phi'$ only if $\underline{p}(\phi) \geq \underline{p}(\phi')$.

\begin{cor}\label{cor comparison closure partition}
Suppose $\phi \in \Phi(G_n)_{\lambda}$ is unramified.
\begin{enumerate}
    \item Let \[d_s:= \sum_{t=1}^{r+1-s} r_{t (t+s-1)}(\phi), \]
which is the sum of the $(s-1)$-th off-diagonal entries of $r(\phi)$, and write
\[\underline{p}(\phi)=[(r+1)^{m_{r+1}}, r^{m_{r}}, \dots, 2^{m_{2}},1^{m_1}],\] 
where $m_s$ is the multiplicity of $s$. Then $m_s$ is given recursively by
\begin{align}
m_s=
    \begin{cases}
d_r &\text{ if }s=r+1,   \\
d_{s-1}- \sum_{t=s+1}^{r+1} (t-s+1)m_t &\text{ if }2 \leq s <r+1,\label{eq rmk partition}
\end{cases}
\end{align}  
 and $m_1$ can be computed from $m_2, \dots, m_{r+1}$.
\item If $\lambda_{\phi}=\lambda_{\phi'}$, then $\phi\geq_C \phi'$ implies that $\underline{p}(\phi) \geq \underline{p}(\phi')$. 
\item In particular, for local Arthur parameters $\psi, \psi' \in \Psi(G_n)_{\lambda}$, $\psi \geq_C \psi'$ implies that $\psi \geq_D \psi'$.
\end{enumerate}
\end{cor}
\begin{proof}
Part (1) follows directly from the fact that each summand $M_{|\cdot|^{x}\otimes S_{a}}$ contributes $ \max\{0, a-s\}$ to $d_{s}$. For Part (2), we remark that for any element $X$ in $ C_{\phi}$, we have
\[ \rank(X^{k})=d_k,\]
and one can compare the proof below with \cite[Lemma 6.2.2]{CM93}. Equivalently, we show that $\underline{p}(\phi) \not\geq \underline{p}(\phi')$ implies that $\phi \not\geq_{C} \phi'$. Write
\begin{align*}
    \underline{p}(\phi)&=[p_1,\dots ,p_x]=[(r+1)^{m_{r+1}}, \dots, 1^{m_1}],\\
    \underline{p}(\phi')&=[q_1,\dots, q_y]=[(r+1)^{m_{r+1}'}, \dots, 1^{m_1'}],
\end{align*}
and let 
\[d_s:= \sum_{t=1}^{r-s+1} r_{t (t+s-1)}(\phi),\ d_s':= \sum_{t=1}^{r-s+1} r_{t (t+s-1)}(\phi'). \]
Take the minimal $j$ such that 
\[ \sum_{i=1}^j p_i < \sum_{i=1}^j q_i.\]
Then the minimality of $j$ gives that $q_j>p_j$, and
\[ \sum_{t=p_j}^{r+1} (t-p_j)m_t = \sum_{i=1}^j (p_i-p_j)<\sum_{i=1}^j (q_i-p_j)< \sum_{t=q_j}^{r+1} (t-p_j)m_t'<\sum_{t=p_j}^{r+1} (t-p_j)m_t'.\]
From \eqref{eq rmk partition}, the left hand side and the right hand side of above inequality are exactly $d_{p_j}$ and $d_{p_j}'$ respectively. Therefore, we obtain
\[ \sum_{t=1}^{r+1-p_j} r_{t(t+p_j-1)}(\phi)=d_{p_j}<d_{p_j}'=\sum_{t=1}^{r+1-p_j} r_{t(t+p_j-1)}(\phi'). \]
It follows that there exists some $1\leq t\leq r+1-p_j$ such that 
\[r_{t(t+p_j-1)}(\phi)<r_{t(t+p_j-1)}(\phi'),\]
and we conclude that $\phi\not\geq_C\phi'$ by Lemma \ref{lem computation of rank matrix}.

Part (3) follows from Part (2) since $\underline{p}(\phi_{\psi})= \underline{p}^D(\psi)$. The proof of the corollary is now complete.
\end{proof}

\subsection{\texorpdfstring{Proof of Theorem \ref{thm three orderings} for $C$ ordering}{}} \label{sec computation for operators}

 Part (2) of Theorem \ref{thm three orderings} follows directly from its Part (1) and Theorem \ref{thm max min intro}. We prove Part (1) in the following, case by case. 

First we fix some notations. 
In each case, we construct an unramified infinitesimal character $\lambda$ of split $\Sp_{2n}, \SO_{2n+1}$ or $\SO_{2n}$ that satisfies Assumption \ref{assu good parity}. Denote  $\{ q^{\lambda_0} , \dots ,q^{\lambda_r} \}$ the set of distinct eigenvalues of $\lambda(\Fr)$, and use $\alpha, \beta$ to denote two numbers in $\{0,\dots ,r\}$ with $\lambda_{\alpha-1}={y}, \lambda_{\beta}= {x}$ such that $x \geq y$. 

The key of the proof is the following computation. Consider a representation $\psi'=|\cdot|^{0} \otimes S_{A+B+1}\otimes S_{A-B+1}$ of the group $W_{F} \times \SL_2(\BC) \times \SL_2(\BC)$, which is a subrepresentation of some local Arthur parameter $\psi$. Let  
\[ \phi_{\psi'}:= \bigoplus_{t=0}^{A-B}|\cdot|^{\half{A-B}-t} \otimes S_{A+B+1}, \]
which is a subrepresentation of the local $L$-parameter $\phi_{\psi}$. By Lemma \ref{lem computation of rank matrix}, we have
\begin{align}
    r_{\alpha\beta}( \phi_{\psi'})&=\#\{ t= 0,\dots , A-B\ | \ A-t \geq x, y \geq -B-t\}\label{eq 0}\\
    &= \max\{ \min\{A-x,A-B\}-\max\{0,-y-B\}+1  ,0 \} \nonumber\\
    &= \max\{ A+B-\max\{x,B\}-\max\{-y,B\}+1  ,0 \}\label{eq 1}.
\end{align}
Now we start the discussion of each case.

\textbf{Case 1:} $T=ui_{i,j}^{-1}$.

In this case, by Definition \ref{def operators on parameters} and Remark \ref{rmk operator}, we may write 
\[\begin{cases} T(\psi)&= \psi'+ \rho\otimes S_{A_i+B_i+1}\otimes S_{A_i-B_i+1} +\rho\otimes S_{A_j+B_j+1}\otimes S_{A_j-B_j+1},\\
\psi&= \psi'+ \rho\otimes S_{A_j+B_i+1}\otimes S_{A_j-B_i+1} +\rho\otimes S_{A_i+B_j+1}\otimes S_{A_i-B_j+1},\end{cases}\]
where $A_j \geq A_i+1 \geq B_j >B_i$. If $A_i+1=B_j$, then we omit the third term in the decomposition of $\psi$. Applying Lemma \ref{lem common part}, it suffices to show that $\phi_1 \geq_{C} \phi_2$, where
\begin{align*}
    \phi_1=&\bigoplus_{t=0}^{A_i-B_i} \rho |\cdot|^{\frac{A_i-B_i}{2}-t}\otimes S_{A_i+B_i+1}  + \bigoplus_{t=0}^{A_j-B_j} \rho |\cdot|^{\frac{A_j-B_j}{2}-t}\otimes S_{A_j+B_j+1}, \\
    \phi_2=& \bigoplus_{t=0}^{A_j-B_i} \rho |\cdot|^{\frac{A_j-B_i}{2}-t}\otimes S_{A_j+B_i+1} + \bigoplus_{t=0}^{A_i-B_j} \rho |\cdot|^{\frac{A_i-B_j}{2}-t}\otimes S_{A_i+B_j+1}.
\end{align*}
Applying Lemma \ref{lem reduction to unramified}, we may assume $\rho$ is the trivial representation. Then $\lambda_{\phi_1}= \lambda_{\phi_2}$ satisfies Assumption \ref{assu good parity}, and Lemma \ref{lem rank matrix closure relation} states that it suffices to show $r_{\alpha\beta}(\phi_1) \geq r_{\alpha\beta}(\phi_2)$ for any $\alpha, \beta$. We shall compare these two numbers explicitly by \eqref{eq 0} and \eqref{eq 1}. 

For $k,l \in \{i,j\}$, we denote the sets
\[ T_{kl}:=\{ t= 0,\dots , A_k-B_l\ | \ A_k-t \geq x, y \geq -B_l-t\}. \]
Our goal is to show that (by \eqref{eq 0})
\begin{align}
    r_{\alpha\beta}(\phi_1)- r_{\alpha\beta}(\phi_2)= \#T_{jj}+\#T_{ii}- \#T_{ji}-\#T_{ij} \geq 0.\label{eq 2}
\end{align} 
It is clear from the definition that $ T_{ij}$ is a subset of $T_{jj}$. Thus we may assume $T_{ji}$ is non-empty. We first claim that if $T_{ij}$ is empty, then $\#T_{jj} \geq \# T_{ji}$, which implies \eqref{eq 2} in this situation.

Let $t_1$ (resp. $t_2$) be the maximal (resp. minimal) number in $T_{ji}$. Observe that $T_{ji} \cap [0,A_j-B_j] \subseteq T_{jj}$, so we may assume $t_1 > A_j-B_j$. Then we have
\begin{align*}
    x \leq A_j-t_1 < B_j.
\end{align*}
On the other hand, since $A_i\geq B_j+1>x$ and $0$ is not in $T_{ij}=\emptyset$, we must have $y<-B_j$. In summary, we have the following inequalities
\[ \begin{cases}
B_i< A_i \leq B_j+1 \leq A_j,\\
A_j-B_j< t_1 \leq A_j-B_i,\\ 
x\leq A_j-t_1,\\
 -B_i-t_2 \leq y <-B_j.
\end{cases} \]
Then for $-y-B_j\leq s \leq -y-B_j+t_1-t_2 $, one can check that
\begin{enumerate}
    \item [$\oldbullet$] $ 0<s \leq B_i-B_j+t_1 \leq A_j-B_j$,
    \item [$\oldbullet$]$x+s\leq  x-y-B_j+t_1-t_2\leq A_j-y- B_j-t_2 \leq A_j + B_i-B_j < A_j $,
    \item [$\oldbullet$] $y+s\geq -B_j$.
\end{enumerate}
Therefore, these $t_1-t_2+1$ integers are in $T_{jj}$. This shows that $ \#T_{jj} \geq t_1-t_2+1 = \#T_{ji}$ and completes the proof of the claim.

Next, we deal with the situation that both $T_{ij}, T_{ji}$ are non-empty. The computation \eqref{eq 1} shows that
\[ \#T_{kl}= \max\{ A_k+ B_l+C_l,0\} ,\]
where
\[ C_l:=-\max\{x,B_l\}-\max\{-y,B_l\}+1. \]
Then since $T_{ij}$ and $T_{ji}$ are both non-empty, we have
\begin{align*}
    r_{\alpha\beta}(\phi_1)-r_{\alpha\beta}(\phi_2)
    &= \max\{ A_i+ B_i+C_i,0\}+\max\{ A_j+ B_j+C_j,0\}\\
    &- (A_j+ B_i+C_i+A_i+ B_j+C_j)\\
    &\geq 0.
\end{align*}
This completes the verification of this case.

\textbf{Case 2:} $T=dual \circ ui_{j,i} \circ dual$.

If the $ui_{j,i}$ is not of type 3', then $dual \circ ui_{j,i} \circ dual= ui_{i,j}^{-1}$, which is covered in the previous case. Therefore, we assume that $ui_{j,i}$ is of type 3'. By Definition \ref{def operators on parameters} and Remark \ref{rmk operator}, we may write
\[\begin{cases}T(\psi)&= \psi' +\rho\otimes S_{A_i+B_j+1}\otimes S_{A_i-B_j+1},\\ \psi&= \psi'+ \rho\otimes S_{A_i+B_i+1}\otimes S_{A_i-B_i+1} +\rho\otimes S_{A_j+B_j+1}\otimes S_{A_j-B_j+1},\end{cases}\]
where $ A_i \geq -B_i=A_j+1 > -B_j$. By the same argument in the previous case, it remains to check that $r_{\alpha\beta}(\phi_1) \geq  r_{\alpha\beta}(\phi_2)$ for any $\alpha,\beta$, where
\begin{align*}
\phi_1=&  \bigoplus_{t=0}^{A_i-B_j} |\cdot|^{\frac{A_i-B_j}{2}-t}\otimes S_{A_i+B_j+1},\\
    \phi_2=&\bigoplus_{t=0}^{A_i-B_i} |\cdot|^{\frac{A_i-B_i}{2}-t}\otimes S_{A_i+B_i+1}  +\bigoplus_{t=0}^{A_j-B_j} |\cdot|^{\frac{A_j-B_j}{2}-t}\otimes S_{A_j+B_j+1} .
\end{align*}
We again consider
\[ T_{kl}:=\{ t= 0,\dots , A_k-B_l\ | \ A_k-t \geq x, y \geq -B_l-t\}, \]
for $k,l \in \{i,j\}$. Our goal is to show the inequality
\begin{align}
    \label{eq 3} r_{\alpha\beta}(\phi_1)-r_{\alpha\beta}(\phi_2)= \#T_{ij}- (\#T_{ii}+\#T_{jj})\geq 0.
\end{align}
We first claim that if $T_{jj}$ is empty, then $ \# T_{ij} \geq \#T_{ii}$, which implies \eqref{eq 3} in this situation.

Let $t_1$ (resp. $t_2$) be the maximal (resp. minimal) number in $T_{ii}$. Observe that $T_{ii} \cap [0, A_i-B_j] \subseteq T_{ij}$, so we may assume $ A_i-B_j< t_1 \leq A_i-B_i $. Then we have
\[ x \leq A_i-t_1 =B_j. \]
On the other hand, since $x \leq B_j\leq A_j$ and $0$ is not in $T_{jj}=\emptyset$, we must have $y<-B_j$. In summary, we have the following inequalities
\[ \begin{cases}
-B_j\leq  A_j = -B_j-1 <-B_i \leq A_i,\\
A_i-B_j< t_1 \leq A_i-B_i,\\ 
x\leq A_i-t_1<B_j,\\
 -B_i-t_2 \leq y<-B_j.
\end{cases} \]
Then for $  -y-B_j \leq s\leq -y-B_j+t_1-t_2  $, one can check that
\begin{enumerate}
    \item [$\oldbullet$] $0 <s \leq B_i-B_j+t_1\leq A_i-B_j $,
    \item [$\oldbullet$] $x+s \leq  x-y-B_j+t_1-t_2\leq A_i-y-B_j-t_2\leq A_i+B_i-B_j<A_i+1 \leq A_i$,
    \item [$\oldbullet$] $ y+s \geq -B_j$.
\end{enumerate}
Therefore, these $t_1-t_2+1$ integers are in $T_{ij}$. This shows that $ \#T_{ij} \geq t_1-t_2+1 = \#T_{jj}$ and completes the proof of the claim.

Next, we deal with the situation that both $T_{ii}$ and $T_{jj}$ are non-empty. The computation \eqref{eq 1} shows that
\[ \#T_{kl}= \max\{ A_k+ B_l+C_l,0\} ,\]
where 
\[ C_l:=-\max\{x,B_l\}-\max\{-y,B_l\}+1. \]
Since (recall that $x \geq y$)
\[ A_j+B_i+C_i\leq (A_j+B_i+1)+(-x+y) \leq 0, \]
and both $T_{ij}$ and $T_{ji}$ are non-empty, we have
\begin{align*}
     &\quad\,
     r_{\alpha\beta}(\phi_1)-r_{\alpha\beta}(\phi_2)\\
     &=(\max\{ A_i+ B_j+C_j,0\} +0)
     -(A_i+ B_i+C_i+A_j+ B_j+C_j)\\
     &\geq A_i+B_j+C_j + A_j+B_i+C_i
     -(A_i+ B_i+C_i+A_j+ B_j+C_j) \\
     &= 0.
\end{align*}
This completes the verification of this case.

\textbf{Case 3:} $T=dual_{k}^{-}$. In this case, by Definition \ref{def operators on parameters} (3), we may write

\[\begin{cases} 
T(\psi)&= \psi'+ \rho\otimes S_{a+1}\otimes S_{a},\\
\psi&= \psi'+ \rho\otimes S_{a}\otimes S_{a+1}.  \end{cases}\]
Following the same argument in the first case,  it remains to check that $r_{\alpha\beta}(\phi_1) \geq  r_{\alpha\beta}(\phi_2)$ for any $\alpha,\beta$ where
\begin{align*}
\phi_1=&  \bigoplus_{t=0}^{a-1} |\cdot|^{\frac{a-1}{2}-t}\otimes S_{a+1},\\
    \phi_2=&\bigoplus_{t=0}^{a} |\cdot|^{\frac{a}{2}-t}\otimes S_{a}.
\end{align*}
Then by \eqref{eq 0}, we have $r_{\alpha\beta}(\phi_i)= \#T_i$, where 
\begin{align*}
    T_1&=\{ t= 0,\dots , a-1 \ | \ a-\half{1}-t \geq x, y \geq -\half{1}-t\},\\
    T_2&=\{ t= 0,\dots , a \ | \ a-\half{1}-t \geq x, y \geq \half{1}-t\}.
\end{align*}
Clearly $T_2 \cap [0,a-1] \subseteq T_1$. It remains to check the case that $a \in T_2$. In this case, we have 
\[ \half{-1} \geq x, y \geq \half{1}-a. \]
Then $t=-\half{1}-y$ is in $T_1$ but not in $T_2$. Therefore, $r_{\alpha\beta}(\phi_1)=r_{\alpha\beta}(\phi_2)$ in this case. This finishes the verification of this case and completes the proof of the theorem.
\qed

We remark that Theorem \ref{thm three orderings} for $D$ ordering (except the uniqueness in Part (2)) also follows directly from the same Theorem for $C$ ordering and Part (3) of Corollary \ref{cor comparison closure partition}.

\section{Proof of Theorem \ref{thm max intro} Part (2)}\label{sec proof}

In this section, we prove Theorem \ref{thm max intro} Part (2) by exploring the relations between $\phi_\pi$ and $\psi^{max}(\pi)$. To this end, we prove a key property of $\psi^{max}(\pi)$  (Proposition \ref{prop max triangle} below) which implies that the local $L$-parameter of $\pi$ and $\phi_{\psi^{max}(\pi)}$ share certain common direct summands.

We first prove a lemma which describes a restriction on the $L$-data of $\pi(\EE)$ in terms of an invariant of $\psi_{\EE}$.

\begin{lemma}\label{lem max b}
Suppose $\EE \in \Rep$. Write $\psi_\EE= \bigoplus_{i} \rho_i'\otimes S_{a_i}\otimes S_{b_i}$, 
    \begin{align*}
        \pi(\EE)= L(\Delta_{\rho_1}[x_1,-y_1],\dots ,\Delta_{\rho_f}[x_f,-y_f]; \pi(\phi,\varepsilon)),
    \end{align*} 
    and let $b_{\rho}= \max\{ b_i \ | \ \rho_i' \cong \rho\}$. Then we have 
    \begin{align}\label{eq max b}
        -b_{\rho}+1 \leq  \min \left(\{ x_i-y_i \ | \ \rho_i \cong \rho\}  \cup \{0\}\right).
    \end{align}    
    In particular, for any positive integer $A,B$ such that $ B-A \leq -b_{\rho}+1$, the parabolic induction
    \[\sigma:= \Delta_{\rho}[B,-A]\rtimes \pi(\EE)\]
    has a unique irreducible subrepresentation $\sigma'$, whose $L$-data can be obtained by inserting $\Delta_{\rho}[B,-A]$ in the front of the $L$-data of $\pi(\EE)$. 
\end{lemma}
\begin{proof}
The second assertion follows from the inequality \eqref{eq max b} and the Langlands classification. Now we show \eqref{eq max b}. By \cite[Theorem 3.19(i)]{HLL22}, we may assume $\EE$ is positive and satisfies (P') by replacing $\EE$ by $sh^{t}(\EE)$ for a sufficiently large $t$ and do row exchanges if necessary. If $b_{\rho}=1$, then the set $\{x_i-y_i\ | \ \rho_i \cong \rho\}$ is empty, and hence the equality holds trivially. Therefore, we also assume $b_{\rho}>1$ in the rest of the proof.

Suppose
\[ \pi=  L(\Delta_{\rho_1}[x_1,-y_1],\dots ,\Delta_{\rho_f}[x_f,-y_f]; \pi(\phi,\varepsilon)).\]
We denote the multi-set
\[ \Delta_{\rho}(\pi):= \{ \Delta_{\rho_i}[x_i,-y_i] \ | \ 1 \leq i \leq  f, \rho_i \cong \rho\}. \]
Write $\EE=\cup_{\rho} \{([A_i,B_i]_{\rho},l_i,\eta_i)\}_{i \in (I_{\rho},>)}$ and identify $(I_{\rho},>)$ with $\{1, \dots, n\}$ where $1<\cdots <n$. To prove \eqref{eq max b}, we apply induction on $n$. The case $n=1$ follows from definition. Assume $n>1$ from now on, and denote $\EE_{k}= sh_n^{k}(\EE)$. We separate the proof into two cases: (1) $b_{\rho}=A_n-B_n+1$; (2) $b_{\rho}>A_n-B_n+1$. 

\textbf{Case (1):} $b_{\rho}=A_n-B_n+1$.

If $l_n=0$, we can split the $n$-th row by $ui^{-1}$ of type 3' (see \cite[Corollary 5.7]{HLL22}), then we reduce to the case that $b_{\rho}> A_n-B_n+1$ or $b_{\rho}=1$.  Thus, we proceed assuming $l_n\geq 1$. In this case, we claim that $\Delta_{\rho}(\pi(\EE_k))$ is of the form
\begin{align} \label{eq 4.3}
    \Delta_{\rho}(\pi(\EE_k))=\{ \Delta_{\rho}[B_n+k,-(A_n+k) ],\Delta_{\rho}[x_1,-y_1],\dots ,\Delta_{\rho}[x_g,-y_g]\},
\end{align}  
with $ -b_\rho+1= (B_n+k) -(A_n+k) \leq x_i-y_i$ for all $i=1,\dots ,g$. This implies \eqref{eq max b} in this case.

Note that by the definition of $\pi(\EE_k)$ and induction hypothesis for $|I_{\rho}|=n-1$, we already know that $\Delta_{\rho}(\pi(\EE_k))$ is of the form of \eqref{eq 4.3} when $k$ is sufficiently large. Therefore, it suffices to show that if $\pi(\EE_k)$ is of the form of \eqref{eq 4.3}, then so is $\pi(\EE_{k-1})$ for $k \geq 1$.

Recall that we have the relation
\begin{align}\label{eq relation}
    \pi(\EE_{k-1})= D_{\rho|\cdot|^{A_n+k}} \circ \cdots \circ D_{\rho|\cdot|^{B_n+k}}(\pi(\EE_k)),
\end{align}   
and each derivative is the highest derivative (see \cite[Lemma 4.2]{HLL22}). We first keep track of the change of the segment $\Delta_{\rho}[B_n+k, -(A_n+k)]$ under the derivatives $D_{\rho|\cdot|^{\alpha}}$ for $\alpha= B_n+k,\dots ,A_n+k$. The first derivative $D_{\rho|\cdot|^{B_n+k}}$ has no choice but to replace $\Delta_{\rho}[B_n+k, -(A_n+k)]$ by $\Delta_{\rho}[B_n+k-1, -(A_n+k)]$. Then applying Lemma \ref{lem highest derivative of order 1} for $\alpha= B_n+k+1,\dots ,A_n+k-1$, the derivatives $D_{\rho|\cdot|^{\alpha}}$ leave the segment $\Delta_{\rho}[B_n+k-1, -(A_n+k)]$ unchanged. Finally, for $\alpha= A_n+k$, notice that by our assumption that $\EE$ satisfies (P') and $A_n-B_n+1=b_{\rho}$, we know that the multiplicity of $\rho|\cdot|^{A_n+k}$ in $\Omega(\EE_{k-1})$ (see \cite[Definition 4.1]{HLL22}) is zero. Then in order that $\pi(\EE_{k-1})$ satisfies \cite[Lemma 4.7]{HLL22}, $D_{\rho|\cdot|^{A_n+k}}$ has to replace $\Delta_{\rho}[B_n+k-1, -(A_n+k)]$ by $\Delta_{\rho}[B_n+k-1, -(A_n+k-1)]$. This shows that $\Delta_{\rho}[B_n+k-1, -(A_n+k-1)]$ is indeed in $\Delta_{\rho}(\pi(\EE_{k-1}))$. 

Next, we need to show that for any other segment $\Delta_{\rho}[x,-y]$ in $\Delta_{\rho}(\pi(\EE_{k-1}))$, we have 
\[(B_n+k-1) -(A_n+k-1) \leq x-y.\]
Applying Lemma \ref{lem highest derivative of order 1} on each derivative in the relation \eqref{eq relation}, we see that $\Delta_{\rho}[x,-y]$ is of one of the following forms.
\begin{enumerate}
    \item [(a)] $\Delta_{\rho}[x_i-1,-y_i]$ where $\Delta_{\rho}[x_i,-y_i]$ is in $\Delta_\rho(\pi(\EE_k))$,
    \item [(b)]$\Delta_{\rho}[x_i,-y_i+1]$ where $\Delta_{\rho}[x_i,-y_i]$ is in $\Delta_\rho(\pi(\EE_k))$,
    \item [(c)] $\Delta_{\rho}[x_i-1,-y_i+1]$ where $\Delta_{\rho}[x_i,-y_i]$ is in $\Delta_\rho(\pi(\EE_k))$,
    \item[(d)] $\Delta_{\rho}[\alpha-1,-\alpha]$.
\end{enumerate}
Since $\Delta_{\rho}(\pi(\EE_{k}))$ is of the form \eqref{eq 4.3}, we have 
\[ B_n+k-1- (A_n+k-1)=B_n+k- (A_n+k) \leq x_i- y_i, \]
so case (b) and (c) are done. Also, we have assumed $b_{\rho}>1$, and hence \[-b_\rho+1=B_n+k-1- (A_n+k-1) \leq -1,\]
so case (d) is also done. It remains to argue that if 
$$ \Delta_{\rho}[x,-y]=\Delta_{\rho}[x_i-1,-y_i],$$
then  
\[B_n+k-1- (A_n+k-1) \leq x_i-1- y_i .\]
Suppose the contrary. Then we have 
\[ \begin{cases} x_i-y_i=B_n+k -(A_n+k)& \text{by } \eqref{eq 4.3} \text{ for }\pi(\EE_k),\\
B_n+k+1 \leq x_i \leq A_n+k-1 & \text{by Lemma \ref{lem highest derivative of order 1}}, \end{cases} \]
and hence $y_i \geq A_n +k +1$. Since the multiplicity of $\rho|\cdot|^{A_n+k+1}$ in $\Omega(\EE_{k-1})$ is zero, this contradicts to \cite[Lemma 4.7]{HLL22}, which completes the proof for the case that $b_{\rho}= A_n-B_n+1$.

\textbf{Case (2):} $b_{\rho} > A_n-B_n+1$. 

In this case, we claim that $\Delta_{\rho}(\pi(\EE_k))$ is of the form 
\begin{align}\label{eq 4.4}
    \Delta_{\rho}(\pi(\EE_k))=\{\Delta_{\rho}[x_1,-y_1],\dots ,\Delta_{\rho}[x_g,-y_g]\},
\end{align}  
where $ -b_\rho+1 \leq  x_i-y_i$ for all $i=1 ,\dots, g$. Again \eqref{eq 4.4} holds for $\pi(\EE_k)$ when $k$ is sufficiently large, and it suffices to show that if $\Delta_{\rho}(\pi(\EE_k))$ is of the form \eqref{eq 4.4}, then so is $\Delta_{\rho}(\pi(\EE_{k-1}))$ for $k \geq 1$.

Suppose $\Delta_{\rho}[x,-y] \in \Delta_{\rho}(\pi(\EE_{k-1}))$. Then similarly, it is of one of the following forms.
\begin{enumerate}
    \item [(a)] $\Delta_{\rho}[x_i-1,-y_i]$ where $\Delta_{\rho}[x_i,-y_i]$ is in $\Delta_\rho(\pi(\EE_k))$.
    \item [(b)]$\Delta_{\rho}[x_i,-y_i+1]$ where $\Delta_{\rho}[x_i,-y_i]$ is in $\Delta_\rho(\pi(\EE_k))$.
    \item [(c)] $\Delta_{\rho}[x_i-1,-y_i+1]$ where $\Delta_{\rho}[x_i,-y_i]$ is in $\Delta_\rho(\pi(\EE_k))$.
    \item[(d)] $\Delta_{\rho}[\alpha-1,-\alpha]$.
\end{enumerate}
Again we only need to deal with case (a). In other words, it remains to show that if $ \Delta_{\rho}[x,-y]=\Delta_{\rho}[x_i-1,-y_i]$, then  
\[-b_{\rho}+1 \leq x_i-1- y_i .\]
Suppose the contrary. Then we have
\[ \begin{cases} x_i-y_i=- b_{\rho}+1& \text{by } \eqref{eq 4.4} \text{ for }\pi(\EE_k),\\
B_n+k \leq x_i \leq A_n+k & \text{by Lemma \ref{lem highest derivative of order 1}}, \end{cases} \]
and hence 
\begin{align}\label{eq 4.1 case 2}
    y_i \geq B_n+k+ b_{\rho} -1>B_n+k+ (A_n-B_n+1) -1= A_n+k.
\end{align}
On the other hand, \cite[Lemma 4.7]{HLL22} for $\EE_{k-1}$ implies the existence of $j \in I_{\rho}$ such that $B_j \leq y_i \leq A_j $. Since we have assumed $\EE$ satisfies (P'), we have $B_j \leq B_n < B_n+k$, and hence 
\[ -b_{\rho}+1= x_i-y_i \geq (B_n+k)- A_j >B_j-A_j \geq -b_{\rho}+1, \]
a contradiction. This completes the proof of the lemma.
\end{proof}

The lemma below describes how to relate the $L$-data of $\pi(\EE)$ with a related representation $\pi(\EE^-),$ provided this representation does not vanish. Later, we see that the non-vanishing of $\pi(\EE^-)$ is guaranteed when $\EE$ is absolutely maximal (see Proposition \ref{prop max triangle}).

\begin{lemma}\label{lem max triangle}
Suppose $\EE= \cup_{\rho} \{([A_i,B_i]_{\rho},l_i,\eta_i)\}_{i \in (I_{\rho},>)} \in \Rep^{(P')}$ and fix $\rho$ such that $I_{\rho}\neq \emptyset$. Suppose $j \in I_{\rho}$ satisfies the following conditions.
\begin{enumerate}
    \item [$\oldbullet$]$j=\max\{i \in I_{\rho}\ |\ B_i=B_j  \}$.
    \item [$\oldbullet$] $\pi(\EE^{-}) \neq 0$ where $\EE^{-}:= add_j^{-1}(\EE)$.
    \item [$\oldbullet$] Write
\[ \pi(\EE^{-})=L(\Delta_{\rho_1}[x_1,-y_1],\dots ,\Delta_{\rho_f}[x_f,-y_f]; \pi(\phi,\varepsilon)). \]
    We have $B_j-A_j \leq x_i- y_i $ if $\rho_i \cong \rho$.
\end{enumerate}
 Then the $L$-data of $\pi(\EE)$ can be obtained from that of $\pi(\EE^{-})$ by inserting $\Delta_{\rho}[B_j,-A_j]$. 
\end{lemma}

\begin{proof}

Our goal is to show that
\[ \pi(\EE)=L(\Delta_{\rho}[B_j,-A_j], \Delta_{\rho_1}[x_1,-y_1],\dots ,\Delta_{\rho_f}[x_f,-y_f]; \pi(\phi,\varepsilon)),\]
which is equivalent to prove the injection
\begin{equation}\label{equ injection}
   \pi(\EE) \hookrightarrow \Delta_{\rho}[B_j, -A_j] \rtimes \pi(\EE^{-}). 
\end{equation}
Identify $(I_{\rho},>)$ with $\{1, \dots, n\}$ where $1<\cdots <n$. Take a sequence of positive integers $\{t_1,\dots,t_n\}$ such that
\[\begin{cases} B_i+t_i >0 & \text{ for all }i,\\
B_i+t_i> A_k+t_k & \text{ for all }i>k,\end{cases}\]
and let 
\begin{align*}
    \EE_k:=\left(\sum_{i=k+1}^{n} sh_i^{t_i} \right)(\EE),\ \     \EE_k^{-}:=\left(\sum_{i=k+1}^{n} sh_i^{t_i} \right)(\EE^{-}),
\end{align*}
and set $\EE_n=\EE$, $ \EE_n^{-}= \EE^{-}$. To show \eqref{equ injection}, we are going to prove 
\begin{align}
    \label{eq 4.1}\pi(\EE_k) \hookrightarrow \Delta_{\rho}[B_j, -A_j] \rtimes \pi(\EE^{-}_k),
\end{align} 
for $j \leq k \leq n$ by applying induction on $k$.

First, we consider the case that $k=j$. For $\alpha=0,\dots ,t_j$, denote $\EE_{j,\alpha}= sh_j^{-\alpha} (\EE_j)$ and $\EE_{j,\alpha}^{-}= sh_j^{-\alpha} (\EE_j^{-})$. We apply induction on $\alpha$ to show the injection  
\begin{align}
    \label{eq 4.2}\pi(\EE_{j,\alpha}) \hookrightarrow \Delta_{\rho}[B_j+t_j-\alpha, -(A_j+t_j-\alpha)] \rtimes \pi(\EE^{-}_{j,\alpha}).
\end{align}  
When $\alpha=0$, the injection follows from the definition of $\pi(\EE_{j})$ and $\pi(\EE_{j}^{-})$.

The representations $\pi(\EE_{j,\alpha})$ and $\pi(\EE_{j,\alpha-1})$ (resp. $\pi(\EE_{j,\alpha}^{-})$ and $\pi(\EE_{j,\alpha-1}^{-})$) are related by the following formulas
\begin{align*}
    \pi(\EE_{j,\alpha})&=\,D_{\rho|\cdot|^{A_j+(t_j-\alpha)+1}} \circ D^- \circ  D_{\rho|\cdot|^{B_j+(t_j-\alpha)+1}} (\pi(\EE_{j,\alpha-1})),\\
    \pi(\EE_{j,\alpha}^{-})&=D^- (\pi(\EE_{j,\alpha-1}^{-})),
\end{align*}
where $D^-:=D_{\rho|\cdot|^{A_j+(t_j-\alpha)}} \circ   \cdots \circ D_{\rho|\cdot|^{B_j+(t_j-\alpha)+2}}$. 
Therefore, we may apply Lemma \ref{lem Frobenius} as follows.
\begin{align*}
    \pi(\EE_{j,\alpha-1})& \hookrightarrow \Delta_{\rho}[B_j+t_j-\alpha+1,-(A_j+t_j- \alpha+1))] \rtimes \pi(\EE_{j,\alpha-1}^{-})\\
    &\hookrightarrow \rho|\cdot|^{B_j+t_j-\alpha+1} \times \Delta_{\rho}[B_j+t_j-\alpha,-(A_j+t_j- \alpha+1)] \\
    &\quad \times \rho|\cdot|^{B_j+t_j-\alpha+2}\times \cdots \times \rho|\cdot|^{A_j+t_j-\alpha} \rtimes \pi(\EE_{j,\alpha}^{-})\\
    &= \rho|\cdot|^{B_j+t_j-\alpha+1} \times \cdots \times \rho|\cdot|^{A_j+t_j-\alpha}\\ &\quad \times \Delta_{\rho}[B_j+t_j-\alpha,-(A_j+t_j- \alpha+1)]  \rtimes \pi(\EE_{j,\alpha}^{-}).
\end{align*}
Then Lemma \ref{lem max b} implies that the parabolic induction
\[ \sigma:= \Delta_{\rho}[B_j+t_j-\alpha,-(A_j+t_j- \alpha+1)]  \rtimes \pi(\EE_{j,\alpha}^{-})\]
has a unique irreducible subrepresentation $\sigma'$, whose $L$-data can be obtained by inserting $\Delta_{\rho}[B_j+t_j-\alpha,-(A_j+t_j- \alpha+1)]$ in the $L$-data of $\pi(\EE_{j,\alpha}^{-})$. Applying the converse direction of Lemma \ref{lem Frobenius}, we have 
\[ D_{\rho|\cdot|^{A_j+(t_j-\alpha)}} \circ   \cdots \circ D_{\rho|\cdot|^{B_j+(t_j-\alpha)+1}} (\pi(\EE_{j,\alpha-1}))\geq \sigma'.\]
Note that each derivative on the left hand side is highest (see \cite[Lemma 4.2(iii)]{HLL22}), and hence the left hand side is irreducible, so the inequality is indeed an equality. Finally, applying the algorithm for positive derivative $D_{\rho|\cdot|^{A_j+(t_j-\alpha)+1}} $ on the $L$-data of $\sigma'$, we get 
\[ \pi(\EE_{j,\alpha})= D_{\rho|\cdot|^{A_j+(t_j-\alpha)+1}}(\sigma') \hookrightarrow \Delta_{\rho}[B_j+t_j-\alpha,-(A_j+t_j- \alpha)]  \rtimes \pi(\EE_{j,\alpha}^{-}). \]
This completes the induction step for \eqref{eq 4.2}, and the proof of \eqref{eq 4.1} for $k=j$. 

Next, we proceed to prove \eqref{eq 4.1} for $k>j$. The representations $ \pi(\EE_{k})$ and  $\pi(\EE_{k-1})$ (resp. $ \pi(\EE_{k}^{-})$ and  $\pi(\EE_{k-1}^{-})$) are related by
\begin{align*}
    \pi(\EE_{k} )&= \circ_{s=0}^{t_k-1} D_{\rho|\cdot|^{ B_{k}+t_k-s,\dots, A_{k}+t_k-s}} (\pi(\EE_{k-1})),\\
    \pi(\EE_{k}^{-} )&= \circ_{s=0}^{t_k-1} D_{\rho|\cdot|^{ B_{k}+t_s-s,\dots, A_{k}+t_k-s}} (\pi(\EE_{k-1}^{-})).
\end{align*}  
Therefore, we may apply Lemma \ref{lem Frobenius} as follows.
\begin{align*}
    \pi(\EE_{k-1}) & \hookrightarrow \Delta_{\rho}[ B_j,-A_j] \rtimes \pi(\EE_{k-1}^{-})\\
    & \hookrightarrow \Delta_{\rho}[ B_j,-A_j] \bigtimes_{s=0}^{t_k-1} \left( \rho|\cdot|^{A_k+t_k-s} \times\cdots \times \rho|\cdot|^{B_k+t_k-s}  \right) \rtimes  \pi(\EE_{k}^{-})\\
    &=\bigtimes_{s=0}^{t_k-1} \left( \rho|\cdot|^{A_k+t_k-s} \times\cdots \times \rho|\cdot|^{B_k+t_k-s}  \right) \times \Delta_{\rho}[ B_j,-A_j]  \rtimes  \pi(\EE_{k}^{-}).
\end{align*}
Here the last equality follows from the fact that any $\rho|\cdot|^{x}$ in the product commutes with $\Delta_{\rho}[ B_j,-A_j]$ since $j= \max\{ i \in I_{\rho} \ | \ B_i=B_j \}$ (see \cite[Theorem 1.1]{Tad14} or \cite[Corollary 6.10]{LM16}). Again, Lemma \ref{lem max b} implies that the parabolic induction
\[ \sigma :=\Delta_{\rho}[ B_j,-A_j]  \rtimes  \pi(\EE_{k}^{-})\]
has a unique irreducible subrepresentation $\sigma'$. Applying Lemma \ref{lem Frobenius}, we get
\[ \pi(\EE_{k}) \geq \sigma' \hookrightarrow \Delta_{\rho}[ B_j,-A_j]  \rtimes  \pi(\EE_{k}^{-}), \]
where the inequality is indeed equality since $ \pi(\EE_k)$ is irreducible. This completes the induction step for \eqref{eq 4.1} and the proof of the lemma.
\end{proof}

The following proposition shows that the hypotheses of Lemma \ref{lem max triangle} hold when $\EE$ is absolutely maximal. Note that if $\EE$ is not absolutely maximal, then Proposition \ref{prop max triangle} may fail. See \cite[Example 10.13]{HLL22} as a counter example.  Similar results are obtained by M{\oe}glin in \cite[Lemma 5.3]{Moe09b}. In comparison with her result, we use Atobe's parametrization of local Arthur packets, and due to the definition of $\psi^{max}$, our statement is simpler and more explicit in terms of $L$-data.  

\begin{prop}\label{prop max triangle}
Suppose $\EE= \cup_{\rho} \{([A_i,B_i]_{\rho},l_i,\eta_i)\}_{i \in (I_{\rho},>)} \in \Rep^{(P')}$ is absolutely maximal. Suppose there is a $\rho$ such that $b_{\rho} := \max \{A_i-B_i+1\ | \ i \in I_{\rho}\} >1 $. Let $ j := \min\{ i \in I_{\rho}\ | \ A_i-B_i+1=b_{\rho} \}$ and assume $j= \max\{ i \in I_{\rho} \ | \ B_i=B_j \}$ by applying row exchanges if necessary. Let $\EE^{-}:=add_j^{-1}(\EE)$, which satisfies (P') by the assumptions. Then, $\pi( \EE^{-}) \neq 0$ and the $L$-data of $\pi(\EE)$ can be obtained from that of $\pi(\EE^{-})$ by inserting $\Delta_{\rho}[B_j,-A_j]$.
\end{prop}

\begin{proof}
If $\pi(\EE^{-})\neq 0$, then the first two conditions of Lemma \ref{lem max triangle} hold from assumptions. By applying Lemma \ref{lem max b} to $\EE^{-}$, the third condition also holds. Therefore, it remains to show $\pi(\EE^{-})\neq 0$. We follow the notation in \cite[\S 3.2]{HLL22}. By \cite[Theorem 3.29]{HLL22}, it is equivalent to show that $\pi(dual(\EE^{-}))$ is nonzero. Now we check Conditions (i) and (ii) in \cite[Theorem 3.19]{HLL22} for $dual(\EE^{-})$, which is equal to $ sh_j^{-1}(dual(\EE))$ by Lemma \ref{lem equalities of operators}(1).

The absolute maximality of $\EE$ implies that we can not split the $j$-th row of $\EE$ by $ui$ inverse of type 3' in the sense of \cite[Corollary 5.7]{HLL22}, which is equivalent to $l_j\geq 1$. Then the dual formula (\cite[Definition 3.27]{HLL22}) shows that $ sh_j^{-1}(dual(\EE))$ satisfies Condition (i) of \cite[Theorem 3.19]{HLL22}.

Now we check Condition (ii) of \cite[Theorem 3.19]{HLL22}, i.e., any adjacent pair $(\alpha, \beta ,\gg)$ of $sh_j^{-1}(dual(\EE_{\rho}))$ satisfies Condition (i) of \cite[Proposition 3.12]{HLL22}. We say the adjacent pair $(\alpha, \beta ,\gg)$ is good in this case for brevity. Since 
\[A_j-B_j+1= \max \{A_i-B_i+1\ | \ i \in I_{\rho}\},\]
we may apply row exchanges to lower the $j$-th row of $dual(\EE)$ to the bottom, and hence it suffices to check the case that $\alpha=j$.

If $(j,\beta, \gg)$ is also an adjacent pair of $dual(\EE)$, then this adjacent pair of $sh_j^{-1}(dual(\EE_{\rho}))$ is not good only if $ui_{\beta ,j}$ is applicable on $dual(\EE)$. This implies that $dual\circ ui_{\beta,j} \circ dual$ is applicable on $\EE$ and contradicts to the absolute maximality of $\EE$. 

If $(j,\beta, \gg)$ is not an adjacent pair of $dual(\EE)$, then we may assume $j>' \beta+1 >' \beta$ is adjacent, where $>'$ is the admissible order of $dual(\EE)$ on $I_{\rho}$, and 
\[A_{j}-1= A_{\beta+1}> A_\beta  ,\  -B_{j}> -B_{\beta+1}> -B_{\beta}. \]
The adjacent pair $(j, \beta+1, >')$ of $sh_j^{-1}(dual(\EE_{\rho}))$ is good since $ ui_{\beta+1, j}$ is not applicable on $dual(\EE)$ by the absolute maximality of $\EE$. The adjacent pair $(\beta, \beta+1, >')$ is good since $\pi(dual(\EE))\neq 0$. Then \cite[Proposition 3.12(ii)]{HLL22} implies that $(j,\beta, \gg)$ is also good. This completes the proof of the proposition.
\end{proof}

As a corollary, the inequality in Lemma \ref{lem max b} is indeed an equality if $\EE$ is absolutely maximal. We rephrase this as follows.

\begin{cor}\label{cor L-data equality}
Suppose $\pi \in \Pi^{gp}_A(G_n)$ with $L$-data
\[  \pi= L(\Delta_{\rho_1}[x_1,-y_1],\dots ,\Delta_{\rho_f}[x_f,-y_f]; \pi(\phi,\varepsilon)),\]
and
\[\psi^{max}(\pi)=\bigoplus_{\rho} \bigoplus_{i \in I_{\rho}} \rho \otimes S_{a_i} \otimes S_{b_i}. \]
For each $\rho \in \{\rho_1,\ldots, \rho_f\}$, let
\[b_{\rho}:= - \min \left(\{ x_i-y_i \ | \ 1 \leq i \leq f , \rho_i \cong \rho \} \cup \{0\} \right) +1.\]
We have the equality
\[ b_{\rho}= \max\{ b_i \ | \ i \in I_{\rho} \}. \]
\end{cor}

Remark that the equality in above corollary does not hold for arbitrary $\psi \in \Psi(\pi)$. The $\psi_{\EE}$ in \cite[Example 10.13]{HLL22} gives a counter example. Now we prove Theorem \ref{thm max intro}.

\begin{proof}[Proof of Theorem \ref{thm max intro}]
    By Corollary \ref{cor reduction to gp}, we may assume $\pi$ is of good parity. Write $\pi=\pi(\EE)$ where $\EE$ is absolutely maximal. Our goal is to show
\begin{align}\label{eq proof of conj}
    \phi_{\pi} \geq_C \phi_{\psi_{\EE}}.
\end{align}
We apply induction on $n$, the rank of $G_n$.

Write $\EE= \cup_{\rho} \{([A_i,B_i]_{\rho},l_i,\eta_i)\}_{i \in (I_{\rho},>)}$ and consider the quantities
\[ b_{\rho}(\EE):=\max_{i \in I_{\rho}} \{ A_i-B_i+1 \}. \]
If $b_{\rho}(\EE)\leq 2$ for every $\rho$, then $\EE$, which is absolutely maximal by assumption, automatically satisfies the conditions in \cite[Theorem 9.5]{HLL22}. Therefore, we have $\phi_\pi= \phi_{\psi_{\EE}}$, so \eqref{eq proof of conj} trivially holds. Note that when $n=1$, we must have $b_{\rho}(\EE)\leq 2$, and hence \eqref{eq proof of conj} holds.

Suppose $n>1$ and $b_{\rho}(\EE) >2$ for some $\rho$. We take $j \in I_{\rho}$ and construct $\EE^{-}=add_j^{-1}(\EE)$ as in Proposition \ref{prop max triangle}. Since $\pi(\EE^{-})\neq 0$ is a representation of $G_m$ for some $m<n$, the induction hypothesis and Theorem \ref{thm three orderings} for closure ordering imply that
\[ \phi_{\pi(\EE^{-})} \geq_{C} \phi_{\psi^{max}(\pi(\EE^{-}))} \geq_C \phi_{\psi_{\EE^{-}}}. \]
On the other hand, Proposition \ref{prop max triangle} gives that 
\begin{align*}
    \phi_{\pi(\EE)}&= \left(\rho|\cdot|^{\half{B_j-A_j}} \otimes S_{B_j+A_j+1}\oplus \rho|\cdot|^{-\half{B_j-A_j}} \otimes S_{B_j+A_j+1}\right)\oplus  \phi_{\pi(\EE^{-})},\\
    \phi_{\psi_\EE}&= \left(\rho|\cdot|^{\half{B_j-A_j}} \otimes S_{B_j+A_j+1}\oplus \rho|\cdot|^{-\half{B_j-A_j}} \otimes S_{B_j+A_j+1}\right)\oplus  \phi_{\psi_{\EE^{-}}}.
\end{align*}
Then applying Lemma \ref{lem common part}, we get $\phi_{\pi(\EE)} \geq_C \phi_{\psi_{\EE}}.$ This completes the proof of the theorem.
\end{proof}

\section{Enhanced Shahidi Conjecture}\label{sec Shahidi}

Let $\RG$ be a connected reductive group defined and quasi-split over $F$ and let $G= \RG(F)$. We assume that there is a theory of local Arthur packets for $G$ as conjectured in \cite[Conjecture 6.1]{Art89}. Namely, let $\Psi(G)$ be the collection of continuous homomorphism $\psi: W_F \times \SL_2^D(\BC) \times \SL_2^A(\BC) \to {}^{L}G$ such that 
\begin{enumerate}
    \item [(i)] the restriction of $\psi$ to $W_F \times \SL_2^D(\BC)$ gives a tempered $L$-parameter, and
    \item [(ii)] the restriction of $\psi$ to $\SL_2^{A}(\BC)$ is analytic.
\end{enumerate}
For each $\psi \in \Psi(G)$, there exists a finite multi-set $\Pi_{\psi} \subset \Pi(G)$ with nice properties. We call such $\psi$ a local Arthur parameter of $G$ and $\Pi_{\psi}$ the local Arthur packet of $\psi$. We say $\psi$ is \emph{tempered} if $\psi|_{\SL_2^{A}}(\BC)$ is trivial.

In this section, we show that Conjecture \ref{main conj}, along with certain assumptions (Working Hypotheses \ref{assumption}), implies the following form of the enhanced Shahidi conjecture (Conjecture \ref{Enhanced Shahidi conjecture intro}), see Theorem \ref{proof of enhanced shahidi} below.

\begin{conj}[{\cite[Conjecture 1.5]{LS22}}, Enhanced Shahidi Conjecture]\label{Enhanced Shahidi conjecture}
For any quasi-split reductive group $G$, a local Arthur parameter $\psi \in \Psi(G)$ is tempered if and only if $ \Pi_{\psi}$ has a generic member.
\end{conj}

 Then, we verify these assumptions for symplectic and split odd special orthogonal groups (Theorem \ref{dual packets} and Lemma \ref{dual generic}). Combining with Theorem \ref{thm max intro}, this gives a new proof of Conjecture \ref{Enhanced Shahidi conjecture} in these cases (Theorem \ref{new proof of enhanced Shahidi conjecture}). 

 It is worth to notice that without assuming generalized Ramanujan Conjecture, $\Psi(G)$ is only a subset of $\Psi^+(G)$, the set of all local Arthur parameters (see \cite[\S 1.3]{Art13} and \eqref{lap} for $G_n$). Thus, Conjecture \ref{Enhanced Shahidi conjecture}, which is stated for $\Psi(G)$, does not directly imply Conjecture \ref{Enhanced Shahidi conjecture intro}, which is stated for $\Psi^{+}(G)$. However, for symplectic and split odd special orthogonal groups, it is known that the two conjectures are equivalent (see the proof of \cite[Theorem 8.6]{HLL22}). We expect that the equivalence for general quasi-split groups can be verified case by case.

\subsection{On the enhanced Shahidi conjecture} \label{sec Shahidi generality}

The Shahidi conjecture (\cite[Conjecture 9.4]{Sha90}) states that if $\phi$ is tempered, then $\Pi_\phi$ contains a generic representation. This generic representation is expected to be unique once the Whittaker data being fixed. 
Notice that if $\phi$ is tempered, then $\phi$ is also a local Arthur parameter and the local $L$-packet is expected to coincide with the corresponding local Arthur packet. The Shahidi conjecture has been proved for many cases (for examples, see \cite{Kon02, JS03, JS04, Liu11, JS12, MW12, Wal12, Art13, JL14, KMSW14, Mok15, Beu16, Var17, Ato17, JL22}). The new content of the enhanced Shahidi conjecture is that if $\psi \in \Psi(G)$ and the local Arthur packet $\Pi_\psi$ contains a generic representation, then $\psi$ is a tempered local Arthur parameter.

 We first recall certain geometric structure on the Vogan varieties. Fix an infinitesimal character $\lambda: W_F\to {}^L \RG$ and let $V_\lambda$ be the Vogan variety. Under the closure ordering on the set of orbits in $V_{\lambda}$, the zero orbit is clearly the unique minimal orbit. On the other hand, it turns out that there is also a unique maximal orbit under this ordering.

\begin{lemma}[{\cite[Proposition 5.6]{CFMMX22}}]\label{max min V_lambda}
The Vogan variety $V_\lambda$ contains a unique open dense orbit.
\end{lemma}

Following \cite{CFMZ22}, a local $L$-parameter $\phi$ is called open if $C_\phi$ is open in $V_{\lambda_\phi}$. The following result describes the basic properties of open parameters.

\begin{prop}[{\cite{CFMZ22}}]\label{CFMZ} Let $\phi$ be a local $L$-parameter of $G$ and let $\lambda=\lambda_\phi$ be the infinitesimal character associated with $\phi$.
\begin{enumerate}
    \item The parameter $\phi$ is tempered if and only if $\phi$ is open and $\phi= \phi_{\psi}$  for some $\psi \in \Psi(G)$.
    \item The parameter $\phi$ is open if and only if $L(s,\phi, \Ad)$ is regular at $s=1$.
\end{enumerate}
\end{prop}
 
 In the following, we don't distinguish a local $L$-parameter in $\Phi(G)_{\lambda}$ and the orbit it represents in $V_{\lambda}$. We denote $\phi_0$ (resp. $\phi^{0}$) the unique closed (resp. open) orbit in $V_{\lambda}$. We shall focus on the case that $V_{\lambda}$ contains a local $L$-parameter of Arthur type. In this case, $\phi^{0}$ and $\phi_0$ are related as follows.

 \begin{lemma}\label{phi_0 phi^0 Arthur type}
If $V_\lambda$ contains a local $L$-parameter $\phi_{\psi}$ for some $\psi \in \Psi(G)$, then $\phi^0 =\phi_{\psi^0}$ and $\phi_0=\phi_{\psi_0}$ for some local Arthur parameters $\psi^0, \psi_0 \in \Psi(G)$. Moreover, we have $\widehat{\psi^0}=\psi_0.$
\end{lemma}

\begin{proof} 
Take $\psi$ a local Arthur parameter such that $\lambda= \lambda_{\phi_{\psi}}$. Then we consider the local $L$-parameter $\psi^{\Delta}$ defined by
\[ \psi^{\Delta}(w,x):= \psi(w,x,x),\]
which is tempered since $\psi|_{W_F}$ has bounded image. It follows from direct computation that $\lambda_{\phi^0}= \lambda_{\phi_{\psi}}=\lambda$, and hence $\psi^{\Delta}=\phi^{0}$ by Proposition \ref{CFMZ}(1). Therefore, $\phi^0= \psi_{\psi^{0}}$, where
\[ \psi^{0}(w,x,y):= \psi^{\Delta}(w,x)= \psi(w,x,x). \]

Next, we show that $\phi_{\widehat{\psi}^0}$ gives the minimal element in $V_{\lambda}$. Indeed, since $\widehat{\psi}^{0}$ is trivial on $\SL_2^D(\BC)$, $\phi_{\widehat{\psi}^0}$ is trivial on $\SL_2(\BC)$. Therefore, $C_{\phi_{\widehat{\psi}^0}}$ is the zero orbit, which is minimal by definition. This completes the proof of the lemma.
\end{proof}

 Conjecture \ref{main conj} has a direct corollary as follows. 

\begin{cor}\label{phi_0 singleton}\label{a dual version of Shahidi}
Assume that Conjecture \ref{main conj} holds for $G$ and that $\pi\in\Pi_{\phi_0}.$ If $\phi_0=\phi_{\psi_0}$ for some $\psi_0 \in \Psi(G)$, then $\Psi(\pi)=\{\psi_0\}.$
\end{cor}

\begin{proof}
By Conjecture \ref{main conj}, if $\pi\in\Pi_\psi$ for some local Arthur parameter $\psi,$ then $\phi_0\geq_C\phi_\psi.$ But $\phi_0$ corresponds to the zero orbit in $V_\lambda$ and hence this can only occur if $\psi=\psi_0.$
\end{proof}

\begin{remark}
For the ABV-packets defined in \cite{CFMMX22}, the analogue of Corollary \ref{phi_0 singleton} is known for general $p$-adic connected reductive groups since the analogue of  Conjecture \ref{main conj} has been proved there for these packets using geometric tools. See \cite{CFZ21,CFZ22} for the analogous results for unipotent representations of the exceptional group $G_2$.
\end{remark}

 Suppose $\pi$ is an unramified representation of $G$. It is known that the local $L$-parameter $\phi_{\pi}$ for $\pi$ is closed, i.e., $\phi_{\pi}$ is the unique closed orbit in $V_{\lambda_{\phi_{\pi}}}$. If we further assume $\pi$ is of Arthur type, then Lemma \ref{phi_0 phi^0 Arthur type} implies that $\phi_{\pi}=\phi_0= \phi_{\psi_{0}}$, and $\widehat{\psi_0} =\psi^{0}$ is tempered. Moreover, $\Psi(\pi)=\{\psi_{0}\}$ is a singleton by Corollary \ref{a dual version of Shahidi} assuming Conjecture \ref{main conj} holds for $G$. We record this special case as follows.
\begin{cor}\label{cor enhanced Shahidi for unramified}
Assuming Conjecture \ref{main conj} holds for $G$, given any local Arthur parameter $\psi \in \Psi(G)$, the local Arthur packet $\Pi_{\psi}$ contains an unramified representation if and only if $\widehat{\psi}$ is tempered and $\psi|_{I_F}$ is trivial, where $I_F$ is the inertia group.
\end{cor}

Remark that the above corollary is proved for $\Sp_{2n}(F)$ and split $\SO_{2n+1}(F)$ in \cite[\S 12]{HLL22} without assuming Conjecture \ref{main conj}. Now we state the Working Hypotheses.

\begin{assumption}\label{assumption} 
Let $\RG$ be a quasi-split connected reductive group over $F$. Suppose that there is a theory of local Arthur packets for $G$ as conjectured in \cite[Conjecture 6.1]{Art89}. We assume that the following hold. 
\begin{enumerate}
    \item Given a local Arthur parameter $\psi$ of $G$, we have $$\Pi_{\widehat\psi}=\{\widehat\pi\ |~\pi\in \Pi_\psi\}.$$
    \item Let $\pi$ be an irreducible generic representation of $G$ of Arthur type, then $\widehat\pi\in \Pi_{\phi_0}$, where $\phi_0$ is the local $L$-parameter which corresponds to the zero orbit in $V_\lambda$ with $\lambda=\lambda_{\phi_\pi}$.
\end{enumerate}
\end{assumption}

In the following, we show that the new content  of the enhanced Shahidi conjecture \ref{Enhanced Shahidi conjecture} is a dual version of Corollary \ref{a dual version of Shahidi} under  Working Hypotheses \ref{assumption}, and thus is a consequence of Conjecture \ref{main conj}.

\begin{thm}\label{proof of enhanced shahidi}
Let $\RG$ be a quasi-split connected reductive group over $F$. Suppose that there is a theory of local Arthur packets for $G$ as conjectured in \cite[Conjecture 6.1]{Art89}. 
 Assume Conjecture \ref{main conj} and the Working Hypotheses  \ref{assumption} hold for $G$. If $\psi \in \Psi(G)$ and $\Pi_\psi$  contains a generic representation, then $\psi$ is tempered.
\end{thm}

\begin{proof}
By Working Hypotheses \ref{assumption}(1),  $\pi\in \Pi_\psi$ if and only if $\widehat\pi\in \Pi_{\widehat \psi}$, which amounts to a bijection $$\Psi(\pi)\to \Psi(\widehat \pi),$$
$$\psi\mapsto \widehat\psi.$$
Suppose that $\pi$ is generic and $\pi\in \Pi_\psi$ for a local Arthur packet  $\psi$ of $G$. Let $\lambda$ be the infinitesimal character associated with $\phi_\psi$. Let $\psi_0$ (resp. $\psi^0$) be the local Arthur parameter of $\phi_0$ (resp. $\phi^0$) as in Lemma \ref{phi_0 phi^0 Arthur type}.  By Working Hypotheses \ref{assumption}(2), we have $\widehat \pi\in \Pi_{\phi_0}$. Conjecture \ref{main conj} implies that $\Psi(\widehat\pi)=\{\psi_0\}$, see Corollary \ref{a dual version of Shahidi}, and hence $\Psi(\pi)=\{\widehat{\psi_0}\}=\{\psi^0\}$. Thus $\psi=\psi^0$, which is tempered by Proposition \ref{CFMZ}. This completes the proof of the theorem.
\end{proof}

\subsection{A new proof of the enhanced Shahidi conjecture for \texorpdfstring{$\Sp_{2n}(F)$}{} and split \texorpdfstring{$\SO_{2n+1}(F)$}{}}\label{sec Shahidi for SO and Sp}

In this subsection, we discuss certain known cases regarding Working Hypotheses \ref{assumption} and give a new proof of Conjecture \ref{Enhanced Shahidi conjecture} for \texorpdfstring{$\Sp_{2n}(F)$}{} and split \texorpdfstring{$\SO_{2n+1}(F)$}{}. 

Working Hypotheses \ref{assumption}(1) essentially says that the Aubert-Zelevinsky involution should be compatible with endoscopic transfer, and it is known to be true for many groups. For example, it is true for quasi-split classical groups as we recall in Theorem \ref{dual packets} in the next section.

Regarding Working Hypotheses \ref{assumption}(2), inspired by \cite{CFMZ22}, we expect that it is true without the assumption that $\pi$ is of Arthur type. 
\begin{conj}{\label{dual generic conj}}
Let  $\pi$ be a generic representation of $G$. Then $\widehat\pi\in \Pi_{\phi_0}$, where $\phi_0$ is the local $L$-parameter corresponding to the zero orbit in $V_\lambda$ with $\lambda=\lambda_{\phi_\pi}$.
\end{conj}

Note that the above conjecture implies the following.

\begin{conj}[\cite{CFMZ22}]\label{generic conj}
A local $L$-parameter $\phi$ of $G$ is generic if and only if $\phi$ is open.
\end{conj}

By Proposition \ref{CFMZ}(2), the above Conjecture \ref{generic conj} is  equivalent to the following conjecture of Gross-Prasad and Rallis. 

\begin{conj}[{\cite[Conjecture 2.6]{GP}}]\label{GP}
A local $L$-parameter $\phi$ of $G$ is generic if and only if $L(s,\phi,\Ad)$ is regular at $s=1$.
\end{conj}

 Conjecture \ref{GP} was proved in \cite[Theorem C, (3)]{JS04} for $\SO_{2n+1}(F)$, in \cite[Theorem 1.2]{Liu11} for $\Sp_{2n}(F)$, and in \cite[Theorem 1.5]{JL14} for $\SO_{2n}(F)$. In a more general setting, Conjecture \ref{GP} was proved by Gan-Ichino \cite[Proposition B.1]{GI} under some hypothesis \cite[Sect. B.2]{GI}, which are known to be true for general linear groups and classical groups. Thus Conjecture \ref{generic conj} and Conjecture \ref{GP} hold for general linear groups and classical groups.

In the geometric setting \cite{CFMMX22}, the local Langlands correspondence gives a bijection between irreducible smooth representations of $G$ with infinitesimal character $\lambda$ and simple objects of the category of $H_\lambda$-equivariant perverse sheaves on $V_\lambda$, which is denoted by $ \mathsf{Per}_{H_\lambda}(V_\lambda)$. For an irreducible representation $\pi$, we denote by $\mathcal{P}(\pi)$ the corresponding perverse sheave.   A Fourier transform operator $\mathsf{F\hskip-1pt t}:\mathsf{Per}_{H_\lambda}(V_\lambda)\to \mathsf{Per}_{H_\lambda}(V_\lambda)$ is defined in \cite{CFMMX22} and it is expected that $\mathsf{F\hskip-1pt t}(\mathcal{P}(\pi))=\mathcal{P}(\widehat \pi) $. Moreover, if $\pi$ is an irreducible representation corresponding to $(\phi^0, 1)$ under the local Langlands correspondence, where $1$ is the trivial representation of $\mathcal{S}_{\phi^0}$, then it is known that $ \mathsf{F\hskip-1pt t}(\mathcal{P}(\pi))$ corresponds to a representation with local $L$-parameter $\phi_0$.  Thus, the geometric analogue of Conjecture \ref{dual generic conj} follows from Conjecture \ref{generic conj} or Conjecture \ref{GP}.

In the following, we show that Working Hypotheses \ref{assumption}(2) holds for quasi-split symplectic or special orthogonal groups.

\begin{lemma}\label{dual generic}
Let $\RG$ be a quasi-split symplectic or special orthogonal group. Let $\pi$ be an irreducible representation of $G$. If $\pi$ is generic and $\pi \in \Pi_{\psi}$ for some $\psi \in \Psi(G)$, then $\widehat{\pi}\in\Pi_{\phi_0}.$
\end{lemma}

\begin{proof}  
 By Proposition \ref{CFMZ}(2) and Conjecture \ref{GP}, which were known in these cases by \cite{Liu11}, \cite{JL14} and \cite{JS04}, or \cite{GI}, the local $L$-parameter $\phi_\pi$ of $\pi$ must be open in $V_\lambda.$ Consequently, $\phi_\pi=\phi^0$ by Lemma \ref{max min V_lambda}. By Lemma \ref{phi_0 phi^0 Arthur type}, $\phi^0= \phi_{\psi^0}$, where $\psi^0$ must be tempered by Proposition \ref{CFMZ}(1). We then have $\pi\in\Pi_{\psi^0}=\Pi_{\phi^0}$. Moreover, the character $\langle \cdot, \pi \rangle_{\psi^{0}}$ in $\widehat{\mathcal{S}}_{\psi^{0}}$ is trivial since $\pi$ is generic (see \cite[Theorem 8.1]{HLL22}).

Denote $\psi_0:= \widehat{\psi^0}$. Note that $\phi_0= \phi_{\psi_0}$ (see the proof of Lemma \ref{phi_0 phi^0 Arthur type}). The character $\langle \cdot, \widehat{\pi} \rangle_{\psi_0}$ in $\widehat{\mathcal{S}}_{\psi_0}$ is equal to $\varepsilon_{\psi_0,\zeta}^{M/MW}$, which lies in $\widehat{\mathcal{S}}_{\phi_{\psi_0}} \subseteq \widehat{\mathcal{S}}_{\psi_0}$ (see \cite{LLS24}). Thus $\widehat{\pi} \in \Pi_{\phi_0}$ by \cite[Lemma 7.4.1]{Art13}. This completes the proof of the lemma.
\end{proof}

Note that a key ingredient in the above proof is Conjecture \ref{GP}. We expect that Conjecture \ref{GP} implies Working Hypotheses \ref{assumption}(2) in general.

Since we have verified the Working Hypotheses \ref{assumption} for $\Sp_{2n}(F)$ and split $\SO_{2n+1}(F)$, we can give a new proof of the enhanced Shahidi conjecture \ref{Enhanced Shahidi conjecture} in these cases.

\begin{thm}\label{new proof of enhanced Shahidi conjecture}
The enhanced Shahidi conjecture \ref{Enhanced Shahidi conjecture} holds for $\Sp_{2n}(F)$ and split $\SO_{2n+1}(F)$.
\end{thm}
\begin{proof}
If $\psi$ is tempered, then it is well-known that $\Pi_\psi$ contains a generic representation. Thus it suffices to show that if $\psi \in \Psi(G)$ and $\Pi_\psi$ contains a generic representation, then $\psi$ must be tempered. This follows from Theorem \ref{proof of enhanced shahidi} directly. Notice that in this case, Conjecture \ref{main conj} is Theorem \ref{thm max intro}, Working Hypotheses \ref{assumption}(1) is verified in \cite[\S A]{Xu17} (see Theorem \ref{dual packets}) and Working Hypotheses \ref{assumption}(2) is Lemma \ref{dual generic}.
\end{proof}

The enhanced Shahidi conjecture \ref{Enhanced Shahidi conjecture} for $\Sp_{2n}(F)$ and split $\SO_{2n+1}(F)$ has been proved in \cite[Theorem 8.6]{HLL22}, whose proof relies on the construction of representations of Arthur type, while the new proof above relies on the geometry of $L$-parameters. More precisely, 
in \cite[Theorem 8.6]{HLL22}, the main tool to show that a generic representation of Arthur type is tempered is the classification of irreducible generic representations in terms of their Langlands classification. Then the uniqueness of a local Arthur packet that contains a given generic tempered representation was proved by showing that none of the operators (see \S\ref{sec Operators on extended multi-segments}) are applicable.
In contrast, in Theorem \ref{new proof of enhanced Shahidi conjecture}, to show that a generic representation of Arthur type and of good parity is tempered, we use the geometric description of a generic parameter. We obtain the uniqueness by taking the Aubert-Zelevinsky involution of the generic representation, which lies in the $L$-packet of an $L$-parameter of Arthur type corresponding to the zero orbit, and then the uniqueness follows from the closure ordering (see the proof of Theorem \ref{proof of enhanced shahidi}).

\begin{remark}
Proposition \ref{CFMZ} implies that open parameters (or generic parameters assuming Conjecture \ref{GP}) are generalizations of tempered parameters. Since ABV-packets defined in \cite{CFMMX22} are expected to be generalizations of the notion of Arthur packets, the ABV-version of Conjecture \ref{Enhanced Shahidi conjecture} is that $\Pi_{\phi}^{\mathrm{ABV}}$ contains a generic representation if and only if $\phi$ is open, see \cite{CFMZ22}. In viewing of the fact that $\Pi_\phi\subset \Pi_{\phi}^{\mathrm{ABV}}$ and Proposition \ref{CFMZ}, it is also a generalization of Conjecture \ref{GP}. The main result of \cite{CFMZ22} is that the above ABV-version conjecture  holds for quasi-split classical groups, which inspires our new proof of the enhanced Shahidi conjecture for $\Sp_{2n}$ and split $\SO_{2n+1}$.
\end{remark}

\section{On non-containment of local Arthur packets}\label{sec noncontainment}

In this section, we show that for $G_n=\Sp_{2n}(F)$ or split $\SO_{2n+1}(F)$, local Arthur packets can not be fully contained in other ones and prove Theorem \ref{thm three orderings} for $D$ ordering.

First, we recall the following theorem which states that the Aubert-Zelevinsky involution is compatible with endoscopic transfer for quasi-split classical groups.  

\begin{thm}[{\cite[\S A]{Xu17}}]\label{dual packets}
Let $\RG=\Sp_{2n}$ or quasi-split $\SO_n$ and let $\psi \in \Psi^+(G)$. Then 
$$
\Pi_{\widehat{\psi}}=\{\widehat{\pi} \, | \, \pi\in\Pi_\psi\}.
$$ 
\end{thm}

Therefore, for any $\pi \in \Pi_{A}(G_n)$, we have a bijective map
\begin{align}
    \Psi(\pi) &\to \Psi(\widehat{\pi}) \label{eq dual of psi}.\\
    \psi & \mapsto \widehat{\psi}\nonumber
\end{align}
The following lemma states that the map \eqref{eq dual of psi} reverses the ordering $\geq_{O}$ on these sets.

\begin{lemma}\label{lem order reversing}
If $T$ is a raising operator applicable on $\psi \in \Psi^+(G_n)$, then there exists another raising operator $T'$ applicable on $ \widehat{T(\psi)}$ such that
\[ \widehat{\psi}= T'(\widehat{T(\psi)} ). \]
In particular, for any $\pi \in \Pi_{A}(G_n)$, the map \eqref{eq dual of psi} is order reversing with respect to $\geq_{O}$.
\end{lemma}
\begin{proof}
For the first part, it follows from definition that if $T= ui_{i,j}^{-1}$ (resp, $dual \circ ui_{i,j} \circ dual$, $dual_k^{-}$), then $ T'=dual \circ ui_{i,j} \circ dual$ (resp. $ui_{i,j}^{-1}$, $dual_k^{-}$). For the second part, if $\psi_1 \geq_O \psi_2$, then we may write
\[ \psi_1 = T_1 \circ \cdots \circ T_m (\psi_2), \]
where $T_l$'s are raising operators. We apply induction on $m$ to show that $\widehat{\psi}_2 \geq_O \widehat{\psi}_1 $. The case $m=0$ is clear. For $m \geq 1$, let $\psi_3:=T_2 \circ \cdots \circ T_m (\psi_2)$ so that $\psi_1= T_1(\psi_3)$. The first part gives a raising operator $T'$ such that
\[ \widehat{\psi}_3 = T' (\widehat{\psi}_1), \]
and hence $\widehat{\psi}_3 \geq_O \widehat{\psi}_1$. The induction hypothesis then gives $\widehat{\psi}_2 \geq_O \widehat{\psi}_3 \geq_O \widehat{\psi}_1$. This completes the proof of the lemma.
\end{proof}

Here is the main theorem of this section.

\begin{thm}\label{noncontainment}
Given any $\psi_1,\psi_2 \in \Psi^+(G_n)$. 
If $\Pi_{\psi_1} \supseteq \Pi_{\psi_2}$, then $\psi_1=\psi_2$.
\end{thm}
% \begin{proof}
% Take $\pi \in \Pi_{\phi_{\psi_2}}$, then Proposition \ref{properties of max and min}(1) shows that $\psi_2=\psi^{max}(\pi)$. Therefore, Theorem \ref{thm max min intro} implies that \begin{align}\label{eq containment 1}
%     \psi_{2} \geq_O \psi_1
% \end{align}
% since $\psi_1 \in \Psi(\pi)$.

% On the other hand, recall that for any $\psi \in \Psi(G_n)$, we have 
% \[ \Pi_{\widehat{\psi}}= \{ \widehat{\pi} \ | \ \pi \in \Pi_{\psi} \}. \]
% Therefore, we also have that $\Pi_{\widehat{\psi}_1} \supseteq \Pi_{\widehat{\psi_2}}$. Then the same argument gives that 
% \[\widehat{\psi}_2 \geq_O \widehat{\psi}_1.\] Applying Lemma \ref{lem order reversing}, we then have 
% \begin{align}\label{eq containment 2}
%     \psi_1 \geq_O \psi_2.
% \end{align}  
% Finally, we conclude from \eqref{eq containment 1} and \eqref{eq containment 2} that $\psi_1=\psi_2$ as $\geq_O$ is a partial order. This completes the proof of the theorem.
% \end{proof}

Notice that $\Pi_{\psi_1} \supseteq \Pi_{\psi_2}$ implies that 
$\Pi_{\widehat{\psi}_1} \supseteq \Pi_{\widehat{\psi}_2}$.
Hence, Theorem \ref{noncontainment} is directly implied by the following more general statement. 

\begin{thm}\label{noncontainment general}
Given any $\psi_1,\psi_2 \in \Psi^+(G_n)$. If $\Pi_{\phi_{\psi_2}}\cap \Pi_{\psi_1}\neq\emptyset$ and $\Pi_{\phi_{\widehat{\psi}_2}}\cap \Pi_{\widehat{\psi}_1}\neq\emptyset$, then $\psi_1=\psi_2$.
\end{thm}
\begin{proof}
Take $\pi \in \Pi_{\phi_{\psi_2}}\cap \Pi_{\psi_1}$. Proposition \ref{properties of max and min}(1) and Theorem \ref{thm max min intro} implies that
\[ \psi_2=\psi^{max}(\pi) \geq_{O} \psi_1. \]
Similarly, take $\pi'\in \Pi_{\phi_{\widehat{\psi}_2}}\cap \Pi_{\widehat{\psi}_1}$, then we obtain
\[ \widehat{\psi}_2= \psi^{max}(\pi') \geq_{O} \widehat{\psi_1}, \]
which implies $\psi_1 \geq_O \psi_2 $ by Lemma \ref{lem order reversing}. Since $\geq_O$ is a partial order on $\Psi(\pi)$, we conclude that $\psi_1=\psi_2$, which completes the proof of the theorem.
\end{proof}

\begin{remark}
\begin{enumerate}
    \item On the contrary, for classical groups over complex fields, Barbasch and Trapa, M{\oe}glin and Renard (see \cite[Introduction and Theorem 12.5]{MR17}) showed that local Arthur packets can be fully contained in other ones. 
    \item Let $\RG$ be a classical group quasi-split over $F$ and $\RG_1$ be a non-quasi-split pure inner form of $\RG$. Let $G=\RG(F)$ and $G_1=\RG_1(F)$. For the conjectural local Arthur packets $\Pi_{\psi}(G_1)$ of $G_1$ (see \cite[Conjecture 9.4.2]{Art13}), it is possible that $\Pi_{\psi_1}(G_1) \supseteq \Pi_{\psi_2}(G_1)$ but $\psi_1 \neq \psi_2$. See \cite[\S 16.1.4]{CFMMX22} for examples. One can observe in these examples that one of $\psi_1$ or $\psi_2$ is $G_1$-relevant, but the associated $L$-parameter is not. Thus one of the $L$-packets $\Pi_{\phi_{\psi_1}}(G_1)$ or $\Pi_{\phi_{\psi_2}}(G_1)$ is empty, in which case Theorem \ref{noncontainment general} does not imply Theorem \ref{noncontainment}.
\end{enumerate}

\end{remark}

By Theorem \ref{noncontainment}, 
there are no proper subsets of local Arthur packets providing stable distributions which are compatible with endoscopic liftings (for the compatibility properties, see \cite[Theorem 2.2.1(a), (2.2.3) and (2.2.4)]{Art13}). 
% Over each whole local Arthur packet, it is believed that there is a unique (up to scalars) stable distribution which is compatible with endoscopic liftings. After taking the action of the center of $G_n$ on this unique stable distribution, one would obtain that representations in any local Arthur packet have the same central character. 
% We record this as the following theorem. 

% \begin{thm}\label{centralcharacter}
% Given any $\psi \in \Psi(G_n)$. All representations in ${\Pi}_{\psi}$ have the same central character.
% \end{thm}

% The contrapositive statement of Theorem \ref{noncontainment general} is the following.

% \begin{thm}
% Suppose $\psi_1\neq\psi_2$ are local Arthur parameters. Then at least one of $\Pi_{\phi_{\psi_2}}\cap \Pi_{\psi_1}$ or $\Pi_{\phi_{\widehat{\psi}_2}}\cap \Pi_{\widehat{\psi}_1}$ is empty.
% \end{thm}

Finally, we prove Theorem \ref{thm three orderings} for $D$ ordering using the order reversing property of $\geq_O$. 

\begin{proof}[Proof of Theorem \ref{thm three orderings} for $D$ ordering]
 Recall that $\underline{p}^D(\psi)= \underline{p}^A(\widehat{\psi})$, and hence $\psi_1 \gneq_{D} \psi_2$ if and only if $\widehat{\psi}_{1} \lneq_{A} \widehat{\psi}_2$. Thus Part (1) follows from that of $A$ ordering and Lemma \ref{lem order reversing}. Part (2) follows from Part (1) and Theorem \ref{thm max min intro}. This completes the proof of the theorem. 
\end{proof}

\end{document}